\documentclass[12pt]{amsart}
\usepackage{amssymb, eucal, amsfonts, amsmath, xypic,latexsym}

\def\gr{\mathrm{gr}}
\def\k{{\Bbbk}}

\def\g{{\mathfrak g}}

\def\h{{\mathfrak h}}
\def\a{{\mathfrak a}}
\def\sl{{\mathfrak sl}}

\def\u{{\mathfrak u}}
\def\N{{\mathcal N}}

\def\V{{\mathcal V}}
\def\SS{{\mathbb S}}
\def\CC{{\mathbb C}}

\def\E{{\mathcal E}}
\def\z{{\mathfrak z}}
\def\m{{\mathfrak m}}
\def\l{{\mathfrak l}}
\def\n{{\mathfrak n}}
\def\p{{\mathfrak p}}
\def\P{{\mathfrak P}}
\def\Z{{\mathbb Z}}
\def\N{{\mathbb N}}
\def\O{{\mathcal O}}
\def\End{\mathop{\fam0 End}}

\def\Lie{\mathop{\fam0 Lie}}
\def\Ad{\mathrm{Ad}\,}
\def\ker{\mathrm{Ker}\,}

\def\Mat{\mathop{\fam0 Mat}\nolimits}
\def\sl{\mathop{\mathfrak{sl}}\nolimits}

\def\ad{\mathrm{ad\,}}

\newcommand{\map}{\longrightarrow}

\def\la{\langle}
\def\ra{\rangle}

\theoremstyle{plain}
\newtheorem{theorem}{Theorem}[section]
\newtheorem{corollary}{Corollary}[section]
\newtheorem{prop}{Proposition}[section]
\newtheorem{lemma}{Lemma}[section]

\theoremstyle{definition}
\newtheorem{defn}{Definition}[section]

\newtheorem{question}{Question}[section]

\theoremstyle{remark}

\newtheorem{rem}{Remark}[section]

\def\subtitle#1. {{\medskip\bf#1\par\nobreak\smallskip}}
\def\proclaim#1. {\medbreak\bgroup\noindent\bf#1. \it}

\def\endproclaim{\egroup
\ifdim\lastskip<\medskipamount\removelastskip\medskip\fi}
\newcount
=0
\def\citedef#1 {\advance\citation by1
  \expandafter\edef\csname#1\endcsname{{\the\citation}}
  \checkendcitedef}
\def\checkendcitedef#1{\ifx#1\endcitedef\else\citedef#1\fi}
\def\cite#1{\csname#1\endcsname}
\citedef  Borh BJ BK BK1 Bo BrB BrK BrK2 DDeC DeSK Dix Eis GG G Gi0
Gi H Im Ja1 Ka Kats Kraft Kr Lo Lo1 LS Mc Mc1 Noe P95 P02  P03 P07
P07' PS Serre Sh ShT Sk Sl Sp Ve
\endcitedef
\newtoks\nextauth
\newif\iffirstauth
\def\checkendauth#1{\ifx\endauth#1
        \iffirstauth\the\nextauth
        \else{} and \the\nextauth\fi,
    \else\iffirstauth\the\nextauth\firstauthfalse
        \else, \the\nextauth\fi
        \expandafter\auth\expandafter#1\fi}
\def\auth#1 #2 {\nextauth={#1 #2}\checkendauth}
\newif\ifinbook
\newif\ifbookref
\def\nextref#1 {\bookreffalse\inbookfalse
    \bibitem[\cite{#1}]{}
    \firstauthtrue
    \ignorespaces}
\def\paper#1{{\it#1,}}
\def\In#1{\inbooktrue In #1,}
\def\book#1{\bookreftrue{\it#1,}}
\def\journal#1{#1\ifinbook,\fi}
\def\bookseries#1{#1,}
\def\Vol#1{\ifbookref Vol. #1,\else\ifinbook Vol. #1,\else{\bf#1}\fi\fi
    \space\ignorespaces}

\def\publisher#1{#1,}
\def\Year#1{\ifbookref #1.\else\ifinbook #1,\else(#1)\fi\fi
    \space\ignorespaces}
\def\Pages#1{\ifinbook pp. #1.\else #1.\fi}
\begin{document}
\title{Commutative quotients of finite $W$-algebras}
\author{Alexander Premet}
\thanks{\nonumber{\it Mathematics Subject Classification} (2000 {\it revision}).
Primary 17B35. Secondary 17B63, 17B81.}
\address{School of Mathematics, University of Manchester, Oxford Road,
M13 9PL, UK} \email{sashap@maths.man.ac.uk}
\begin{abstract}
\noindent Let $e$ be a nilpotent element in a Chevalley form $\g_\Z$
of a simple Lie algebra $\g$ over $\mathbb C$ and let
$\bar{e}=e\otimes 1$ be the corresponding nilpotent element in the
restricted Lie algebra $\g_\k=\g\otimes_\Z\,\k$, where $\k$ is the
algebraic closure of ${\mathbb F}_p$. Assume that $p\gg 0$ and set
$\chi:=\kappa(\bar{e},\,\cdot\,)$, where $\kappa$ is the Killing
form of $\g_\k$. Let $G_\k$ be the simple, simply connected
algebraic $\k$-group with $\g_\k=\Lie(G_\k)$, write ${\mathcal
O}_\k$ for the adjoint $G_\k$-orbit of $\bar{e}$, and denote by
$U(\g,e)$ the finite $W$-algebra associated to $e$. In this paper we
prove that if $U(\g,e)$ has a $1$-dimensional representation, then
the reduced enveloping algebra $U_\chi(\g_\k)$ possesses a simple
module of dimension $p^{d(\bar{e})}$, where $d(\bar{e})$ is the
half-dimension of ${\mathcal O}_\k$.   We also show that if $e$ is
induced from a nilpotent element $e_0$ in a Levi subalgebra
$\mathfrak l$ of $\g$ and the finite $W$-algebra $U([{\mathfrak l},
{\mathfrak l}],e_0)$ admits $1$-dimensional representations, then so
does $U(\g,e)$. This reduces the problem of $1$-dimensional
representations for finite $W$-algebras to the case where $e$ is a
rigid nilpotent element in a Lie algebra of type ${\rm F}_4$, ${\rm
E}_6$, ${\rm E}_7$, ${\rm E}_8$. We use Katsylo's results on
sections of sheets to determine, in many cases, the Krull dimension
of the largest commutative quotient of the algebra $U(\g,e)$
\end{abstract}
\maketitle

\section{\bf Introduction}
\subsection{}\label{1.1} This paper is a continuation of
[\cite{P07'}]. Let $U(\g)$ denote the universal enveloping algebra
of a finite-dimensional simple Lie algebra $\g$ over $\mathbb C$.
Roughly speaking, the main result of [\cite{P07'}] states that the
primitive ideals of $U(\g)$ having rational infinitesimal characters
admit {\it finite} generalised Gelfand--Graev models. One of the
goals of this paper is to remove the unnecessary rationality
assumption from the statement of [\cite{P07'}, Thm.~1.1] and thus
confirm [\cite{P07}, Conjecture~3.2] in full generality; see
Theorem~\ref{main}. This was announced in [\cite{P07'}, p.~745], and
very few changes to the original proof in [\cite{P07'}] are actually
required.

In the meantime two different proofs of [\cite{P07}, Conjecture~3.2]
have appeared in the literature; the first one was found by Losev in
[\cite{Lo}] and the second one by Ginzburg in [\cite{Gi}]. Our proof
relies on the method developed in [\cite{P07'}], the only difference
being that in the present case our base ring is a finitely generated
$\Z$-subalgebra of $\mathbb C$ rather than $\mathbb Q$. In this
setting, we have to produce sufficiently many primes $p$ for which
the reduction procedure described in [\cite{P07'}] leads to
irreducible representations of the $p$-Lie algebra
$\g_\Z\otimes_\Z\overline{\mathbb F}_p$ with $p$-characters
belonging to the modular counterpart of our initial nilpotent orbit;
see Section~4.
\subsection{}\label{1.2}
Denote by $G$ a simple, simply connected algebraic group over
$\mathbb C$,  let $(e,h,f)$ be an $\mathfrak{sl}_2$-triple in the
Lie algebra $\g=\Lie(G)$, and denote by $(\,\cdot\,,\,\cdot\,)$ the
$G$-invariant bilinear form on $\g$ for which $(e,f)=1$. Let
$\chi\in\g^*$ by such that $\chi(x)=(e,x)$ for all $x\in\g$ and
write $U(\g,e)$ for the quantisation of the Slodowy slice
$e+\text{Ker}\,\ad\,f$ to the adjoint orbit $\O:=(\Ad G)e$. Recall
that $U(\g,e)\,=\,{\End}_{\g}\,(Q_\chi)^{\text{op}}$, where $Q_\chi$
is the generalised Gelfand--Graev $\g$-module associated with the
triple $(e,h,f)$. The module $Q_\chi$ is induced from a
$1$-dimensional module ${\mathbb C}_\chi$ over of a nilpotent
subalgebra $\m$ of $\g$ whose dimension equals
$\frac{1}{2}\dim\,\O$. The Lie subalgebra $\m$ is $(\ad h)$-stable,
all eigenvalues of $\ad h$ on $\m$ are negative, and $\chi$ vanishes
on $[\m,\m]$. The action of $\m$ on ${\mathbb C}_\chi={\mathbb C}
1_\chi$ is given by $x(1_\chi)=\chi(x)1_\chi$ for all $x\in\m$; see
[\cite{P02}, \cite{GG}] for more detail. The algebra $U(\g,e)$
shares many remarkable features with the universal enveloping
algebra $U(\g)$ and is often referred to as the enveloping algebra
of the Slodowy slice to $\O$. As an example, $U(\g,e)$ is a
deformation of the universal enveloping algebra $U(\z_\chi)$, where
$\z_\chi$ is the stabiliser of $\chi$ in $\g$; see [\cite{P07}]. It
is also known that $U(\g,e)$ is isomorphic to the Zhu algebra of the
vertex $W$-algebra $W^{\rm aff}(\g,e)$. The Zhu algebra of $W^{\rm
aff}(\g,e)$ is, in turn, isomorphic to the finite $W$-algebra
$W^{\rm fin}(\g,e)$ associated with $\g$ and $e$; see [\cite{DDeC}]
and [\cite{DeSK}].
\subsection{}\label{1.3} In [\cite{P07}], the author conjectured
that every algebra $U(\g,e)$ admits a $1$-dimensional
representation; see [\cite{P07}, Conjecture~3.1(1)]. In [\cite{Lo}],
Losev proved this conjecture for $\g$ classical. In this paper, we
take another step towards proving [\cite{P07}, Conjecture~3.1(1)].
Recall that $\O$ is said to be {\it induced} from a nilpotent orbit
$\O_0$ in a Levi subalgebra $\l$ of $\g$, if $\O$ intersects densely
with the Zariski closed set $\overline{\O}_0+\n$, where $\n$ is the
nilradical of a parabolic subalgebra of $\g$ with Levi component
$\l$. If $\O$ is not induced, then one says that $\O$ is a {\it
rigid} orbit.
\begin{theorem}\label{A}
Suppose the orbit $\O$ is induced from a nilpotent orbit $\O_0$ in a
proper Levi subalgebra $\l$ of $\g$, and let $e_0\in\O_0$. If the
finite $W$-algebra $U([\l,\l],e_0)$ admits a $1$-dimensional
representation, then so does $U(\g,e)$.
\end{theorem}
Theorem~\ref{A} is proved in Section~3. Combined with [\cite{Lo},
Thm.~1.2.3(1)] it reduces proving [\cite{P07}, Conjecture~3.1(1)] to
the case of rigid nilpotent orbits in exceptional Lie algebras. We
say that $\g$ is {\it well-behaved} if for any proper Levi
subalgebra $\l$ of $\g$ and any nilpotent element $e_0\in\l$ the
finite $W$-algebra $U([\l,\l],e_0)$ admits a $1$-dimensional
representation. In view of [\cite{Lo}, Thm.~1.2.3(1)] the Lie
algebras of types ${\rm A_\ell,\,B_\ell,\,C_\ell
,\,D_\ell,\,G_2,\,F_4,\,E_6}$ are well-behaved.

Given an associative algebra $\Lambda$ we denote by $\Lambda^{\rm
ab}$ the factor-algebra $\Lambda/\Lambda\cdot[\Lambda,\Lambda]$,
where $\Lambda\cdot[\Lambda,\Lambda]$ is the ideal of $\Lambda$
generated by all commutators $[a,b]$ with $a,b\in\Lambda$. Clearly,
$\Lambda^{\rm ab}$ is the largest commutative quotient of $\Lambda$.
Since $U(\g,e)$ is Noetherian, by [\cite{P02}, 4.6], so is the
commutative $\CC$-algebra $U(\g,e)^{\rm ab}$. By Hilbert's
Nullstellensatz, the maximal spectrum ${\E}:={\rm Specm}\,U(\g,
e)^{\rm ab}$ parametrises the $1$-dimensional representations of
$U(\g,e)$. Our main goal in Section~3 is to determine the Krull
dimension of the algebra $U(\g,e)^{\rm ab}$ under the assumption
that $\g$ is well-behaved. In proving the main results of Section~3
we shall rely on Borho's classification of sheets in semisimple Lie
algebras and Katsylo's results on sections of sheets.

Given $x\in \g$ we denote by $G_x$ the centraliser of $x$ in $G$.
For $d\in\N$, set $\g^{(d)}:=\{x\in\g\,|\,\,\dim\,G_x=d\}$. The
irreducible components of the quasi-affine variety $\g^{(d)}$ are
called {\it sheets} of $\g$. The sheets are $(\Ad G)$-stable,
locally closed subsets of $\g$. It is well-known that  every sheet
contains a unique nilpotent orbit and there is a bijection between
the sheets of $\g$ and the $G$-conjugacy classes of pairs
$(\l,\O_0)$, where $\l$ is a Levi subalgebra of $\g$ and $\O_0$ is a
rigid nilpotent orbit in $[\l,\l]$.

If $\l$ if a Levi subalgebra of $\g$, then the centre $\z(\l)$ of
$\l$ is a toral subalgebra of $\g$. Denote by $\z(\l)_{\rm reg}$ the
set of all $z\in\z(\l)$ for which $\ad\,z$ acts invertibly $\g/\l$.
Given a nilpotent element $e_0\in [\l,\l]$ define ${\mathcal
D}(\l,e_0):=(\Ad G)\cdot (e_0+\z(\l)_{\rm reg})$, a locally closed
subset of $\g$, and call ${\mathcal D}(\l, e_0)$ a {\it
decomposition class} of $\g$. By [\cite{Borh}], every sheet
${\mathcal S}$ of $\g$ contains a unique open decomposition class.
Moreover, if ${\mathcal D}(\l, e_0)$ is such a class, then
$\O_0:=(\Ad L)\cdot e_0$ is rigid in $[\l,\l]$ and the $(\Ad
G)$-orbit induced from $\O_0$ is contained in $\mathcal S$ (here $L$
is the Levi subgroup of $G$ with $\Lie(L)=\l$).

The group $C(e):=G_e\cap G_f$ is reductive and its finite quotient
$\Gamma(e):=C(e)/C(e)^\circ$ identifies with the component group of
$G_e$. If ${\mathcal S}(e)$ is a sheet containing $e$, then the set
$X:={\mathcal S}(e)\cap (e+\ker \ad f)$ is $C(e)$-stable and Zariski
closed in $\g$. By [\cite{Kats}], the identity component
$C(e)^\circ$ acts trivially on $X$ and the component group
$\Gamma(e)$ permutes transitively the irreducible components of $X$.
Furthermore, if ${\mathcal D}(\l, e_0)$ is the open decomposition
class of ${\mathcal S}(e)$ and $Y$ is any irreducible component of
$X$, then $\dim\,Y=\dim\,\z(\l)$.

For an algebraic variety $Z$, we denote by ${\rm Comp}(Z)$ the set
of all irreducible components of $Z$. Our main result in Section~3
is the following:
\begin{theorem}\label{B}
Suppose $\g$ is well-behaved and $\O$ is not rigid. Let ${\mathcal
S}_1,\ldots, {\mathcal S}_t$ be the pairwise distinct sheets of $\g$
containing $e\in\O$. Let ${\mathcal D}(\l_i,e_i)$ be the open
decomposition class of ${\mathcal S}_i$ and $X_i={\mathcal S}_i\cap
(e+\ker\ad f)$. Then there is a surjection
$$\tau\colon\,{\rm Comp}(\E)\twoheadrightarrow\,{\rm
Comp}(X_1)\sqcup\ldots\sqcup{\rm Comp}(X_t)$$ such that
$\dim\,{\mathcal Y}=\dim\,\z(\l_i)$ for every ${\mathcal
Y}\in\tau^{-1}({\rm Comp}(X_i))$, where $1\le i\le t$.
\end{theorem}
It follows from Theorem~\ref{B} that if $\g$ is well-behaved and
$\O$ is not rigid, then
$$\dim\,U(\g,e)^{\rm ab}=\,\max_{1\le i\le t}\,\dim\,\z(\l_i).$$
We also show in Section~3 that if $\O$ is rigid and $e\in \O$, then
$\E$ is a finite set (possibly empty). In this case we do not
require $\g$ to be well-behaved.

For $\g=\mathfrak{gl}(N)$, we obtain a much stronger result. Recall
that to any partition $\lambda=(p_n\ge p_{n-1}\ge\cdots\ge p_1)$ of
$N$ there corresponds a nilpotent element
$e_\lambda\in\mathfrak{gl}(N)$ of Jordan type $(p_1,p_2,\ldots,
p_n)$, and any nilpotent element in $\mathfrak{gl}(N)$ is conjugate
to one of the $e_\lambda$'s. At the end of Section~3 we show that
$$U(\mathfrak{gl}(N),e_\lambda)^{\rm ab}\,\cong\,\CC[X_1,\ldots,X_l],\quad\ l=p_n.$$
In proving this isomorphism we use Theorem~\ref{B} and the explicit
presentation of finite $W$-algebras of type $A$ found by
Brundan--Kleshchev in [\cite{BrK}].
\subsection{}\label{1.4} Our proof of Theorems~\ref{A} and \ref{B}
relies on characteristic $p$ methods developed in [\cite{P07'}]. We
have to generalise several technical results proved in
[\cite{P07'}]; see Section~2. The algebra $U(\g,e)$ is defined over
a suitable localisation $A=\Z[d^{-1}]$ of $\Z$. More precisely,
there exists an $A$-subalgebra $U(\g_A,e)$ of $U(\g,e)$ free as an
$A$-module and such that $U(\g,e)\cong U(\g_A,e)\otimes_A\CC$. We
take a sufficiently large prime $p$ invertible in $A$, denote by
$\k$ the algebraic closure of ${\mathbb F}_p$, and set $U(\g_\k,
e)=U(\g_A,e)\otimes_A\k$. Here $\g_\k=\g_\Z\otimes_\Z\k$, where
$\g_\Z$ is a Chevalley $\Z$-form of $\g$ containing $e$.  We
identify $e$ with its image in $\g_\k$ and regard
$\chi=(e,\,\cdot\,)$ as a linear function on $\g_\k$ (this is
possible because the bilinear form $(\,\cdot\,,\,\cdot\,)$ is
$A$-valued).

The subalgebra $\m$ from (\ref{1.2}) is defined over $A$ and we set
$\m_\k:=\m_A\otimes_A\k$, where $\m_A=\m\cap\g_A$ (it can be assumed
that $\m_A$ is a free $A$-module). By construction, the Lie algebra
$\m_\k$ possesses a $1$-dimensional module on which it acts via
$\chi$; we call it $\k_\chi$. We then consider the induced
$\g_\k$-module $Q_{\chi,\,\k}:=U(\g_\k)\otimes_{U(\m_\k)}\,\k_\chi$,
denote by $\rho_\k$ the corresponding representation of $U(\g_\k)$,
and define
$$\widehat{U}(\g_\k,e)\,:=\,({\End}_{\g_\k}\,Q_{\chi,\,\k})^{\rm op}.$$
It is easy to see that $U(\g_\k,e)$ is a subalgebra of
$\widehat{U}(\g_\k,e)$. Let $Z_p=Z_p(\g_\k)$ denote the $p$-centre
of $U(\g_\k)$ (it is generated by all $x^{p}-x^{[p]}$ with
$x\in\g_\k$, where $x\mapsto x^{[p]}$ is the $p$-th power map of the
restricted Lie algebra $\g_\k$). Clearly, $\rho_\k(Z_p)\subseteq
\widehat{U}(\g_\k,e)$. Given a subspace $V$ of $\g_\k$ we write
$Z_p(V)$ for the subalgebra of $Z_p$ generated by all $v^p-v^{[p]}$
with $v\in V$. In Section~2 we prove:
\begin{theorem}\label{C}
The algebra $\widehat{U}(\g_\k,e)$ is generated by $U(\g_\k,e)$ and
$\rho_\k(Z_p)$; moreover, $\widehat{U}(\g_\k,e)$ is a free
$\rho_\k(Z_p)$-module of rank $p^r$, where $r=\dim G_e$. There is a
subspace $\a_\k$ of $\g_\k$ with $\dim\,\a_\k=\frac{1}{2}\dim\,\O$
such that $\widehat{U}(\g_\k,e)\cong U(\g_\k,e)\otimes_\k
Z_p(\a_\k)$ as $\k$-algebras.
\end{theorem}
Let $G_\k$ be a simple, simply connected algebraic $\k$-group with
$\Lie(G_\k)=\g_\k$. Recall that for $\xi\in\g_\k^*$ the {\it reduced
enveloping algebra} $U_\xi(\g_\k)$ is defined as the quotient of
$U(\g_\k)$ by its ideal generated by all $x^p-x^{[p]}-\xi(x)^p$ with
$x\in\g_\k$. One of the challenging open problems in the
representation theory of $\g_\k$ is to show that for every
$\xi\in\g_\k^*$ the reduced enveloping algebra $U_\xi(\g_\k)$ has a
simple module of dimension $p^{(\dim\,\O(\xi))/2}$, where
$\O(\xi)=({\rm Ad}^*\,G_\k)\xi$. As a consequence of Theorem~\ref{C}
we obtain:
\begin{theorem}\label{D}
If the finite $W$-algebra $U(\g,e)$ admits a $1$-dimensional
representation, then for $p\gg 0$ the reduced enveloping algebra
$U_\chi(\g_\k)$ has a simple module of dimension
$p^{(\dim\,\O(\chi))/2}$.
\end{theorem}
Together with Theorem~\ref{A} and [\cite{Lo}, Thm.~1.2.3(1)] this
yields:
\begin{corollary}\label{D'}
If $\g$ is classical and $p\gg 0$, then for any $\xi\in\g_\k^*$ the
reduced enveloping algebra $U_\xi(\g_\k)$ has a simple module of
dimension $p^{(\dim\,\O(\xi))/2}$.
\end{corollary}
It also follows from Theorems~\ref{A} and \ref{D} that if $\O$ is
induced from $\O_0\subset\l$ and the finite $W$-algebra
$U([\l,\l],e_0)$ with $e_0\in\O_0$ has a $1$-dimensional
representation, then for $p\gg 0$ the reduced enveloping algebra
$U_\chi(\g_\k)$ has a module of dimension $p^{(\dim\,\O(\chi))/2}$.

\bigskip

\noindent{\bf Acknowledgement.} Part of this work was done during my
stay at the Max Planck Institut f{\"u}r Mathematik (Bonn) in the
spring of 2007. I would like to thank the institute for worm
hospitality and support. The results presented here were announced
in my talks at the MSRI workshop on Lie Theory (March 2008) and at
the conference in celebration of the 65th birthday of Victor Kac
(Cortona, Italy, June 2008).

\section{\bf Finite $W$-algebras and their modular analogues}
\subsection{}\label{4.1}
Let $G$ be a simple, simply connected algebraic group over $\mathbb
C$, and $\g=\Lie(G)$. Let $\h$ be a Cartan subalgebra of $\g$ and
$\Phi$ the root system of $\g$ relative to $\h$. Choose a basis of
simple roots $\Pi=\{\alpha_1,\ldots,\alpha_\ell\}$ in $\Phi$, let
$\Phi^+$ be the corresponding positive system in $\Phi$, and put
$\Phi^-:=-\Phi^+$. Let $\g=\n^-\oplus\h\oplus\n^+$ be the
corresponding triangular decomposition of $\g$ and choose a
Chevalley basis ${\mathcal
B}=\{e_\gamma\,|\,\,\gamma\in\Phi\}\cup\{h_\alpha\,|\,\,\alpha\in\Pi\}$
in $\g$. Set ${\mathcal
B}^{\pm}:=\{e_\alpha\,|\,\,\alpha\in\pm\Phi^+\}$. Let $\g_\Z$ and
$U_\Z$ denote the Chevalley $\Z$-form of $\g$ and the Kostant
$\Z$-form of $U(\g)$ associated with $\mathcal B$. Given a
$\Z$-module $V$ and a $\Z$-algebra $A$, we write $V_A:=
V\otimes_{\Z}A$.

Take a nonzero nilpotent element $e\in\g_\Z$ and choose
$f,h\in\g_{\mathbb Q}$ such that $(e,h,f)$ is an $\sl_2$-triple in
$\g_{\mathbb Q}$. Denote by $(\,\cdot\,,\,\cdot\,)$ a scalar
multiple of the Killing form $\kappa$ of $\g$ for which $(e,f)=1$
and define $\chi\in\g^*$ by setting $\chi(x)=(e,x)$ for all $x\in\g$
(it follows from the $\sl_2$-theory that $\kappa(e,f)$ is a positive
integer). Given $x\in\g$ we set ${\mathcal O}(x):=(\Ad G)\cdot x$
and $d(x):=\frac{1}{2}\dim {\mathcal O}(x)$.
\begin{defn}
We call a commutative ring $A$ {\it admissible} if $A$ is a finitely
generated $\Z$-subalgebra of $\mathbb C$, $\kappa(e,f)\in A^\times$,
and all bad primes of the root system of $G$ and the determinant of
the Gram matrix of $(\,\cdot\,,\,\cdot\,)$ relative to a Chevalley
basis of $\g$ are invertible in $A$.
\end{defn}
It is clear from the definition that every admissible ring is a
Noetherian domain. Given a finitely generated $\Z$-subalgebra $A$ of
$\CC$ we denote by $\pi(A)$ the set of all primes $p\in \mathbb N$
such that $A/{\mathfrak P}\cong{\mathbb F}_p$ for some maximal ideal
${\mathfrak P}$ of $A$.

Let $\g(i)=\{x\in\g\,|\,\,[h,x]=ix\}$. Then
$\g=\bigoplus_{i\in\Z}\,\g(i)$, by the $\mathfrak{sl}_2$-theory, and
all subspaces $\g(i)$ are defined over $\mathbb Q$. Also,
$e\in\g(2)$ and $f\in\g(-2)$. We define a skew-symmetric bilinear
form $\la\,\cdot\,,\,\cdot\,\ra$ on $\g(-1)$ by setting $\la
x,y\ra:=(e,[x,y])$ for all $x,y\in\g(-2)$. This skew-symmetric
bilinear form is nondegenerate, hence there exists a basis
$B=\{z_1',\ldots,z_s',z_1,\ldots,z_s\}$ of $\g(-1)$ contained in
$\g_\mathbb Q$ and such that
$$\la z_i',z_j\ra=\delta_{ij},\qquad\ \la z_i,z_j\ra\,=\,\la z_i',z_j'\ra=0
\qquad\quad\ (1\le i,j\le s).$$ As explained in [\cite{P07'}, 4.1],
after enlarging $A$ if need be, one can assume that
$\g_A=\bigoplus_{i\in\,\Z}\,\g_A(i)$, that each
$\g_A(i):=\g_A\cap\g(i)$ is a freely generated over $A$ by a basis
of the vector space $\g(i)$, and that $B$ is a free basis of the
$A$-module $\g_A(-1)$.

Put $\m:=\g(-1)^0\oplus\sum_{i\le  -2}\,\g(i)$ where $\g(-1)^0$
denotes the $\mathbb C$-span of $z_1',\ldots, z_s'$. Then $\m$ is a
nilpotent Lie subalgebra of dimension $d(e)$ in $\g$ and $\chi$
vanishes on the derived subalgebra of $\m$; see [\cite{P02}] for
more detail. It follows from our assumptions on $A$ that
$\m_A=\g_A\cap\m$ is a free $A$-module and a direct summand of
$\g_A$. More precisely, $\m_A=\g_A(-1)^0\oplus\sum_{i\le
-2}\,\g_A(i)$, where $\g_A(-1)^0= \g_A\cap\g(-1)=
Az_1'\oplus\cdots\oplus Az_s'$.

Enlarging $A$ further we may assume that $e,f\in\g_A$ and that
$[e,\g_A(i)]$ and $[f,\g_A(i)]$ are direct summands of $\g_A(i+2)$
and $\g_A(i-2)$, respectively. Then $\g_A(i+2)=[e,\g_A(i)]$ for all
$i\ge 0$; see [\cite{P07'}, 4.1].

Write $\g_e$ for the centraliser of $e$ in $\g$. Similar to
[\cite{P02}, 4.2 and 4.3] we choose a basis $x_1,\ldots,
x_r,x_{r+1},\ldots, x_m$ of the free $A$-module
$\p_A:=\bigoplus_{i\ge 0}\,\g_A(i)$ such that
\begin{itemize}
\item[(a)]
$x_i\in\g_A(n_i)$ for some $n_i\in\Z_+$;
\item[(b)]
$x_1,\ldots, x_r$ is a free basis of the $A$-module $\g_A\cap\g_e$;
\item[(c)]
$x_{r+1},\ldots, x_m \in[f,\g_A]$.
\end{itemize}

\subsection{}\label{4.2}
Let $Q_\chi$ be the generalised Gelfand-Graev $\g$-module associated
to $e$. Recall that $Q_\chi=U(\g)\otimes_{U(\m)}{\mathbb C}_\chi$,
where ${\mathbb C}_\chi={\mathbb C}1_\chi$ is a $1$-dimensional
$\m$-module such that $x\cdot 1_\chi=\chi(x)1_\chi$ for all
$x\in\m$. Given $({\bf a},{\bf b})\in\Z_+^m\times\Z_+^s$ we let
$x^{\bf a}z^{\bf b}$ denote the monomial $x_1^{a_1}\cdots
x_m^{a_m}z_1^{b_1}\cdots z_s^{b_s}$ in $U(\g)$. Set
$Q_{\chi,\,A}:=U(\g_A)\otimes_{U(\m_A)}A_\chi$, where
$A_\chi=A1_\chi$. Note that $Q_{\chi,\,A}$ is a $\g_A$-stable
$A$-lattice in $Q_\chi$ with $\{x^{\bf i}z^{\bf j}\otimes
1_\chi,\,|\,\,({\bf i},{\bf j})\in\Z_+^m\times\Z_+^s\}$ as a free
basis; see [\cite{P07'}] for more detail. Given $({\bf a},{\bf
b})\in\Z_+^m\times Z_+^s$ we set
$$|({\bf a},{\bf b})|_e\,:=\,\,\sum_{i=1}^m\,a_i(n_i+2)+\sum_{i=1}^s\,b_i.$$
According to [\cite{P02}, Thm.~4.6], the algebra
$U(\g,e):=(\End_\g\,Q_\chi)^{\rm op}$ is generated over $\mathbb C$
by endomorphisms $\Theta_1,\ldots,\Theta_r$ such that
\begin{eqnarray}\label{lam} \qquad\ \,\Theta_k(1_\chi)\,=\,\Big(x_k+\sum_{0<|({\bf i},{\bf j})|_e\le
n_k+2}\,\lambda_{{\bf i},\,{\bf j}}^k\,x^{\bf i}z^{\bf
j}\Big)\otimes 1_\chi,\qquad\quad \ 1\le k\le r,\end{eqnarray} where
$\lambda_{{\bf i},\,{\bf j}}^k\in\mathbb Q$ and $\lambda_{{\bf
i},\,{\bf j}}^k=0$ if either $|({\bf i},\,{\bf j})|_e=n_k+2$ and
$|{\bf i}|+|{\bf j}|=1$ or ${\bf i}\ne {\bf 0}$, ${\bf j}={\bf 0}$,
and $i_l=0$ for $l>r$. Moreover, the monomials
$\Theta_1^{i_1}\cdots\Theta_r^{i_r}$ with $(i_1,\ldots,
i_r)\in\Z_+^r$ form a PBW basis of the vector space $U(\g,e)$.

The monomial $\Theta_1^{i_1}\cdots\Theta_r^{i_r}$ is said to have
{\it Kazhdan degree} $\sum_{i=1}^r\,a_i(n_i+2)$. For $k\in\Z_+$ we
let $U(\g,e)_k$ denote the $\mathbb C$-span of all monomials
$\Theta_1^{i_1}\cdots\Theta_r^{i_r}$ of Kazhdan degree $\le k$. The
union $\bigcup_{k\ge 0}\,U(\g,e)_k$ is an increasing algebra
filtration of $U(\g,e)$, called the {\it Kazhdan filtration}; see
[\cite{P02}]. The corresponding graded algebra $\gr\,U(\g,e)$ is a
polynomial algebra in $\gr\,\Theta_1,\ldots,\gr\,\Theta_r$. It is
immediate from [\cite{P02}, Thm.~4.6] that there exist polynomials
$F_{ij}\in{\mathbb Q}[X_1,\ldots, X_r]$, where $1\le i<j\le r$, such
that
\begin{eqnarray}\label{relations}
\qquad\
[\Theta_i,\Theta_j]\,=\,F_{ij}(\Theta_1,\ldots,\Theta_r)\qquad\quad\
\ (1\le i<j\le r).
\end{eqnarray}
Moreover, if $[x_i,x_j]\,\,=\,\,\sum_{k=1}^r \alpha_{ij}^k\, x_k$ in
$\g_e$, then
$$F_{ij}(\Theta_1,\ldots,\Theta_r)\,\equiv\,\sum_{k=1}^r\alpha_{ij}^k\Theta_k
+q_{ij}(\Theta_1,\ldots,\Theta_r)\ \ \ \big({\rm
mod}\,U(\g,e)_{n_i+n_j}\big),$$ where the initial form of $q_{ij}\in
{\mathbb Q}[X_1,\ldots, X_r]$ has total degree $\ge 2$ whenever
$q_{ij}\ne 0$. By [\cite{P07'}, Lemma~4.1], the algebra $U(\g,e)$ is
generated by $\Theta_1,\ldots,\Theta_r$ subject to the relations
(\ref{relations}). As in [\cite{P07'}], we assume that our
admissible ring $A$ contains all $\lambda_{{\bf i},{\bf j}}^k$ in
(\ref{lam}) and all coefficients of the $F_{ij}$'s in
(\ref{relations}).
\subsection{} \label{4.2'}
Let $N_\chi$ denote the ideal of codimension one in $U(\m)$
generated by all $x-\chi(x)$ with $x\in \m$. Then $Q_\chi\cong
U(\g)N_\chi$ as $\g$-modules. By construction, the left ideal
${\mathcal I}_\chi:=U(\g)N_\chi$ of $U(\g)$ is a $\big(U(\g),
U(\m)\big)$-bimodule. The fixed point space $(U(\g)/{\mathcal
I}_\chi)^{\ad \m}$ carries a natural algebra structure given by
$(x+{\mathcal I}_\chi)\cdot (y+{\mathcal I}_\chi)=xy+{\mathcal
J}_\chi$ for all $x,y\in U(\g)$. Moreover, $U(\g)/{\mathcal
I}_\chi\cong Q_\chi$ as $\g$-modules via the $\g$-module map sending
$1+{\mathcal J}_\chi$ to $1_\chi$, and $(U(\g)/{\mathcal
I}_\chi)^{\ad \m}\cong U(\g,e)$ as algebras. Any element of
$U(\g,e)$ is uniquely determined by its effect on the generator
$1_\chi\in Q_\chi$ and the canonical isomorphism between
$(U(\g)/{\mathcal I}_\chi)^{\ad \m}$ and $U(\g,e)$ is given by
$u\mapsto u(1_\chi)$ for all $u\in(U(\g)/{\mathcal I}_\chi)^{\ad
\m}$. It is clear that this isomorphism is defined over $A$. In what
follows we shall identify $Q_\chi$ with $U(\g)/{\mathcal I}_\chi$
and $U(\g,e)$ with $(U(\g)/{\mathcal I}_\chi)^{\ad \m}$.

Let $U(\g)=\bigcup_{j\in\Z}\, {\sf K}_jU(\g)$ be the Kazhdan
filtration of $U(\g)$; see [\cite{GG}, 4.2]. Recall that ${\sf
K}_jU(\g)$ is the $\mathbb C$-span of all products $x_1\cdots x_t$
with $x_i\in\g(n_i)$ and $\sum_{i=1}^t\, (n_i+2)\le j$ (the identity
element is in ${\sf K}_0U(\g)$ by convention). The Kazhdan
filtration on $Q_\chi$ is defined by ${\sf K}_jQ_\chi:=\pi({\sf
K}_jU(\g)),$ where $\pi\colon U(\g)\twoheadrightarrow
U(\g)/{\mathcal I}_\chi$ is the canonical homomorphism; see
[\cite{GG}, 4.3]. It turns $Q_\chi$ into a filtered $U(\g)$-module.
As explained in [\cite{GG}] the Kazhdan grading of $\gr\, Q_\chi$
has no negative components. The Kazhdan filtration of $U(\g,e)$
defined in (\ref{4.2}) is nothing but the filtration of
$U(\g,e)=(U(\g)/{\mathcal I}_\chi)^{\ad \m}$ induced from the
Kazhdan filtration of $Q_\chi$ through the embedding
$(U(\g)/{\mathcal I}_\chi)^{\ad \m}\hookrightarrow Q_\chi$; see
[\cite{GG}] for more detail.

Let $U(\g_A,e)$ denote the $A$-span of all monomials
$\Theta_1^{i_1}\cdots\Theta_r^{i_r}$ with
$(i_1,\ldots,i_r)\in\Z_+^r$. Our assumptions on $A$  guarantee that
$U(\g_A,e)$ is an $A$-subalgebra of $U(\g,e)$ contained in
$(\End_{\g_A}\,Q_{\chi,\,A})^{\rm op}$. It is immediate from the
above discussion that $Q_{\chi,A}$ identifies with the $\g_A$-module
$U(\g_A)/U(\g_A)N_{\chi,A},$ where $N_{\chi,A}$ stands for the
$A$-subalgebra of $U(\m_A)$ generated by all $x-\chi(x)$ with $x\in
\m_A$. Hence $U(\g_A,e)$ embeds into the $A$-algebra
$\big(U(\g_A)/U(\g_A)N_{\chi,A}\big)^{\ad\,\m_A}$. Since
$Q_{\chi,\,A}$ is a free $A$-module with basis $\{x^{\bf i}z^{\bf
j}\otimes 1_\chi,\,|\,\,({\bf i},{\bf j})\in\Z_+^m\times\Z_+^s\}$,
easy induction on Kazhdan degree (based on [\cite{P02}, Lemma~4.5]
and the formula displayed in [\cite{P02}, p.~27]) shows that
\begin{eqnarray}\label{A-algebras}
U(\g_A,e)\,=\,({\End}_{\g_A}\,Q_{\chi,\,A})^{\rm op}\,\cong\,
\big(U(\g_A)/U(\g_A)N_{\chi,A}\big)^{\ad\,\m_A}.
\end{eqnarray}
Repeating verbatim Skryabin's argument in [\cite{P02}, p.~53] one
also observes that $Q_{\chi,\,A}$ is free as a right
$U(\g_A,e)$-module.
\subsection{}\label{4.3}
We now pick $p\in\pi(A)$ and denote by $\k$ the algebraic closure of
${\mathbb F}_p$. Since the form $(\,\cdot\,,\,\cdot\,)$ is
$A$-valued on $\g_A$, it induces a symmetric bilinear form on the
Lie algebra $\g_\k\cong\g_A\otimes_A\k$. We use the same symbol to
denote this bilinear form on $\g_\k$. Let $G_\k$ be the simple,
simply connected algebraic $\k$-group with hyperalgebra
$U_\k=U_\Z\otimes_Z\k$. Note that $\g_\k=\Lie(G_\k)$ and the form
$(\,\cdot\,,\,\cdot\,)$ is $(\Ad\, G_\k)$-invariant and
nondegenerate. For $x\in\g_A$ we set $\bar{x}:=x\otimes 1$, an
element of $\g_\k$. To ease notation we identify $e, f$ with the
nilpotent elements $\bar{e},\bar{f}\in\g_\k$ and $\chi$ with the
linear function $(e,\,\cdot\,)$ on $\g_\k$ (this will cause no
confusion).

The Lie algebra $\g_\k=\Lie(G_\k)$ carries a natural $[p]$-mapping
$x\mapsto x^{[p]}$ equivariant under the adjoint action of $G_\k$.
For every $x\in\g_\k$ the element $x^p-x^{[p]}$ of the universal
enveloping algebra $U(\g_\k)$. The subalgebra of $U(\g_\k)$
generated by all $x^p-x^{[p]}\in U(\g_\k)$ is called the $p$-{\it
centre} of $U(\g_\k)$ and denoted $Z_p(\g_\k)$ or $Z_p$ for short.
It is immediate from the PBW theorem that $Z_p$ is isomorphic to a
polynomial algebra in $\dim\g$ variables and $U(\g_\k)$ is a free
$Z_p$-module of rank $p^{\dim\,\g}$. For every maximal ideal $J$ of
$Z_p$ there is a unique linear function $\eta=\eta_J\in\g_\k^*$ such
that
$$J=\,\langle x^p-x^{[p]}-\eta(x)^p1\,|\,\,\,x\in\g_\k\rangle.$$
Since the Frobenius map of $\k$ is bijective, this enables us to
identify the maximal spectrum ${\rm Specm}\,Z_p$ with $\g_\k^*$.

Given $\xi\in\g_\k^*$ we denote by $I_\xi$ the two-sided ideal of
$U(\g_\k)$ generated by all $x^p-x^{[p]}-\xi(x)^p1$ with
$x\in\g_\k$, and set $U_\xi(\g_\k):=U(\g_k)/I_\xi$. The algebra
$U_\xi(\g_\k)$ is called the {\it reduced enveloping algebra} of
$\g_\k$ associated to $\xi$. The preceding remarks imply that
$\dim_\k U_\xi(\g_\k)=p^{\dim\g}$ and $I_\xi \cap Z_p=J_\xi$, the
maximal ideal of $Z_p$ associated with $\xi$. Every irreducible
$\g_\k$-module is a module over $U_\xi(\g_\k)$ for a unique
$\xi=\xi_V\in\g_\k^*$. The linear function $\xi_V$ is called the
{\it $p$-character} of $V$; see [\cite{P95}] for more detail. By
[\cite{P95}], any irreducible $U_\xi(\g_\k)$-module has dimension
divisible by $p^{(\dim\,\g-\dim\,\z_\xi)/2},$ where
$\z_\xi=\{x\in\g_\k\,|\,\,\xi([x,\g_\k])=0\}$ is the stabiliser of
$\xi$ in $\g_\k$.
\subsection{}\label{4.3'}
For $i\in\Z$, set $\g_\k(i):=\g_A(i)\otimes_A\k$ and put
$\m_{\k}:=\m_A\otimes_A\k$. Due to our assumptions on $A$ the
elements $\bar{x}_1,\ldots, \bar{x}_r$ form a basis of the
centraliser $(\g_\k)_e$ of $e$ in $\g_\k$ and that $\m_{\k}$ is a
nilpotent subalgebra of dimension $d(e)$ in $\g_\k$. Set
$Q_{\chi,\,\k}:=U(\g_\k)\otimes_{U(\m_\k)}\k_{\chi},$ where
$\k_\chi=A_\chi\otimes_A\k\,=\,\k1_{\chi}$. Clearly, $\k1_{\chi}$ is
a $1$-dimensional $\m_\k$-module with the property that
$x(1_\chi)=\chi(x)1_\chi$ for all $x\in\m_\k$. Define
$$\widehat{U}(\g_\k,e)\,:=\,({\End}_{\g_\k}\,Q_{\chi,\,\k})^{\rm op}.$$
It follows from our discussion in (\ref{4.2}) and (\ref{4.2'}) that
$Q_{\chi,\,\k}\cong Q_{\chi,\,A}\otimes_A\k$ as modules over $\g_\k$
and $Q_{\chi,\,\k}$ is a free right module over the $k$-algebra
$$U(\g_\k,e):=U(\g_A,e)\otimes_A\k.$$ Thus we may identify $U(\g_\k,e)$
with a subalgebra of $\widehat{U}(\g_\k,e)$. Note that the algebra
$U(\g_\k,e)$ has $\k$-basis consisting of all monomials
$\bar{\Theta}_1^{i_1}\cdots \bar{\Theta}_{r}^{i_r}$ with
$(i_1,\ldots, i_r)\in\Z_+^r$, where $\bar{\Theta}_i:=\Theta_i\otimes
1\in U(\g_A,e)\otimes_A\k$. Given a polynomial $g\in A[X_1,\ldots,
X_n]$ we let $^p g$ denote the image of $g$ in the polynomial
algebra $\k[X_1,\ldots, X_n]\,=\,A[X_1,\ldots, X_n]\otimes_A\k$.
Since all polynomials $F_{ij}$ are in $A[X_1,\ldots, X_r]$, it
follows from the relations (\ref{relations}) that
\begin{eqnarray}\label{relations'}
\qquad\
[\bar{\Theta}_i,\bar{\Theta}_j]\,=\,{^p\!}F_{ij}(\bar{\Theta}_1,\ldots,\bar{\Theta}_r)\qquad\quad\
\ (1\le i<j\le r).
\end{eqnarray}
\begin{lemma}\label{L1}
The algebra $U(\g_\k,e)$ is generated by the elements
$\bar{\Theta}_1,\ldots,\bar{\Theta}_r$ subject to the relations
(\ref{relations'}).
\end{lemma}
\begin{proof}
We argue as in the proof of [\cite{P07'}, Lemma~4.1]. Let ${\mathbf
I}$ be the two-sided ideal of the free associative algebra
$\k\langle X_1,\ldots, X_r\rangle$ generated by all
$[X_i,X_j]-{^p\!}F_{ij}(X_1,\ldots,X_r)$ with $1\le i<j\le r$. Let
$\bar{X}_i$ denote the image of $X_i$ of in the factor-algebra
${\mathcal U}:=\k\langle X_1,\ldots, X_r\rangle/\mathbf{I}$. There
is a natural algebra epimorphism $\psi\colon\,{\mathcal
U}\twoheadrightarrow U(\g_\k,e)$ sending $\bar{X}_i$ to
$\bar{\Theta}_i$ for all $i$. For $k\in\Z_+$ let ${\mathcal U}_k$
denote the $\k$-span of all products $\bar{X}_{j_1}\cdots
\bar{X}_{j_m}$ with $\sum_{t=1}^m\, (n_{j_t}+2)\le k$ and let
${\mathcal U}'$ be the $\k$-span of all monomials
$\bar{X}_1^{i_1}\cdots \bar{X}_{r}^{i_r}$ with $(i_1,\ldots,
i_r)\in\Z_+^r$.  Double induction on $k$ and $m$ (upward on $k$ and
downward on $m$) based on the relations (\ref{relations'}) shows
that ${\mathcal U}'=\mathcal U$. Since the monomials
$\bar{\Theta}_1^{i_1}\cdots \bar{\Theta}_{r}^{i_r}$ with
$(i_1,\ldots, i_r)\in\Z_+^r$ are linearly independent over $\k$, we
obtain ${\mathcal U}\cong U(\g_\k,e)$, as required.
\end{proof}
Given an associative algebra $\Lambda$ we set $\Lambda^{\rm
ab}:=\Lambda/\Lambda\cdot[\Lambda,\Lambda],$ where
$\Lambda\cdot[\Lambda,\Lambda]$ is the (two-sided) ideal of
$\Lambda$ generated by all commutators $[a,b]=ab-ba$ with
$a,b\in\Lambda$. It is immediate from [\cite{P07'}, Lemma~4.1] that
$U(\g,e)^{\rm ab}$ is isomorphic to the quotient of the polynomial
algebra  $\mathbb{C}[X_1,\ldots,X_r]$ by its ideal generated by all
polynomials $F_{ij}$ with $1\le i<j\le r)$. Given a subfield $K$ of
$\mathbb C$ containing $A$ we denote by $\mathcal E(K)$ the set of
all common zeros of the polynomials $F_{ij}$ in the affine space
$\mathbb{A}^r_K$. Clearly, the $A$-defined Zariski closed set
$\mathcal E(\mathbb C)$ parametrises the $1$-dimensional
representations of the algebra $U(\g,e)$. Let $\mathcal{E}(\k)$
denote the set of all common zeros of the polynomials $^p\!F_{ij}$
in $\mathbb{A}^r_\k$. By Lemma~\ref{L1}, the set $\mathcal{E}(\k)$
parametrises the $1$-dimensional representations of the algebra
$U(\g_\k,e)$. This has the following consequence:
\begin{corollary}\label{1-dim rep}
If the algebras $U(\g_\k,e)$, where $\k=\overline{\mathbb F}_p$,
afford $1$-dimensional representations for infinitely many
$p\in\pi(A)$, then the finite $W$-algebra $U(\g,e)$ has a
$1$-dimensional representation.
\end{corollary}
\begin{proof}
Suppose for a contradiction $U(\g,e)$ has no $1$-dimensional
representations. Then $\mathcal{E}(\overline{\mathbb
Q})=\varnothing$, where $\overline{\mathbb Q}$ denotes the algebraic
closure of $\mathbb Q$ in $\mathbb C$. Since $\overline{\mathbb Q}$
is algebraically closed, there exists a finite Galois extension $K$
of $\mathbb Q$ and polynomials $g_{ij}\in K[X_1,\ldots, X_r]$ such
that $\sum_{i,j}\,g_{ij}F_{ij}=1$. Let ${\mathcal O}_K$ denote the
ring of algebraic integers of $K$. Rescaling the coefficients of the
$g_{ij}$'s if necessary, we can find  $h_{ij}\in {\mathcal
O}_K[X_1,\ldots, X_r]$ such that
$\sum_{i,j}\,h_{ij}F_{ij}=\tilde{n}$ for some positive integer
$\tilde{n}$. For each $p\in \pi(A)$ choose $\mathfrak{P}\in{\rm
Spec}\, \mathcal{O}_K$ with $\mathfrak{P}\cap \Z=p\Z$. Since
$\mathcal{O}_K$ is a Dedekind ring, $\mathcal{O}/\mathfrak{P}\cong
\mathbb{F}_q$  for some $p$-power $q$. Let
$$\varphi\colon\,\mathcal{O}[X_1,\ldots,X_r]\longrightarrow\, ({\mathcal
O}_K/\mathfrak{P})[X_1,\ldots,X_r]\hookrightarrow\,
\k[X_1,\ldots,X_r]$$ denote the homomorphism of polynomial algebras
induced by inclusion ${\mathbb F}_q\hookrightarrow\k$. Note that
$\varphi(F_{ij})=\,^p\!F_{ij}$ and $\varphi(\widetilde{n})$ is just
the residue of $\tilde{n}$ modulo $p$. As $\tilde{n}$ has finitely
many prime divisors, we derive that the ideal of $\k[X_1,\ldots,
X_r]$ generated by the $^p\!F_{ij}$'s coincides with $\k[X_1,\ldots,
X_r]$ for almost all $p\in\pi(A)$. As $\mathcal{E}(\k)=\varnothing$
for all such $p$, this implies that the algebra $U(\g_\k,e)$ has no
$1$-dimensional representations for almost all $p\in\pi(A)$. Since
this contradicts our assumption, the corollary follows.
\end{proof}
\subsection{}\label{4.0}
Let $\g_A^*$ be the $A$-module dual to $\g_A$, so that
$\g^*=\g_A^*\otimes_A\mathbb C$ and $\g_\k^*=\g_A^*\otimes_A\k$. Let
$\m_A^\perp$ denote the set of all linear functions on $\g_A$
vanishing on $\m_A$, a free $A$-submodule and a direct summand of
$\g_A^*$ (by our assumptions on $A$). Note that
$\m_A^\perp\otimes_A\mathbb C$ and $\m_A\otimes_A\k$ identify
naturally with with the annihilators
$\m^\perp:=\{f\in\g\,|\,\,f(\m)=0\}$ and
$\m_\k^\perp:=\{f\in\g^*_\k\,|\,\,f(\m_\k)=0\}$, respectively.

For $\eta\in\chi+\m_\k^\perp$ we set
$Q_{\chi}^{\eta}:=Q_{\chi,\,\k}/I_\eta Q_{\chi,\,\k}$, where
$I_\eta$ is the ideal of $U(\g_\k)$ generated by all
$x^p-x^{[p]}-\eta(x)^p1$ with $x\in\g_\k$. Evidently,
$Q_{\chi}^{\eta}$ is a $\g_\k$-module with $p$-character $\chi$.
Note that $Q_\chi^\chi\,=\,Q_\chi^{[p]}$ in the notation of
[\cite{P07'}, 4.3]. Each $\g_\k$-endomorphism $\Theta_i\otimes 1$ of
$Q_{\chi,\,\k}= Q_{\chi,\,A}\otimes_{A}\k$ preserves the submodule
$I_\eta Q_{\chi,\,\k}$, hence induces a $\g_\k$-endomorphism of
$Q_{\chi}^{\eta}$. To ease notation we call this endomorphism
$\theta_i$. Let $U_\eta(\g_\k,e)$ denote the algebra
$\big(\!\End_{\g_\k}\,Q_{\chi}^{\eta}\big)^{\rm op}$. Since the
restriction of $\eta$ to $\m_\k$ coincides with that of $\chi$, the
ideal of $U(\m_\k)$ generated by all $x-\eta(x)$ with $x\in\m_\k$
equals $N_{\chi,\,\k}=N_{\chi,A}\otimes_A\k$ and $\k_\chi=\k_\eta$
as $\m_\k$-modules.

In what follows we require a slight generalisation of [\cite{P07'},
Prop.~4.1].
\begin{lemma}\label{Q-eta}
The following are true:
\begin{itemize}
\item[(i)] $Q_{\chi}^{\eta}\,\cong\,
U_{\eta}(\g_\k)\otimes_{U_{\eta}(\m_\k)}\k_\chi$ as $\g_\k$-modules;

\smallskip

\item[(ii)] $U_{\eta}(\g_\k,e)\,\cong\,\big(U_{\eta}(\g_\k)/
U_{\eta}(\g_\k)N_{\chi,\,\k}\big)^{\ad\m_\k}$;

\smallskip
\item[(iii)] $Q_{\chi}^{\eta}$ is a projective generator for
$U_{\eta}(\g_\k)$ and
$U_{\eta}(\g_\k)\,\cong\,\Mat_{p^{d(e)}}\big(U_{\eta}(\g_\k,e)\big)$;

\smallskip

\item[(iv)] the monomials $\theta_1^{i_1}\cdots\theta_r^{i_r}$ with
$0\le i_k\le p-1$ form a $\k$-basis of $U_{\eta}(\g_\k,e)$.
\end{itemize}
\end{lemma}
\begin{proof}
Let $\bar{1}_\chi$ be the image of $1_\chi\in Q_{\chi,\,\k}$ in
$Q_\chi^\eta$. By the universality property of induced modules the
is a surjection $\tilde{\alpha}\colon\,
Q_{\chi,\,\k}=U(\g_\k)\otimes_{U(\m_\k)}\k_\chi\twoheadrightarrow\,
U_{\eta}(\g_\k)\otimes_{U_\eta(\m_\k)}\k_\chi$. As $I_\eta
Q_{\chi,\,\k}\subseteq\ker\,\tilde{\alpha}$, it gives rise to an
epimorphism $\alpha\colon\,Q_{\chi}^\eta\twoheadrightarrow\,
U_\eta(\g_\k)\otimes_{U_\eta(\m_\k)}\k_\chi$. On the other hand,
$Q_\chi^\eta$ is generated by its $1$-dimensional
$U_\eta(\m_\k)$-submodule $\k\bar{1}_\chi=\k_\chi$. The universality
property of induced $U_\eta(\g_\k)$-modules now shows that there is
a surjection
$\alpha'\colon\,U_\eta(\g_\k)\otimes_{U_\eta(\m_\k)}\k_\chi\twoheadrightarrow
\, Q_\chi^\eta$. But then $\alpha$ is an isomorphism by dimension
reasons, proving (i). Part~(ii) is an immediate consequence of
part~(i); see [\cite{P02}, p.~10] for more detail.

Suppose $\m_\k\cap\z_\eta$ contains a nonzero element, say $y$, and
write $y=\sum_{i\le -1}\,y_i$ with $y_i\in\g_\k(i)$. Let $d\in\Z$ be
such that $y_d\ne 0$ and $y_i=0$ for $i>d$. Since
$\eta\in\chi+\m_\k^\perp$, we can write $\eta=(e+a\,,\,\cdot\,)$ for
some $a\in\sum_{i\le 1}\,\g_\k(i)$. As $\z_\eta=(\g_\k)_{e+a}$ and
$\z_\chi=(\g_\k)_e$, our choice of $d$ forces
$y_d\in\m_\k\cap\z_\chi$. Since $(\g_\k)_e\subset\sum_{i\ge
0}\,\g_\k(i)$, this is impossible. So $\m_\k\cap\z_\eta=0$, implying
that $\m_\k$ is an $\eta$-admissible subalgebra of dimension $d(e)$
in $\g_\k$; see [\cite{P02}, 2.3 and 2.6]. Part~(iii) now follows
from [\cite{P02}, Thm.~2.3].

By (i) and (ii), the Kazhdan filtration of the module
$Q_{\chi,\,\k}$ indices that on the algebra
$U_\eta(\g_\k,e)=(Q_{\chi,\,\k}/I_\eta Q_{\chi,\,\k})^{\ad\m_\k}$.
Repeating verbatim the argument from the proof of [\cite{P02},
Thm.~3.4(i)] one obtains that the monomials
$\theta_1^{i_1}\cdots\theta_r^{i_r}$ with $0\le i_k\le p-1$ are
linearly independent in $U_\eta(\g_\k,e)$. Since
$\dim\,U_\eta(\g_\k,e)=p^r$ by part~(iii), these monomials form a
basis of $U_\eta(\g_\k,e)$.
\end{proof}
\subsection{}\label{4.4} Recall from (\ref{4.1}) the $A$-basis $\{x_1,\ldots, x_r, x_{r+1},\ldots,x_m\}$
of $\p_A$. Set
$$X_i=\left\{
\begin{array}{ll}
z_i&\mbox{if $\ 1\le i\le s$},\\
x_{r-s+i}&\mbox{if $\ s+1\le i\le m-r+s$}.
\end{array}\right .$$
For ${\bf a}\in \Z_+^{d(e)}$, put $X^{\bf a}:=X_1^{a_1}\cdots
X_{d(e)}^{a_{d(e)}}$ and $\bar{X}^{\bf a}:=\bar{X}_1^{a_1}\cdots
\bar{X}_{d(e)}^{a_{d(e)}}$, elements of $U(\g_A)$ and $U(\g_\k)$,
respectively. By [\cite{P07'}, Lemma~4.2(i)], the monomials $X^{\bf
a}\otimes 1_\chi$ with ${\bf a}\in\Z^{d(e)}$ form a free basis of
the right $U(\g_A,e)$-module $Q_{\chi,A}$.
\begin{lemma}\label{lem4} Let $\bar{1}_\chi$ be the image of  $1_\chi\in Q_{\chi,\,\k}$ in
$Q_\chi^\eta$. For every $\eta\in\chi+\m_\k^\perp$ the right
$U_\eta(\g_\k,e)$-module $Q_\chi^\eta$ is free with basis
$\big\{\bar{X}^{\bf a}\otimes \bar{1}_{\chi}\,|\,\, 0\le a_i\le
p-1\big\}$.
\end{lemma}
\begin{proof}
The Kazhdan filtration of the $U(\g_\k)$-module $Q_{\chi,\,\k}$
induces that on the factor-module $Q_\chi^\eta=Q_{\chi,\,\k}/I_\eta
Q_{\chi,\,\k}$. For $k\ge 0$ denote by $(Q_\chi^\eta)_k$ the $k$th
component of the Kazhdan filtration of $Q_\chi^\eta$. Call a tuple
${\bf a}\in\Z_+^l$ {\it admissible} if $a_i\le p-1$ for all $i$. By
Lemma~\ref{Q-eta}(iv), the monomials $\theta^{\bf
a}:=\theta_1^{a_1}\cdots\theta_r^{a_r}$, where ${\bf a}$ runs
through the admissible tuples in $\Z_+^r$, form a $\k$-basis of
$U_\eta(\g_\k,e)$. Using (\ref{lam}) and induction on the Kazhdan
degree $k=\sum_{i=1}^ra_i(n_i+2)$ of $\Theta^{\bf a}$ it is easy to
observe that
$$\theta^{\bf a}(\bar{1}_\chi)\,\equiv\, \bar{x}_1^{a_1}\cdots
\bar{x}_r^{a_r}\otimes \bar{1}_\chi+\ \sum_{|({\bf i},{\bf
j})|_e=k,\ |{\bf i}|+|{\bf j}|>|\bf a|}\,\gamma_{{\bf i},{\bf
j}}\,\bar{x}^{\bf i}\bar{z}^{\bf j}\otimes 1_\chi\quad\
(\mathrm{mod}\ (Q_\chi^\eta)_{k-1}$$ for some $\gamma_{{\bf i},{\bf
j}}\in \k$. This relation in conjunction with double induction on
$|({\bf i},{\bf j})|_e$ and $|{\bf i}|+|{\bf j}|$ (upward on $|({\bf
i},{\bf j})|_e$ and downward on $|{\bf i}|+|{\bf j}|$) yields that
every $\bar{x}^{\bf i}\bar{z}^{\bf j}\otimes \bar{1}_\chi$ belongs
to the $\k$-submodule of $Q_{\chi}^\eta$ spanned by the vectors
$\bar{X}^{\bf a}\theta^{\bf b}(\bar{1}_\chi)$ with admissible ${\bf
a}\in\Z_+^{d(e)}$ and ${\bf b}\in\Z_+^r$. Since $\dim_\k\,
Q_\chi^\eta=p^{d(e)+r}$ by Lemma~\ref{Q-eta}(i), these vectors are
linearly independent. The result follows.
\end{proof}
Let $\a_k$ be the $\k$-span of $\bar{X}_1,\ldots, \bar{X}_{d(e)}$ in
$\g_\k$ and put $\widetilde{\a}_\k:=\a_\k\oplus\z_\chi$. By our
assumptions on $x_{r+1},\ldots, x_m$ in (\ref{4.1}) and the
inclusion $\ker\ad f\subset\bigoplus_{i\le 0}\,\g_\k(i)$, we have
that
\begin{equation}\label{orthog}
\a_\k\,=\,\{x\in\widetilde{\a}_\k\,|\,\,\,(x,\ker\ad f)=0\}.
\end{equation}
The bilinear form $(\,\cdot\,,\,\cdot\,)$ allows us to identify the
symmetric algebra $S(\widetilde{\a}_\k)$ with the coordinate ring
$\k[\chi+\m_\k^\perp]$. Given a subspace $V$ in $\g_\k$ we denote by
$Z_p(V)$ the subalgebra of the $p$-centre $Z(\g_k)$ generated by all
$x^p-x^{[p]}$ with $x\in V$. Clearly, $Z_p(V)$ is isomorphic to a
polynomial algebra in $\dim_\k V$ variables. Let $\rho_\k$ denote
the representation of $U(\g_\k)$ in $\End_\k Q_{\chi,\,\k}$.

Our next result is, in a sense, analogous to Velkamp's theorem
[\cite{Ve}] on the structure of the centre of $U(\g_\k)$. Similarity
becomes apparent when one takes for $e$ a regular nilpotent element
in $\g_\k$ and observes that in this special case $U(\g_\k,e)$
identifies with the invariant algebra $U(\g_\k)^{G_\k}$.
\begin{theorem}\label{U-hat}
The following hold for any nilpotent element $e\in\g_\k$:
\begin{itemize}
\item[(i)]
the algebra $\widehat{U}(\g_\k,e)$ is generated by its subalgebras
$U(\g_\k,e)$ and $\rho_\k(Z_p)$;
\smallskip
\item[(ii)]
$\rho_\k(Z_p)\cong  Z_p(\widetilde{\a}_\k)$ and
$\widehat{U}(\g_\k,e)$ is a free $\rho_\k(Z_p)$-module of rank
$p^{r}$;
\smallskip
\item[(iii)]
$\widehat{U}(\g_\k,e)\cong U(\g_\k,e)\otimes_\k Z_p(\a_\k)$ as
$\k$-algebras.
\end{itemize}
\end{theorem}
\begin{proof}
(a) First note that $Z_p(\g_\k)\cong Z_p(\m_\k)\otimes_\k
Z_p(\widetilde{\a}_\k)$ as algebras, and $Z(\m_\k)\cap\ker\,\rho_\k$
is an ideal of codimension $1$ in $Z_p(\m_\k)$. Hence
$\rho_\k(Z_p)=\rho_\k(Z_p(\widetilde{\a}_\k))$. As the monomials
$\bar{x}^{\bf i}\bar{z}^{\bf j}\otimes 1_\chi$ with $({\bf i},{\bf
j})\in\Z_+^m\times \Z_+^s$ form a basis of $Q_{\chi,\,\k}$ and
$Z_p(\widetilde{\a}_\k)$ is polynomial algebra in
$\bar{z}_i^p-\bar{z}_i^{[p]}$ $\,(1\le i\le s)$ and
$\bar{x}_j^p-\bar{x}_j^{[p]}$ $\,(1\le j\le m)$, we have that
$Z_p(\widetilde{\a}_\k)\cap\ker\,\rho_\k=\{0\}$. It follows that
$\rho_\k(Z_p)\cong Z_p(\widetilde{\a}_\k)\cong\k[\chi+\m_\k^\perp]$
as $\k$-algebras.

\smallskip

\noindent (b) Denote by ${\bf I}_l$ the set of all admissible tuples
in $\Z_+^{l}$ and let ${\bf e}_i$ denote the tuple in ${\bf I}_l$
whose only nonzero component equals $1$ and occupies the $i$-th
position. As an immediate consequence of (\ref{lam}), for $1\le k\le
r$ we have that
\begin{eqnarray}\label{Theta-p}
\bar{\Theta}_k^p(1_\chi)-\Big(\bar{x}_k^p+\sum_{|({\bf i},{\bf
j})|_e= n_k+2}\,\mu_{{\bf i},\,{\bf j}}^k\,\bar{x}^{p\,{\bf
i}}\bar{z}^{p\,{\bf j}}\Big)\otimes 1_\chi
\in\big(Q_{\chi,\,\k}\big)_{p(n_k+2)-1}
\end{eqnarray}
for some $\mu_{{\bf i},\,{\bf j}}^k\in{\mathbb F}_p$. Also,
$\gr(\bar{z}_i^p-\bar{z}_i^{[p]})=\gr(\bar{z}_i)^p$ and
$\gr(\bar{x}_i^p-\bar{x}_j^{[p]})=\gr(\bar{x}_j)^p$ for all $1\le
i\le s$ and $1\le j\le m$. On the other hand, [\cite{P07'},
Lemma~4.2(i)] implies that the vectors $\bar{X}^{\bf a}\otimes
1_\chi$ with ${\bf a}\in\Z_+^{d(e)}$ form a free basis of the right
$U(\g_\k,e)$-module $Q_{\chi,\,\k}$. As $Q_{\chi,\,\k}$ is a Kazhdan
filtered $U(\g_\k)$-module, straightforward induction on filtration
degree, based on (\ref{Theta-p}), shows that $Q_{\chi,\,\k}$ is
generated as a $Z_p(\widetilde{\a}_\k)$-module by the set
$\{\bar{X}^{\bf i}\bar{\Theta}^{\bf j}\otimes 1_\chi\,|\,\,{\bf
i}\in {\bf I}_{d(e)},\,{\bf j}\in{\bf I}_{r}\}$.

Let $h$ be an arbitrary element of $\widehat{U}(\g_\k,e)$. Then
$h(1_\chi)=\sum_{({\bf i}, {\bf j})\in\,{\bf I}_{d(e)}\times{\bf
I}_r}f_{{\bf i},{\bf j}}\bar{X}^{\bf i}\bar{\Theta}^{\bf j}(1_\chi)$
for some $f_{{\bf i},{\bf j}}\in Z_p(\widetilde{\a}_\k)$. For every
$\xi\in\chi+\m_\k^\perp$ the image of $f_{{\bf i}, {\bf j}}$ in
$U_\xi(\g_\k)$ is a scalar which shall be denoted by $\xi({\bf i},
{\bf j})$. Suppose $f_{{\bf a}, {\bf b}}\ne 0$ for a nonzero ${\bf
a}\in{\bf I}_{d(e)}$ and some ${\bf b}\in{\bf I}_r$. Then there
exists $\eta\in\chi+\m_\k^\perp$ such that $\eta({\bf a}, {\bf
b})\ne 0$. Let $h(\eta)$ be the image of $h$ in
$U_\eta(\g_\k,e)=\big(\End_{\g_\k} Q_{\chi}^\eta\big)^{\text{op}}$.
Lemma~\ref{Q-eta}(iv) implies that $h(\eta)(\bar{1}_\chi)$ is a
$\k$-linear combination of $\theta^{\bf i}(\bar{1}_\chi)$ with ${\bf
i}\in{\bf I}_r$. By Lemma~\ref{lem4}, the set $\{\bar{X}^{\bf
i}\otimes\bar{1}_\chi\,|\,\,{\bf i}\in{\bf I}_{d(e)}\}$ is a free
basis of the right $U_\eta(\g_\k,e)$-module $Q_\chi^\eta$. Since
$\eta({\bf a},{\bf b})\ne 0$ and $\theta^{\bf i}$ is the image of
$\bar{\Theta}^{\bf i}$ in $U_\eta(\g_\k,e)$, it is now evident that
$h(\eta)(\bar{1}_\chi)$ cannot be a $\k$-linear combination of
$\theta^{\bf i}(\bar{1}_\chi)$ with ${\bf i}\in{\bf I}_r$. This
contradiction shows that $f_{{\bf i},{\bf j}}=0$ unless ${\bf
i}={\bf 0}$. As a consequence, the set $\{\bar{\Theta}^{\bf
i}\,|\,\,{\bf i}\in{\bf I}_r\}$ generates $\widehat{U}(\g_\k,e)$ as
a $Z_p(\widetilde{\a}_\k)$-module. Specialising at a suitable
$\eta\in\chi+\m_\k^\perp$ and applying Lemma~\ref{Q-eta}(iv) one
more time we deduce that the set $\{\bar{\Theta}^{\bf i}\,|\,\,{\bf
i}\in{\bf I}_r\}$ is a free basis of the
$Z_p(\widetilde{\a}_\k)$-module $\widehat{U}(\g_\k,e)$.

\smallskip

\noindent (c) Our next goal is to show that
$\widehat{U}(\g_\k,e)=U(\g_\k,e)\cdot Z_p(\a_\k)$. Every
$\g_\k$-endomorphism of $Q_{\chi,\,\k}$ is uniquely determined by
its value at $1_\chi$. For  a nonzero $u\in \widehat{U}(\g_\k,e)$
write $u(1_\chi)\,=\,\sum_{|({\bf i},{\bf j})|_e\le n}\lambda_{{\bf
i},{\bf j}}\,\bar{x}^{\bf i}\bar{z}^{\bf j}\otimes 1_\chi$, where
$n=n(u)$ and $\lambda_{{\bf i},{\bf j}}\ne 0$ for at least one
$({\bf i},{\bf j})$ with $|({\bf i},{\bf j})|_e=n$. For $k\in \Z_+$
put $\Lambda^k(u):=\{({\bf i},{\bf j})\in \Z_+^m\times
\Z_+^s\,|\,\,\lambda_{{\bf i},{\bf j}}\ne 0\ \,\& \ \, |({\bf
i},{\bf j})|_e=k\}$ and denote by $\Lambda^{\rm max}(u)$ the set of
all $({\bf a},{\bf b})\in\Lambda^n(u)$ for which the quantity
$n-|{\bf a}|-|{\bf b}|$ assumes its maximum value. This maximum
value will be denoted by $n'=n'(u)$. For each $({\bf a},{\bf
b})\in\Lambda^{\rm max}$ we have that
$$|({\bf a},{\bf b})|_e-|{\bf a}|-|{\bf b}|\,=\,\textstyle{\sum}_{i=1}^m\,
(n_i+2)a_i+\textstyle{\sum}_{i=1}^s\,b_i-|{\bf a}|-|{\bf b}|\ge 0.$$
Consequently, $n(u),\,n'(u)\in \Z_+$ and $n(u)\ge n'(u)$.

Put $\Omega:=\{(a,b)\in\Z_+^2\,|\,\,a\ge b\}$. By the preceding
remark, $(n(u),n'(u))\in\Omega$ for all nonzero
$u\in\widehat{U}(\g_\k,e)$. It is immediate from (\ref{lam}) and our
discussion in part~(b) that $\Lambda^{\rm
max}(\bar{\Theta}_k)=\{({\bf e}_k,{\bf 0})\}$, $\,\Lambda^{\rm
max}\big(\rho_\k(\bar{x}^p_i-\bar{x}_i^{[p]})\big)=\{(p\,{\bf
e}_i,{\bf 0})\}$ for $1\le i\le m$, and $\Lambda^{\rm
max}\big(\rho_\k(\bar{z}^p_j-\bar{z}_j^{[p]})\big)=\{({\bf
0},p\,{\bf e}_j)\}$ for $1\le j\le s$. Since $Q_{\chi,\,\k}$ is a
Kazhdan filtered $U(\g_\k)$-module, this implies that
$$\Lambda^{\rm max}\Big(\textstyle{\prod}_{i=1}^m\,\rho_\k(\bar{x}_i^p-\bar{x}_i^{[p]})^{a_i}\cdot
\textstyle{\prod}_{i=1}^s\,\rho_\k(\bar{z}_i^p-\bar{z}_i^{[p]})^{b_i}\cdot\bar{\Theta}^{\bf
c}\Big) \,=\, \big\{\big(p\,{\bf a}+\textstyle{\sum}_{i=1}^r\,
c_i{\bf e}_i,p\,{\bf b}\big)\big\}$$ for all $({\bf a},{\bf
b})\in\Z_+^m\times\Z_+^s$ and all ${\bf c}\in{\bf I}_r$. Since
$\widehat{U}(\g_\k,e)$ is generated as a $Z_p(\a_\k)$-module by the
set $\{\bar{\Theta}^{\bf i}\,|\,\,{\bf i}\in{\bf I}_r\}$, it follows
that for every $u\in \widehat{U}(\g_\k,e)$ with $(n(u),n'(u))=(d,l)$
there exists a $\k$-linear combination $u'$ of the endomorphisms
$$u({\bf a},{\bf b},{\bf
c}):=\textstyle{\prod}_{i=1}^m\,\rho_\k(\bar{x}_i^p-\bar{x}_i^{[p]})^{a_i}\cdot
\textstyle{\prod}_{i=1}^s\,\rho_\k(\bar{z}_i^p-\bar{z}_i^{[p]})^{b_i}\cdot\bar{\Theta}^{\bf
c},\quad\ ({\bf a},{\bf b})\in\Z_+^m\times \Z_+^s, \ {\bf c}\in{\bf
I}_r,$$ with $\Lambda^{\rm max}\big(u({\bf a},{\bf b},{\bf
c})\big)\subseteq \Lambda^{\rm max}(u)$ such that either $n(u-u')<d$
or $n(u-u')=d$ and $n'(u-u')< l$.

Order the tuples in $\Omega$ lexicographically and assume that $u\in
U(\g_\k,e)\cdot Z_p(\a_\k)$ for all nonzero
$u\in\widehat{U}(\g_\k,e)$ with $(n(u),n'(u))\prec(d,l)$ (when
$(n(u),n'(u))=(0,0)$ this is a valid assumption). Now let $u\in
\widehat{U}(\g_\k,e)$ be such that $(n(u),n'(u))=(d,l)$. By the
preceding remark, there exists $u'=\sum_{({\bf a},{\bf b},{\bf
c})}\,\lambda_{{\bf a},{\bf b},{\bf c}}\,u({\bf a},{\bf b},{\bf c})$
with $\Lambda^{\rm max}\big(u({\bf a},{\bf b},{\bf c})\big)\subseteq
\Lambda^{\rm max}(u)$ for all $({\bf a},{\bf b},{\bf c})$ with
$\lambda_{{\bf a},{\bf b},{\bf c}}\ne 0$ such that
$(n(u-u'),n'(u-u'))\prec (d,l)$. Set
$$v({\bf a},{\bf b},{\bf
c})\,:=\,u((0,\ldots, 0,a_{r+1},\ldots,a_m),{\bf b},{\bf
0})\cdot\textstyle{\prod}_{i=1}^r\,\bar{\Theta}^{pa_i}\cdot\bar{\Theta}^{\bf
c}.$$ Using (\ref{Theta-p}) it is easy to observe that $\Lambda^{\rm
max}\big(u({\bf a},{\bf b},{\bf c})\big)=\Lambda^{\rm
max}\big(v({\bf a},{\bf b},{\bf c})\big)$ and $$\big(n(u({\bf
a},{\bf b},{\bf c})-v({\bf a},{\bf b},{\bf c})), n'(u({\bf a},{\bf
b},{\bf c})-v({\bf a},{\bf b},{\bf c}))\big)\prec \big(n(u({\bf
a},{\bf b},{\bf c})),n'(u({\bf a},{\bf b},{\bf c})\big).$$ We now
put $u'':=\sum_{({\bf a},{\bf b},{\bf c})}\,\lambda_{{\bf a},{\bf
b},{\bf c}}\,v({\bf a},{\bf b},{\bf c})$, an element of
$U(\g_\k,e)\cdot Z_p(\a_\k)$. Because $(n(u-u''),n'(u-u''))\prec
(n(u),n'(u))$, the equality $\widehat{U}(\g_\k,e)=U(\g_\k,e)\cdot
Z_p(\a_\k)$ follows by induction on the length of $(d,l)$ in the
linearly ordered set $(\Omega,\prec)$.

\smallskip

\noindent (d) It remains to show that $\widehat{U}(\g_\k,e)\cong
U(\g_\k,e)\otimes_\k Z_p(\a_\k)$. We have already mentioned that the
vectors $\bar{X}^{\bf a}\otimes 1_\chi$ with ${\bf a}\in\Z_+^{d(e)}$
form a free basis of the right $U(\g_\k,e)$-module $Q_{\chi,\,\k}$.
Since $\bar{X}_i^p$ and $\bar{X}_i^p-\bar{X}_i^{[p]}$ have the same
Kazhdan degree in $U(\g_\k)$ and $Q_{\chi,\,\k}$ is a Kazhdan
filtered $U(\g_\k)$-module, it follows that the vectors
$$\Big\{\textstyle{\prod}_{i=1}^{d(e)}\,\rho_\k(\bar{X}_i^p-\bar{X}_i^{[p]})^{a_i}
\cdot\bar{\Theta}^{\bf c}\otimes 1_\chi\,|\,\, a_i\in\Z_+,\,\,{\bf
c}\in \Z_+^{r}\Big\}$$ are linearly independent. This implies that
$\widehat{U}(\g_\k,e)\cong U(\g_\k,e)\otimes_\k Z_p(\a_\k)$ as
algebras, completing the proof.
\end{proof}
\subsection{}\label{4.5}
As an immediate consequence of Theorem~\ref{U-hat} we obtain:
\begin{corollary}\label{ab}
$\widehat{U}(\g_\k,e)^{\rm ab}\cong U(\g_\k,e)^{\rm ab}\otimes
Z_p(\a_\k)$ as $\k$-algebras.
\end{corollary}
\begin{proof}
If $C$ is an associative commutative $\k$-algebra, then for any
associative $\k$-algebra $\Lambda$ we have that $$[\Lambda\otimes_\k
C,\Lambda\otimes_\k C]\cdot(\Lambda\otimes_\k
C)\,=\,\big([\Lambda,\Lambda]\otimes_\k
C\big)\cdot(\Lambda\otimes_\k
C)\,=\,[\Lambda,\Lambda]\cdot\Lambda\otimes_\k C.$$ Hence
$(\Lambda\otimes_\k C)^{\rm ab}\,\cong\,\Lambda^{\rm ab}\otimes_\k
C$ as $\k$-algebras. In view of Theorem~\ref{U-hat} the corollary
obtains by setting $\Lambda:=U(\g_\k,e)$ and $C:=Z_p(\a_\k)$.
\end{proof}
We are now in a position to prove the main result of this section.
\begin{theorem}\label{U-U}
If the finite $W$-algebra $U(\g,e)$ affords a $1$-dimensional
representation, then for $p\gg 0$ the reduced enveloping algebra
$U_\chi(\g_\k)$ admits irreducible representations of dimension
$p^{d(e)}$.
\end{theorem}
\begin{proof}
(a) Suppose $U(\g,e)$ affords a $1$-dimensional representation. Then
${\mathcal E}(\mathbb C)\ne \varnothing$. Since the affine variety
${\mathcal E}({\mathbb C})={\rm Specm}\,U(\g,e)^{\rm ab}$ is defined
over $\mathbb Q$ and $\overline{\mathbb Q}$ is algebraically closed,
it follows that ${\mathcal E}(\overline{\mathbb Q})\ne \varnothing$.
Hence ${\mathcal E}(K)\ne \varnothing$ for some finite Galois
extension $K$ of $\mathbb Q$. It follows that there exists
$d\in\mathbb N$ such that $\mathcal E$ has a point with coordinates
in ${\mathcal O}_K[d^{-1}]$, where ${\mathcal O}_K$ stands for the
ring of algebraic integers of $K$. If $p\nmid d$, then there is
${\mathfrak P}\in{\rm Spec}\,{\mathcal O}_K[d^{-1}]$ such that
${\mathcal O}_K[d^{-1}]/{\mathfrak P}\cong {\mathbb F}_q$, where $q$
is a power of $p$. Embedding ${\mathbb F}_q$ into
$\k=\overline{\mathbb F}_p$ we see that ${\mathcal E}(\k)\ne
\varnothing$ for all such $p$. In view of Lemma~\ref{L1} this
implies that $U(\g_\k,e)$ affords $1$-dimensional representations
for all primes $p$ satisfying $p\nmid d$.

\smallskip

\noindent (b) Now suppose that $p\gg 0$ and $U(\g_\k,e)$ affords a
$1$-dimensional representation. Then Theorem~\ref{U-hat}(iv) yields
that the $\k$-algebra $\widehat{U}(\g_\k,e)$ affords a
$1$-dimensional representation too; we call it $\nu$.  By
Theorem~\ref{U-hat}(ii), $\rho_\k(Z_p)\cap \ker\,\nu$ is a maximal
ideal of the algebra $\rho_\k(Z_p)\cong
Z_p(\widetilde{\a}_\k)\cong\k[\chi+\m_\k^\perp]$. So there exists
$\eta\in\chi+\m_\k^\perp$ such that
$\rho_\k(x^p-x^{[p]}-\eta(x)^p)\in\ker\nu$ for all $x\in\g_\k$. Our
choice of $\eta$ ensures that the $\k$-algebra
$\widehat{U}_\eta(\g_\k,e):=\widehat{U}(\g_\k,e)\otimes_{Z_p(\widetilde{\a}_\k)}\k_\eta$
affords a $1$-dimensional representation.  On the other hand, the
canonical projection
$Q_{\chi,\,\k}\twoheadrightarrow\,Q_{\chi,\,\k}/I_\eta
Q_{\chi,\,\k}=Q_\chi^\eta$ gives rise to an algebra homomorphism
$\rho_\eta\colon\,\widehat{U}_\eta(\g_\k,e)\rightarrow\,\big(\End_{\g_\k}Q_\chi^\eta\big)^{\rm
op}=U_\eta(\g_\k,e)$. As $\dim_\k\widehat{U}_\eta(\g_\k,e)\le p^r$
by Theorem~\ref{U-hat}(ii), applying Lemma~\ref{Q-eta}(iv) yields
that $\rho_\eta$ is an algebra isomorphism. As
$U_\eta(\g_\k)\,\cong\, \Mat_{p^{d(e)}}\big(U_\eta(\g_\k,e)\big)$ by
Lemma~\ref{Q-eta}(iii), it follows that the algebra $U_\eta(\g_\k)$
has an irreducible representation of dimension $p^{d(e)}$.

\smallskip

\noindent (c) Let $\Xi$ denote the set of all $\xi\in\g_\k^*$ for
which the algebra $U_\xi(\g_\k)$ contains a two-sided ideal of
codimension $p^{2d(e)}$. It is immediate from [\cite{PS}, Lemma~2.3]
that the set $\Xi$ is Zariski closed in $\g_\k^*$. If $\xi'=({\rm
Ad}^*\,g)(\xi)$ for some $g\in G_\k$, then $U_\xi(\g_\k)\cong
U_{\xi'}(\g_\k)$ as algebras. Hence $\Xi$ is stable under the
coadjoint action of $G_\k$.

We claim that $\k^\times\cdot \xi\subset \Xi$ for all $\xi\in\Xi$.
To prove the claim we first recall that $\xi=(x,\,\cdot\,)$ for some
$x\in\g_\k$. Let $x=x_s+x_n$ be the Jordan--Chevalley decomposition
of $x$ in the restricted Lie algebra $\g_\k$ and put
$\xi_s:=(x_s,\,\cdot\,)$, $\xi_n:=(x_n,\,,\cdot\,)$, and ${\mathfrak
l}:=\z(\chi_s)$. As $p\gg 0$ and $x_s$ is semisimple, $\l$ is a Levi
subalgebra of $\g_\k$. It $t\in\k^\times$, then $tx=tx_s+tx_n$ is
the Jordan--Chevalley decomposition of $tx$. Obviously,
$\z(t\xi_s)=\mathfrak l$.

Put $d:=\frac{1}{2}(\dim_\k\g_\k-\dim_\k\mathfrak l)$. It follows
from the Kac--Weisfeiler theorem (or rather from its generalisation
due to Friedlander--Parshall) that
$U_\xi(\g_\k)\,\cong\,\Mat_{p^d}\big(U_\xi({\mathfrak l})\big)$ and
$U_{t\xi}(\g_\k)\,\cong\,\Mat_{p^d}\big(U_{t\xi}({\mathfrak
l})\big)$; see [\cite{P02}, 2.5], for example. Since $p\gg 0$, we
have a direct sum decomposition ${\mathfrak l}={\mathfrak
s}\oplus\z({\mathfrak l})$, where ${\mathfrak s}=[{\mathfrak
l},{\mathfrak l}]$, and induced tensor product decompositions
$U_\xi({\mathfrak l})\,\cong\,U_{\xi}({\mathfrak s})\otimes_\k
U_{\xi}(\z({\mathfrak l}))$ and $U_{t\xi}({\mathfrak
l})\,\cong\,U_{t\xi}({\mathfrak s})\otimes_\k U_{t\xi}(\z({\mathfrak
l}))$. As $\z({\mathfrak l})$ is a toral subalgebra of $\g_\k$, the
reduced enveloping algebra $U_\psi(\z({\mathfrak l}))$ is
commutative and semisimple for every $\psi\in \z({\mathfrak l})^*$.
From this it is immediate that $U_\xi(\z({\mathfrak
l}))\,\cong\,U_{t\xi}(\z({\mathfrak l}))$ as algebras.

Let $L$ be the Levi subgroup of $G_\k$ with $\Lie(L)=\mathfrak l$.
It acts on ${\mathfrak s}$ as restricted Lie algebra automorphisms.
Note that $\xi_{\vert{\mathfrak s}}=\xi_n$. As $x_n$ is nilpotent
and $L$ is reductive, all nonzero scalar multiples of $x_n$ are
conjugate under the adjoint action of $L$. This implies that the
algebras $U_{\xi}({\mathfrak s})$ and $U_{t\xi}({\mathfrak s})$ are
isomorphic. In view of our earlier remarks this shows that
$U_{t\xi}({\mathfrak l})\,\cong\, U_\xi({\mathfrak l})$ and
$U_{t\xi}(\g_\k)\,\cong\,U_\xi(\g_\k)$ for all $t\in\k^\times$. Our
claim is an immediate consequence of the last isomorphism.

\smallskip

\noindent (d) Since $\Xi$ is Zariski closed and $\k^\times\cdot
\xi\subset \Xi$ for all $\xi\in\Xi$, the set $\Xi$ is conical. As
$U_\eta(\g_\k)$ has a simple module of dimension $p^{d(e)}$, we have
$\eta\in\Xi$. As $\eta\in\chi+\m_\k^\perp$ we can write
$\eta=(e+y,\,\cdot\,)$ for some $y=\sum_{i\le -1}\,y_i$ with
$y_i\in\g_\k(i)$. There is a cocharacter
$\lambda\colon\,\k^\times\rightarrow\,G_\k$ such that
$(\Ad\,\lambda(t))\,x=t^{j}x$ for all $x\in\g_\k(j)$, $j\in\Z$ and
$t\in\k^\times$. For $i\le -1$, set $\eta_i:=(y_i,\,\cdot\,)$. Then
$\eta=\chi+\sum_{i\le -1}\,\eta_i$ and $({\rm
Ad}^*\,\lambda(t))\,\eta\,=\,t^2\chi+\sum_{i\le 1}\,t^i\eta_i$. As
$\Xi$ is conical and $({\rm Ad}^*\, G_\k)$-invariant, this implies
that $\chi+\sum_{i\le 1}\,t^{2-i}\, \eta_i\in\Xi$ for all
$t\in\k^\times$. Since $\Xi$ is Zariski closed, this yields
$\chi\in\Xi$.

Let $I$ be a two-sided ideal of codimension $p^{2d(e)}$ in
$U_\chi(\g_\k)$. Then $U_\chi(\g_\k)/I$ is a
$U_\chi(\g_\k)$-bimodule. Since $U_\chi(\g_\k)\otimes_\k
U_\chi(\g_\k)^{\rm op}\,\cong\, U_{(\chi,-\chi)}(\g_\k\oplus\g_\k)$
as $\k$-algebras, it is immediate from [\cite{P95}, Thm.~3.10] that
the bimodule $U_\chi(\g_\k)/I$ is irreducible. But then
$U_\chi(\g_\k)/I\,\cong\,\Mat_{p^{d(e)}}(\k)$. This shows that
$U_\chi(\g_\k)$ has a simple module of dimension $p^{d(e)}$,
completing the proof.
\end{proof}
\subsection{}
We call a representation of $U_\xi(\g_\k)$ {\it small} if it has
dimension equal to $p^{(\dim\,G_\k\cdot\,\xi)/2}$. To prove that
every reduced enveloping algebra $U_\xi(\g_\k)$ has such a
representation is a well-known open problem in the modular
representation theory of Lie algebras; see [\cite{P95}, p.~114],
[\cite{Ka}], [\cite{H}, p.~110], for example. This problem has a
positive solution for Lie algebras type $\rm A$ due to the fact that
all nilpotent elements in $\mathfrak{gl}_n$ are Richardson. This
enables one to construct small representations by inducing up
$1$-dimensional representations of appropriate parabolic
subalgebras. However, outside type $\rm A$ the problem of small
representations is wide open, and in the most interesting cases it
is impossible to obtain such representations by parabolic induction.
Our next result solves the problem of small representations for Lie
algebras of types $\rm B$, $\rm C$, $\rm D$ under the assumption
that $p\gg 0$.
\begin{corollary}\label{ABCD}
If $\g_\k$ is of type $\rm B$, $\rm C$ or $\rm D$, then the problem
of small representations for $\g_\k$ has a positive solution for
almost all primes. More precisely, if $\k=\overline{\mathbb F}_p$
and $p\gg 0$, then for every $\xi\in\g_\k^*$ the reduced enveloping
algebra $U_\xi(\g_\k)$ has a simple module of dimension
$p^{(\dim\,G_\k\cdot\,\xi)/2}$.
\end{corollary}
\begin{proof}
If $\mathfrak l$ Levi subalgebra of $\g_\k$, then ${\mathfrak
l}=[\l,\l]\oplus\mathfrak{z(l)}$ and $[{\mathfrak l},{\mathfrak l}]$
decomposes as a direct sum of ideals each of which is a simple Lie
algebra of type ${\rm A}$, ${\rm B}$, ${\rm C}$, ${\rm D}$ (one
should keep in mind here that $p\gg 0$). In view of the
Kac--Weisfeiler theorem this reduces the problem of small
representations to the case where $\xi=(\bar{n},\,\cdot\,)$ for some
nilpotent element $\bar{n}\in\g_\k$; see [\cite{P02}, 2.5] or
[\cite{H}, p.~114]. Furthermore, it can be assumed that
$\bar{n}=n\otimes 1$ for some nilpotent element $n\in\g$. By
[\cite{Lo}, Thm.~1.2.3(1)], the finite $W$-algebra $U(\g,n)$ admits
a $1$-dimensional representation (the argument in [\cite{Lo}] relies
on earlier results of McGovern on completely prime primitive ideals;
see [\cite{Mc}, Ch.~5]). Applying Theorem~\ref{U-U} we now see that
the reduced enveloping algebra $U_\xi(\g_\k)$ has a module of
dimension $p^{(\dim\,G_\k\cdot\,\xi)/2}$. This module is irreducible
thanks to [\cite{P95}, Thm.~3.10].
\end{proof}
\begin{rem} Applying successively [\cite{P07'}, 4.3],
Corollary~\ref{1-dim rep}, [\cite{P07}, Thm.~3.1(ii)], and
[\cite{Lo}, Prop.~3.4.6] one observes that if the problem of small
representations for $\g_\k$ has a positive solution for almost all
primes, then for every nilpotent orbit ${\mathcal O}\subset\g$ there
exists a completely prime primitive ideal $I$ of $U(\g)$ such that
$\mathcal{VA}(I)=\overline{\mathcal O}$ (here $\mathcal{VA}(I)$
stands for the associated variety of $I$).
\end{rem}
\begin{rem}
It seems likely that Corollary~\ref{ABCD} remains true for all
$p>2$. To relax the assumption on $p$ in the statement of
Corollary~\ref{ABCD} by the methods of this paper one would need a
more explicit presentation of $U(\g,e)$ in the spirit of
[\cite{BrK}].
\end{rem}
\section{\bf Sheets and commutative quotients of finite $W$-algebras}
\subsection{}\label{5.1}
Our main goal in this section is to estimate the number of
irreducible components of the affine variety ${\rm Specm}\,
U(\g,e)^{\rm ab}$ and determine their dimensions. Since our
arguments will rely on Corollary~\ref{ABCD}, we have to leave aside
some nilpotent orbits in Lie algebras of type ${\rm E}_7$ and ${\rm
E}_8$.

Because the field $\overline{\mathbb Q}$ is algebraically closed,
all irreducible components of ${\mathcal E}({\mathbb C})={\rm
Specm}\,U(\g,e)^{\rm ab}$ are defined over an algebraic number
field, $K$ say. Let $R$ denote the ring of algebraic integers of
$K$. For any maximal ideal $\mathfrak p$ of $R$ the residue field
$R/\mathfrak p$ is finite. Denote by $\k(\p)$ the algebraic closure
of $R/\p$ and let $\varphi\colon\,R[X_1,\ldots, X_r]\rightarrow\,
(R/\p)[X_1,\ldots, X_r]$ be the homomorphism of polynomial algebras
induced by inclusion $R/\p\hookrightarrow \k(\p)$. Given a Zariski
closed set $V\subseteq {\mathbb A}^r_{\mathbb C}$ with defining
ideal $J\subset K[X_1,\ldots, X_r]$ we let $\p(V)$ stand for the
zero locus of $\varphi(J\cap R[X_1,\ldots, X_r])$ in ${\mathbb
A}^r_{\k(\p)}$.

Given an algebraic variety $Y$ we let ${\rm Comp}(Y)$ denote the set
of all irreducible components of $Y$. If is a regular function $f$
on $Y$, we write $V(f)$ for the zero locus of $f$ in $Y$.

\begin{lemma}\label{base change}
For any $p\gg 0$ there exists a bijection
$\sigma\colon\,\mathrm{Comp}({\mathcal E}({\mathbb
C}))\stackrel{\sim}{\rightarrow}\mathrm{Comp}({\mathcal E}(\k))$
such that $\dim_{\mathbb C}{\mathcal Y}=\dim_\k\sigma(\mathcal Y)$
for all ${\mathcal Y}\in \mathrm{Comp}({\mathcal E}({\mathbb C}))$.
\end{lemma}
\begin{proof} Let ${\mathcal Y}_1,\ldots, {\mathcal Y}_q$ be the irreducible components of ${\mathcal
E}({\mathbb C})$. Since the ${\mathcal Y}_i$'s are defined over $K$,
it follows from [\cite{Noe}, Satz~XVII],  [\cite{ShT}, Ch.~III,
Prop.~17] and [\cite{Sh}, Prop.~18 and Thm.~28] that for almost all
$\p\in{\rm Spec}\,R$ the affine varieties $\p({\mathcal
Y}_1),\ldots, \p({\mathcal Y}_q)$ are irreducible and nonempty, that
$\dim_{\mathbb C} {\mathcal Y}_i=\dim_\k \p({\mathcal Y}_i)$ for all
$i$, and that $\p({\mathcal E}({\mathbb C}))\,=\,\p({\mathcal
Y}_1)\cup\ldots\cup\p({\mathcal Y}_q)$.

Note that $A\subseteq S^{-1}R$ and ${\mathcal E}({\mathbb
C})\,=\,\bigcap_{\,i,j}V(F_{ij})$. Passing to a finite extension of
$K$ if necessary, we may assume that all hypersurfaces $V(F_{ij})$
are defined over $K$ and the sets ${\mathcal
Y}_1(K),\ldots,{\mathcal Y}_q(K)$ are pairwise distinct. By
[\cite{ShT}, Ch.~III, Prop.~19], if Zariski closed sets $V_1$ and
$V_2$ are defined over $K$, then $\p(V_1\cap
V_2)\,=\,\p(V_1)\cap\p(V_2)$ for almost all $\p$. This shows that
$\p({\mathcal E}({\mathbb C}))\,=\,\bigcap_{\,i,j}\,\p(V(F_{ij}))$
for almost all $\p\in{\rm Spec}\,R$. If $p={\rm char}\,\k(\p)$, then
$\k(\p)=\k$ and $\p(V(F_{ij}))=V(^p\!F_{ij})$ for all $i,j$. As a
consequence,
$$\p({\mathcal E}({\mathbb
C}))\,=\,\textstyle{\bigcap}_{\,i,j}\,\p(V(F_{ij}))\,=\,
\textstyle{\bigcap}_{\,i,j}V(^p\!F_{ij})\,=\,{\mathcal E}(\k)$$ for
almost all $\p\in{\rm Spec}\,R$ (see [\cite{G}, pp.~28, 30] for a
similar reasoning). Since the morphism ${\rm Spec}\,R\rightarrow
{\rm Spec}\,\mathbb Z$ induced by inclusion ${\mathbb Z}\subset R$
is surjective, we obtain that ${\rm Comp}({\mathcal
E}(\k))=\{\p({\mathcal Y}_1),\ldots, \p({\mathcal Y}_q)\}$ for all
but finitely many $p\in\pi(A)$. As $\p({\mathcal
Y}_1),\ldots,\p({\mathcal Y}_q)$ are pairwise distinct for almost
all $\p$ and $\dim_{\mathbb C} {\mathcal Y}_i=\dim_\k\p({\mathcal
Y}_i)$ for all $i$, the statement follows.
\end{proof}
\subsection{}\label{5.2} In what follows we are going to use
the Lusztig--Spaltenstein theory of induced nilpotent orbits and the
Borho--Kraft theory of sheets in $\g_\k$; see [\cite{LS}] and
[\cite{BK1}]. Our main reference here is [\cite{Borh}]. Although the
base field in [\cite{Borh}] is assumed to have characteristic $0$,
the results in {\it loc.\,cit.} that we actually need are valid over
$\k$ under the assumption that ${\rm char}\,\k$ is a good prime for
the root system of $G_\k$; see [\cite{LS}], [\cite{Borh}, p.~289],
[\cite{Sp}, p.~33] and [\cite{Mc1}] for related discussions.

At some point, we are going to invoke Katsylo's results
[\cite{Kats}] on sections of sheets. The original argument in
[\cite{Kats}] involved Hausdorff neighbourhoods and holomorphic
maps, but a purely algebraic proof was recently found by Im Hof; see
[\cite{Im}, pp.~8--14]. Since all results of [\cite{Borh}] used in
[\cite{Im}, pp.~8--14] apply in good characteristic, one can see by
inspection that Im Hof's arguments are valid in positive
characteristic provided that $(\g_\k)_f\cap[e,\g_\k]=0$. The latter
holds for all $p\gg 0$.

From now on we assume that $p\gg 0$. Let $F$ be either $\mathbb C$
or $\k$ and put $\g_F:=\g_\Z\otimes_\Z F$. Then $\g_F=\Lie(G_F)$ and
$(\g_F,G_F)$ is either $(\g,G)$ or $(\g_\k,G_\k)$ (depending ${\rm
char}\,F$). Let ${\mathfrak l}_F=\Lie\,L_F$ be a proper Levi
subalgebra of $G_F$ and let ${\mathcal O}_0$ be a nilpotent orbit in
${\mathfrak l}_F$. Let $\g_F=\u_{-,\,F}\oplus{\mathfrak
l}_F\oplus\u_{+,\,F}$ be a triangular decomposition of $\g_F$ with
${\mathfrak l}_F\oplus\u_{-,\,F}$ and ${\mathfrak
l}_F\oplus\u_{+,\,F}$ being conjugate parabolic subalgebras of
$\g_F$. Since the number of nilpotent orbits in $\g_F$ is finite,
there is a unique nilpotent orbit ${\mathcal O}\subset\g_F$ which
intersects densely  with the irreducible Zariski closed set
$\overline{\mathcal O}_0+\u_{+,\,F}$. We say that the orbit
${\mathcal O}$ is {\it induced} from ${\mathcal O}_0$, written
${\mathcal O}={\rm Ind}_{\mathfrak l_F}^{\g_F}\,{\mathcal O}_0$. It
is known that $\mathcal O$ is independent of the choice of a
triangular decomposition of $\g_F$ involving ${\mathfrak l}_F$,
which justifies the notation; see [\cite{LS}, [\cite{Borh}],
[\cite{Sp}]. If $e_0\in{\mathcal O}_0$ and $e\in {\rm
Ind}_{{\mathfrak l}_F}^{\g_F}\,{\mathcal O}_0$, then $e$ is said to
be {\it induced} from $e_0$. If a nilpotent orbit $\O\subset\g_F$ is
{\it not} induced from a nilpotent orbit in a proper Levi subalgebra
of $\g_F$, then $\O$ is said to be {\it rigid} and every $x\in\O$ is
called a {\it rigid} nilpotent element of $\g_F$.
\begin{theorem}\label{rigid} Let ${\mathcal O}_0$ be a
nilpotent orbit in a proper Levi subalgebra $\mathfrak l$ of $\g$,
and ${\mathcal O}={\rm Ind}_{{\mathfrak l}}^{\g}\,{\mathcal O}_0$.
Let $e_0\in{\mathcal O}_0$ and $e\in {\mathcal O}$. If the finite
$W$-algebra $U([{\mathfrak l},{\mathfrak l}],e_0)$ affords a
$1$-dimensional representation, then so does the finite $W$-algebra
$U(\g,e)$.
\end{theorem}
\begin{proof}
(a) By the Bala--Carter theory, we may assume that ${\mathfrak
l}=\Lie(L)$ is a standard Levi subalgebra of $\g$ and
$e_0\in{\mathfrak l}_\Z$, where ${\mathfrak l}_\Z={\mathfrak
l}\cap\g_\Z$. Let $\p_\Z=\l_\Z\oplus\u_{\Z}$ be a standard parabolic
$\Z$-subalgebra of $\g_\Z$ with nilradical $\u_\Z$. By our earlier
discussion, we may also assume that $\mathcal O$ intersects densely
with $\overline{\mathcal O}_0+\u$, where $\u:=\u_\Z\otimes_\Z\mathbb
C$. Set $\bar{e}_0:=e_0\otimes 1$, an element of ${\mathfrak
l}_\k=\l_\Z\otimes_\Z\k$. As explained in [\cite{P07'}, 2.5], we may
choose $e_0$ such that $\dim_{\mathbb C}\,{\mathcal O}_0
=\dim_\k\,{\mathcal O}_{\k,\,0}$, where ${\mathcal O}_{\k,\,0}:=(\Ad
L_\k)\cdot\bar{e}_0$.

Since ${\rm Ind}_{\l}^{\g}\,{\mathcal O}_{0}$ contains a nonempty
Zariski open subset of $\overline{\mathcal O}_{0}+\u$ and the set
$\big((\Ad L({\mathbb Q})\big)\cdot e_0+\u_\mathbb Q$ is dense in
${\mathcal O}_0+\u$, there is $e_1\in\u_\mathbb Q$ with
$e:=e_0+e_1\in{\rm Ind}_{\l}^{\g}\,{\mathcal O}_{0}$. Enlarging $A$
if necessary, we may assume that $e_1\in\u_A$. For $p\in\pi(A)$ set
$\bar{e}:=\bar{e}_0+\bar{e}_1$, an element of
$\g_\k=\g_A\otimes_A\k$. It follows from [\cite{LS}, Thm.~1.3] that
$\dim\g_e=\dim \l_{e_0}$ and $\g_e\subset \p$, where
$\p=\p_\Z\otimes_\Z\mathbb C$. Therefore, $\dim\, [\p,e]=\dim
\,[\l,e_0]+\dim\, \u$, forcing $[\p_{\mathbb Q},e]=[\l_{\mathbb
Q},e_0]+\u_\mathbb Q$. Extending $A$ further, we may assume that
$[\l_A,e_0]$ is a direct summand of $\l_A$ and
$[\p_A,e]=[\l_A,e_0]+\u_A$. Then
$[\p_\k,\bar{e}]=[\l_\k,\bar{e}_0]+\u_\k$ for all $p\in\pi(A)$,
implying that $(\Ad P_\k)\cdot \bar{e}$ is dense in
$\overline{\mathcal O}_{0\,\k}+\u_\k$ (here $P_\k$ is the parabolic
subgroup of $G_\k$ with $\Lie(P_\k)=\p_\k$). This shows that
$\bar{e}\in{\rm Ind}_{\l_\k}^{\g_\k}\,{\mathcal O}_{0,\,\k}$ for all
$p\gg 0$. Extending $A$ even further we include $e$ into an
$\sl_2$-triple $\{e,h,f\}\subset\g_A$ and then consider the finite
$W$-algebra $U(\g_A,e)$ as in (\ref{4.2'}).

\smallskip

\noindent (b) Put $\xi_0:=(\bar{e}_0,\,\cdot\,)$ and
$\xi:=(\bar{e},\,\cdot\,)$, linear functions on $\l_\k$ and $\g_\k$,
respectively. Note that $\xi$ vanishes on $\u_\k$ and the
restriction of $\xi$ to $\l_\k$ equals $\xi_0$. As $[\l_\k,\l_\k]$
is a direct sum of simple ideals and $U_{\xi_0}(\l_\k)\cong
U_{\xi_0}([\l_\k,\l_\k])\otimes_\k U_{\xi_0}(\z(\l_\k))$, it is
immediate from Theorem~\ref{U-U} that for all $p\gg 0$ the reduced
enveloping algebra $U_{\xi_0}(\l_\k)$ has a simple module of
dimension $p^{d(e_0)}$, where $d(e_0)=(\dim {\mathcal O}_{0})/2$.
Given such a module $V$ we regard it as a $U_\xi(\p_\k)$-module with
the trivial action of $\u_\k$ and consider the induced
$U_\xi(\g_\k)$-module $\widetilde{V}:=
U_\xi(\g_\k)\otimes_{U_\xi(\p_\k)}V$. It follows from the PBW
theorem that
$$\dim \widetilde{V}=p^{\dim \g_\k-\dim \p_\k}\cdot
p^{d(e_0)}=p^{( \dim \g-\dim \l+\dim {\mathcal O}_{0})/2}=p^{d(e)}.
$$ Since $\dim_\k\, ({\rm Ad}^* G_\k)\cdot\xi=2d(e)$ by our choice of
$e$, Lemma~\ref{Q-eta}(iii) entails that the algebra
$U_\xi(\g_\k,\bar{e})$ affords a $1$-dimensional representation.
Then so does the algebra $U(\g_\k,e)$ thanks to
Lemmas~\ref{Q-eta}(iv) and \ref{L1}. Since this holds for all $p\gg
0$, Corollary~\ref{1-dim rep} yields that the finite $W$-algebra
$U(\g,e)$ affords a $1$-dimensional representation too. This
completes the proof.
\end{proof}
\begin{corollary}\label{c-p}
Let ${\mathcal O}_0$ and ${\mathcal O}$ be as in
Theorem~\ref{rigid}. If the finite $W$-algebra $U([\l,\l],e_0)$
affords a $1$-dimensional representation, then the enveloping
algebra $U(\g)$ has a completely prime primitive ideal $I$ with
$\mathcal{VA}(I)=\overline{\mathcal O}$.
\end{corollary}
\begin{proof}
Let $\chi=(e,\,\cdot\,)$, a linear function on $\g$. By
Theorem~\ref{rigid}, the finite $W$-algebra $U(\g,e)$ has a
$1$-dimensional module, ${\mathbb C}_0$ say. By Skryabin's
equivalence, the annihilator $I:={\rm
Ann}_{U(\g)}\big(Q_\chi\otimes_{U(\g,\,e)}{\mathbb C}_0\big)$ is a
primitive ideal of $U(\g)$; see [\cite{Sk}]. By [\cite{P07},
Thm.~3.1], the associated variety of $I$ equals $\overline{\mathcal
O}$. By [\cite{Lo}, Prop.~3.4.6], the primitive quotient $U(\g)/I$
is a domain, that is $I$ is completely prime.
\end{proof}
\begin{rem}
Corollary~\ref{c-p} reduces to rigid orbits the well-known open
problem of assigning to any nilpotent orbit ${\mathcal O}$ in $\g$ a
completely prime ideal primitive ideal $I$ of $U(\g)$ with
$\mathcal{VA}(I)=\overline{\mathcal O}$. Closely related results
were recently obtained by Borho--Joseph through a careful study of
the behaviour of Goldie rank under parabolic induction; see
[\cite{BJ}, 4.8 and 7.4]. Our arguments are completely different
(and more elementary). We recall from the proof of
Corollary~\ref{ABCD} that if all components of the semisimple Lie
algebra $[\l,\l]$  are of type ${\rm A,\,B,\,C,\,D}$, then
$U([\l,\l],e_0)$ affords $1$-dimensional representations (this
follows from [\cite{Mc}, Ch.~5] and [\cite{Lo}, Thm.~1.2.3(1)]).
\end{rem}
\subsection{}\label{5.3} The group $G_\k$ contains
a unique connected unipotent group $M_\k$ of dimension $d(e)$ with
the property that $\exp\ad x\in \Ad M_\k$ for all  $x\in\m_\k$
(since $p\gg 0$ exponentiating nilpotent derivations of $\g_\k$ does
not cause us any problems). Note that $\Lie M_\k=\m_\k$. The group
$M_\k$ is a characteristic $p$ analogue of the unipotent group $M$
from [\cite{Gi}] which, in turn, is a special instance of a group
$N_l$ for $l=\g(-1)^0$ (the group $N_l$ can be defined for any
totally isotropic subspace $l\subset\g(-1)$; see [\cite{GG}]).

In what follows we need a characteristic $p$ version of [\cite{GG},
Lemma~2.1]. Let $\kappa\colon\,\g\stackrel{\sim}{\rightarrow}\g^*$
be the Killing isomorphism given by $x\mapsto\,(x,\,\cdot\,)$, so
that $\chi=\kappa(e)$, and write $\SS_\k$ for the Slodowy slice
$\chi+\kappa(\ker \ad f)$ to the coadjoint orbit $({\rm
Ad}^*\,G_\k)\cdot \chi$. Since $\chi$ vanishes on $[\m_\k\,\m_\k]$,
the group  ${\rm Ad}^*\,M_\k$ preserves the affine subspace
$\chi+\m_\k^\perp\subset\g_\k^*$. Set
$\g_\k(1)^0:=\,\{x\in\g_\k(1)\,\vert\,\,(x,\g_\k(-1)^{0})=0\}$, an
$s$-dimensional subspace of $\g(1)$. Then
$$\kappa^{-1}(\m_\k^\perp)\,=\,\g_\k(1)^0\,\textstyle{\oplus\,\bigoplus}_{i\le
0}\,\g_\k(i).$$ Let $\lambda_e\in X_*(G_\k)$ be the cocharacter such
that $(\Ad \lambda_e(x))\cdot x=t^ix$ for all $x\in\g_\k(i)$ and
$i\in\Z$ and define a rational action
$\rho_e\colon\,\k^\times\rightarrow {\rm GL}(\g_\k)$ by setting
$\rho_e(t)(x):=t^2(\Ad \lambda_e)(t^{-1})(x)$ for all $x\in\g_\k$.
\begin{lemma}{\rm (cf. [\cite{GG}, Lemma~2.1])}\label{action-map}
The coadjoint action-map $\alpha\colon\,M_\k\times
\SS_\k\rightarrow\, \chi+\m_\k^\perp$ is an isomorphism of affine
varieties.
\end{lemma}
\begin{proof}
As $M_\k$ is a connected unipotent group, we have that $M_\k\cong
{\mathbb A}_\k^{d(e)}$ as affine varieties. Set
$\widetilde{\m}_\k:=\kappa^{-1}(\m_\k^\perp)$. In order to prove the
lemma we need to show that the adjoint action-map $\alpha\colon\,
M_\k\times(e+\ker\ad f)\map e+\widetilde{\m}_\k$ is an isomorphism.
It is easy to see that both varieties have the same dimension.

The differential ${\rm d}_{(1,\,e)}\alpha\colon\,\m_\k\oplus\ker\ad
f\map\,\widetilde{\m}_\k$ is given by $x+z\mapsto [x,e]+z$ for all
$x\in\m_\k$ and $z\in\ker\ad f$. Since $\ad e$ is injective on
$\m_\k$ and $(\ker\ad f)\cap (\mathrm{Im}\,\ad e)=0$ under our
assumptions on $p$, the map ${\rm d}_{(1,\,e)}\alpha$ is a linear
isomorphism. As in [\cite{GG}], we define a $\k^\times$-action on
the affine variety $M_\k\times (e+\widetilde{\m}_\k)$ by
$$t\cdot
(g,x):=\,(\lambda_e(t)^{-1}g\lambda_e(t),\rho_e(t)(x))\qquad\quad\
(t\in\k^\times,\,g\in M_\k,\,x\in\widetilde{\m}_\k).$$ As in
[\cite{GG}, p.~246], we see that this $\k^\times$-action is
contracting and the Zariski closure of the set
$\{t\cdot(g,x)\,|\,\,t\in\k^\times\}$ contains $(1, e)$. Since the
morphism $\alpha$ is $\k^\times$-equivariant, we can apply
[\cite{Sl}, Lemma~8.1.1] to complete the proof.
\end{proof}
\begin{rem}
Instead of applying [\cite{Sl}, Lemma~8.1.1] we could finish the
proof of Lemma~\ref{action-map} by a more geometric argument
outlined in [\cite{Gi0}, p.~553]. This argument is purely algebraic
and works in all characteristics.
\end{rem}
\subsection{}\label{5.4}
Let $\widehat{\mathcal E}$ denote the maximal spectrum of
$\widehat{U}(\g_\k,e)^{\rm ab}$.  Composing the embedding
$Z_\p(\widetilde{\a}_\k)\hookrightarrow \widehat{U}(\g_\k,e)$ with
the canonical homomorphism $\widehat{U}(\g_\k,e)\twoheadrightarrow
\widehat{U}(\g_\k,e)^{\rm ab}$ we get a map
$\k[\chi+\m_\k^\perp]\rightarrow \widehat{U}(\g_\k,e)^{\rm ab}$
which, in turn, gives rise to an algebra homomorphism
$$\beta^*\colon\,\k[\chi+\m_\k^\perp]\map\,\widehat{U}(\g_\k,e)^{\rm ab}/{\rm
nil}\,\widehat{U}(\g_\k,e)^{\rm ab}\,=\,\k[\widehat{\mathcal E}]$$
(as in (\ref{4.4}), we identify $Z_p(\widetilde{\a}_\k)$ with the
coordinate algebra $\k[\chi+\m_\chi^\perp]$). Let $J_\chi=\ker
\beta^*$ and denote by $Y_\chi$ the zero locus of $J_\chi$ in
$\chi+\m_\chi^\perp$. As $\widehat{U}(\g_\k,e)$ is a finite
$Z_p(\widetilde{\a}_\k)$-module by Theorem~\ref{U-hat}(ii),
$\k[\widehat{\mathcal E}]=\widehat{U}(\g_\k,e)^{\rm ab}/{\rm
nil}\,\widehat{U}(\g_\k,e)^{\rm ab}$ is a finite module over
$\k[Y_\chi]$. So $\beta^*$ induces a finite (hence surjective)
morphism of affine varieties $$\beta\colon\,\widehat{\mathcal E}\map
Y_\chi.$$

The group $M_\k$ preserves the left ideal $U(\g_\k)N_{\chi,\,\k}$
and therefore acts on
$\widehat{U}(\g_\k,e)=\big(\End_{\g_\k}U(\g_\k)/U(\g_\k)
N_{\chi,\,\k}\big)^{\rm op}$ as algebra automorphisms. Hence $M_\k$
acts on $\widehat{U}(\g_\k,e)^{\rm ab}$. As $M_\k$ preserves
$\rho_\k(Z_p)\cong \k[\chi+\m_\chi]$, the map $\beta^*$ is a
homomorphism of $M_\k$-modules. Thus, both
 $\widehat{\mathcal E}$ and $Y_\chi$ are
$M_\k$-varieties and the morphism $\beta$ is $M_\k$-equivariant.
Thanks to Lemma~\ref{action-map}, the action-map $M_\k\times
\SS_\k\rightarrow \chi+\m_\k^\perp$ induces an isomorphism
\begin{equation}\label{542} Y_\chi\,\cong\, M_\k\times(\SS_\k \cap Y_\chi).
\end{equation}
\begin{prop}\label{E-hat} The following statements hold:
\begin{itemize}
\item[(1)] $\widehat{\mathcal E}\,\cong\,M_\k\times{\mathcal E}(\k)$ as affine varieties.
\smallskip
\item[(2)] The map $\beta$ induces a finite morphism
$\bar{\beta}\colon\,{\mathcal E}(\k)\twoheadrightarrow\,\SS_\k\cap
Y_\chi.$
\end{itemize}
\end{prop}
\begin{proof}
(a) Let $\widehat{\mathcal E}_0:=\beta^{-1}(\SS_\k\cap Y_\chi)$, a
Zariski closed subset of $\widehat{\mathcal E}$. Since $\beta$ is
$M_\k$-equivariant, we have a natural morphism
$\gamma\colon\,M_\k\times\widehat{\mathcal E}_0\rightarrow
\widehat{\mathcal E}$. As $\beta$ is surjective, (\ref{542}) entails
that so is $\gamma$. If $p_1$ is the first projection
$Y_\chi\stackrel{\sim}{\rightarrow}\,M_\k\times (\SS_\k\cap
Y_\chi)\twoheadrightarrow\,M_\k$, then
$p_1(x)^{-1}(x)\in\widehat{\mathcal E}_0$ for every
$x\in\widehat{\mathcal E}$, and the the morphism
$$\widehat{\mathcal E}\map\,M_\k\times\widehat{\mathcal
E}_0,\qquad x\mapsto \big(p_1(x),p_1(x)^{-1}(x)\big),$$ is the
inverse of $\gamma$. Hence $\widehat{\mathcal E}\cong
M_\k\times\widehat{\mathcal E}_0$ as affine varieties.

\smallskip

\noindent(b) By Corollary~\ref{ab}, $\widehat{U}(\g_\k,e)^{\rm
ab}\cong U(\g_\k,e)^{\rm ab}\otimes Z_p(\a_\k)$. Since $Z_p(\a_\k)$
is a domain, it follows that $\k[\widehat{\mathcal E}]\,\cong\,
\k[{\mathcal E}(\k)]\otimes Z_p(\a_\k)$ as algebras. Therefore,
$Z_p(\a_\k)$ embeds into $\k[\widehat{\mathcal E}]$. It also follows
that the ideal $\k[\widehat{\mathcal E}]\a_\k$ of
$\k[\widehat{\mathcal E}]$ is radical and its zero locus, $\V$ say,
is isomorphic to ${\mathcal E}(\k)$. On the other hand, it is
evident from (\ref{orthog}) that the ideal of
$Z_p(\widetilde{\a}_\k)=\k[\chi+\m_\k^\perp]$ generated by $\a_\k$
is nothing but the defining ideal of $\SS_\k$ in
$\k[\chi+\m_\k^\perp]$. As a consequence, $\beta(\V)\subseteq
\SS_\k\cap Y_\chi$, implying $\V\subseteq \widehat{\mathcal E}_0$.

Now  $\widehat{\mathcal E}\cong \widehat{\mathcal E}_0\times M_\k$
by part~(a) and $\widehat{\mathcal E}\cong
\mathcal{E}(\k)\times{\mathbb A}_\k^{d(e)}$ by our earlier remarks
in this part. As $M_\k\cong {\mathbb A}_\k^{d(e)}$ and ${\mathcal
E}(\k)\cong \V$, we deduce that there exists a bijection $\tau$
between ${\rm Comp}(\V)$ and ${\rm Comp}(\widehat{\mathcal E}_0)$
such that $\dim X=\dim \tau(X)$ for all $X\in{\rm Comp}(\V)$. As
$\V\subseteq\widehat{\mathcal E}_0$, this yields
$\V=\widehat{\mathcal E}_0$ and statement~(1) follows.

\smallskip

\noindent (c) Let $I_1$ be the augmentation ideal of the Hopf
algebra $\k[M_\k]$. By part~(b), we can identify $\k[M_\k]\otimes
\k[{\mathcal E}(\k)]$ and $\k[M_\k]\otimes\k[\SS_\k\cap Y_\chi]$
with $\k[\widehat{\mathcal E}]$ and $\k[Y_\chi]$, respectively, in
such a way that $\widetilde{I}_1:=I_1\otimes\k[{\mathcal E}(\k)]$
identifies with the defining ideal of the closed subset
$\widehat{\E}_0\cong\E(\k)$ of $\widehat{\E}$. Since $\beta$ is
$M_\k$-equivariant, composing $\beta^*$ with the canonical
homomorphism $\k[\widehat{\mathcal E}]\twoheadrightarrow
\,\k[\widehat{\mathcal E}]/\widetilde{I}_1$ induces an algebra map
$\bar{\beta}^*\colon\,\k[Y_\chi]\map\,\k[{\mathcal E}(\k)]$ whose
kernel equals $I_1\otimes\k[\SS_\k\cap Y_\chi]$. Since $\beta$ is a
finite morphism and $\ker\bar{\beta}^*$ identifies with the defining
ideal of $\{1\}\times (\SS_\k\cap Y_\chi)\cong \,\SS_\k\cap Y_\chi$,
we thus obtain a finite morphism $\bar{\beta}\colon\,{\mathcal
E}(\k)\twoheadrightarrow\, \SS_\k\cap Y_\chi$. This completes the
proof. \end{proof}
\subsection{}\label{5.5} In order to obtain a
good lower bound on the number of irreducible components of
${\mathcal E}(\CC)$ we now need more information the affine variety
$\SS_\k\cap Y_\chi$.

For $d\in\N$, define
$\g_\k^{(d)}:=\{x\in\g_\k\,|\,\,\dim\,(\g_\k)_x=d\}$. When $p\gg 0$,
the centraliser $(\g_\k)_x$ coincides with the Lie algebra of
$(G_\k)_x=Z_{G_\k}(x)$ and $\dim\,(\g_\k)_x=\dim\,(G_\k)_x$ for all
$x\in\g_\k$; see [\cite{Ja1}], for instance. Since the set
$\g_\k^{(d)}$ is quasi-affine, it decomposes as a union of finitely
many irreducible components $\g_\k$. The irreducible components of
the $\g_\k^{(d)}$'s  are called {\it sheets} of $\g_\k$. The sheets
are $(\Ad G_\k)$-stable, locally closed subsets of $\g_\k$. By one
of the main result of [\cite{Borh}], there is a bijection between
the sheets of $\g_\k$ and the $G_\k$-conjugacy classes of pairs
$(\l,\O_0)$, where $\l$ is a Levi subalgebra of $\g_\k$ and $\O_0$
is a rigid nilpotent orbit in $[\l,\l]$. Borho's classification of
sheets remains valid over $\k$ under the assumption that ${\rm
char}\,\k$ is a good prime for the root system of $G$; see
(\ref{5.2}) for related references. By [\cite{BK1}, 5.8], every
sheet of $\g_\k$ contains a unique nilpotent orbit. However, outside
type $\rm A$ sheets are {\it not} disjoint, and when two sheets
overlap, they always contain the same nilpotent orbit.

Let $\l$ be a Levi subalgebra of $\g_\k$. The centre $\z(\l)$ of
$\l$ is a toral subalgebra of $\g_\k$, and $(\g_\k)_z\supseteq \l$
for all $z\in\z(\l)$. We denote by $\z(\l)_{\rm reg}$ the set of all
$z\in\z(\l)$ for which the equality $(\g_\k)_z=\l$ holds; this is a
nonempty Zariski open subset of $\z(\l)$. For a nilpotent element
$e_0\in [\l,\l]$ define ${\mathcal D}(\l,e_0):=(\Ad G_\k)\cdot
(e_0+\z(\l)_{\rm reg})$, a locally closed subset of $\g_\k$. We call
${\mathcal D}(\l, e_0)$ a {\it decomposition class} of $\g_\k$ (this
term has to do with the Jordan--Chevalley decomposition in $\g_\k$).
Each sheet ${\mathcal S}\subset \g_\k$ is a finite union of
decomposition classes and contains a unique open such class; see
[\cite{Borh}, 3.7]. Moreover, if ${\mathcal D}(\l, e_0)$ is open in
$\mathcal S$, then $\O_0:=(\Ad L)\cdot e_0$ is rigid in $[\l,\l]$,
the orbit ${\rm Ind}_\l^{\g_\k}(\O_0)$ is contained in $\mathcal S$,
and $\dim\,({\mathcal S}/G_\k)=\dim\z(\l)$. These results,
established in [\cite{Borh}, 3.2, 4.3 and 5.6], are valid under our
assumption on $p$.

Let $C(e):=(G_\k)_e\cap(G_\k)_f$. This is a reductive group and its
finite quotient $\Gamma(e):=C(e)/C(e)^\circ$ identifies naturally
with the component group $\Gamma(e):=(G_\k)_e/(G_\k)_e^\circ$; see
[\cite{P03}], for instance. If ${\mathcal S}(e)$ is a sheet
containing $e$, then the set $X:={\mathcal S}(e)\cap (e+\ker \ad f)$
is Zariski closed and connected. Indeed, since $e\in X$, this
follows from the fact that $X$ is preserved by the contracting
action of the $1$-dimensional torus $\rho_e(\k^\times)$ introduced
in (\ref{5.3}). Clearly, $X$ is stable under the adjoint action of
$C(e)$.

Assume for a moment that $\k=\CC$. In [\cite{Kats}], Katsylo proved
that the connected group $C(e)^\circ$ acts trivially on $X$ and the
irreducible components of $X$ are permuted transitively by the
component group $\Gamma(e)$. The action-morphism
$\varphi\colon\,G_\k\times X\map\, \mathcal S(e)$ is smooth,
surjective of relative dimension $\dim\,(\g_\k)_e$. By
[\cite{Kats}], it gives rise to an open morphism
$\psi\colon\,{\mathcal S}(e)\map X/\Gamma(e)$, whose fibres are
$(\Ad G_\k)$-orbits, such that for any open set $U\subseteq
X/\Gamma(e)$ the induced map $\k[U]\map\,\k[ \psi^{-1}(U)]^{G_\k}$
is an isomorphism. In brief, $\psi$ is a geometric quotient. Since
$\Gamma(e)$ acts transitively on ${\rm Comp}(X)$, it is
straightforward to see that $X/\Gamma(e)\,=\,{\rm Specm}\,
\k[X]^{\Gamma(e)}$ is an irreducible affine variety.

A purely algebraic (and rather short) proof of Katsylo's results was
given in [\cite{Im}]. It is a matter of routine to check that this
proof works under our assumption on $p$.

Summarising, if ${\mathcal D}(\l,e_0)$ is the open decomposition
class in ${\mathcal S}(e)$, then $e\in{\rm Ind}_\l^{\g_\k}\,\O_0$,
the orbit $\O_0=(\Ad\,L)\cdot e_0$ is rigid in $[\l,\l]$, and
\begin{equation}\label{Gamma(e)}
\dim\,\z(\l)\,=\,\dim\,{\mathcal S}(e)/G_\k\,=\,\dim\,X_i\qquad\quad
\forall\,X_i\in\mathrm{Comp}(X).
\end{equation}
\subsection{}\label{5.6}
Let ${\mathcal S}_1,\ldots, {\mathcal S}_t$ be the pairwise distinct
sheets of $\g_\k$ containing our nilpotent element $e$. For $1\le
i\le t$ set $X_i:={\mathcal S}_i\cap (e+\ker\ad f)$ and denote by
${\mathcal D}(\l_i,e_i)$ the open decomposition classes of
${\mathcal S}_i$. Recall from (\ref{5.3}) the Killing isomorphism
$\kappa\colon\,\g_\k\stackrel{\sim}{\rightarrow}\g_\k^*$ and put
$Y_i\,:=\,\kappa(X_i)\,=\,\kappa({\mathcal S}_i)\cap\SS_\k$, where
$1\le i\le t$.
\begin{prop}\label{Y-chi}
The following are true for all $p\gg0$:
\begin{itemize}
\item[(i)\,]
$Y_\chi\cap \SS_\k\,\subseteq\, Y_1\cup\ldots\cup Y_t$.

\smallskip

\item[(ii)\,] $\dim\, {\mathcal E}(\CC)=\dim\,{\mathcal E}(\k)\,\le\, \max_{1\le i\le t}\,
\dim\,\z(\l_i)$.

\smallskip

\item[(iii)\,] If $e$ is rigid, then ${\mathcal E}(\k)$ and
${\mathcal E}(\CC)$ are finite sets of the same cardinality.
\end{itemize}
\end{prop}
\begin{proof}
If $\eta\in Y_\chi$, then the definition of $\beta^*$ in (\ref{5.4})
shows that the algebra
$\widehat{U}_\eta(\g_\k,e)=\widehat{U}(\g_\k,e)\otimes_{Z_p(\widetilde{\a}_\k)}\k_\eta$
affords a $1$-dimensional representation. In part~(b) of the proof
of Theorem~\ref{U-U} we have shown that this algebra is isomorphic
to $U_\eta(\g_\k,e)$. By Lemma~\ref{Q-eta}(iii), the reduced
enveloping algebra $U_\eta(\g_\k)$ affords a representation of
dimension $p^{d(e)}$. Then [\cite{P95}, Thm.~3.10] yields
$\dim\z(\eta)\le d(e)$.

On the other hand, our discussion in (\ref{5.3}) shows that
$\eta=\kappa(e+x)$ for some $x\in\bigoplus_{i\le 1}\,\g_\k(i)$.
Since $e$ lies in the Zariski closure of $\rho_e(\k^\times)(e+x)$
and the centralisers of $\rho_e(t)(e+x)$ and $e+x$ in $\g_\k$ have
the same dimension for all $t\in\k^\times$, it must be that $\dim
\z(\eta)\ge r$. As a result, $e+x\in\g_\k^{(r)}$. Every irreducible
component of $\g_\k^{r)}$ containing $e+x$ must contain
$\rho_e(\k^\times)(e+x)$ and hence $e$. This yields
$$Y_\chi\,\subseteq \,\textstyle{\bigcup}_{1\le i\le t}\,\big(\kappa({\mathcal
S}_i)\cap (\chi+\m_\k^\perp)\big),$$ from which statement~(i) is
immediate. Since $\dim (Y_\chi\cap\SS_\k)=\dim\,{\mathcal E}(\k)$ by
Proposition~\ref{E-hat}(2) and $\dim\,\E(\k)=\dim\,\E(\CC)$ by
Lemma~\ref{base change}, statement~(ii) now follows from
(\ref{Gamma(e)}). When $e$ is rigid, there is only one sheet
containing $e$, namely, the orbit $\O=(\Ad G_\k)\cdot e$. So
(\ref{Gamma(e)}) implies that $X=\O\cap (e+\ker\ad f) =\{e\}$ (for
$X$ is connected). Then (i) shows that either
$Y_\chi\cap\SS_\k=\{\chi\}$ or $Y_\chi\cap\SS_\k=\varnothing$. By
Proposition~\ref{E-hat}(2) and Lemma~\ref{base change}, the sets
$\E(\CC)$ and $\E(\k)$ are finite and have the same cardinality.
\end{proof}

We say that $\g$ is {\it well-behaved} if for any proper Levi
subalgebra $\l$ of $\g$ and any nilpotent element $e_0\in\l$ the
finite $W$-algebra $U([\l,\l],e_0)$ admits a $1$-dimensional
representation. Thanks to [\cite{Mc}, Ch.~5] and [\cite{Lo},
Thm.~1.2.3(1)] the Lie algebras of types ${\rm
A_\ell,\,B_\ell,\,C_\ell ,\,D_\ell,\,G_2,\,F_4,\,E_6}$ are
well-behaved (in these cases all irreducible components of the
proper subsets of $\Pi$ have type ${\rm A,\,B,\,C,\, D}$).
\begin{prop}\label{behave}
If $\g$ is well-behaved and $e$ is not rigid, then $Y_\chi\cap\,
\SS_\k\, = \,Y_1\cup\ldots\cup Y_t$ for all $p\gg 0$.
\end{prop}
\begin{proof}
Since $\beta$ is a closed morphism, we just need to show that
$\beta(\widehat{\E})$ contains an open dense subset of each $Y_i$.
By (\ref{5.5}), the adjoint action-map $\varphi\colon\, G_\k\times
X_i\map\,{\mathcal S}_i$ is surjective. As ${\mathcal D}(\l_i,e_i)$
is open in ${\mathcal S}_i$ and $C(e)$ permutes the components of
$X_i$ transitively, the set $\varphi^{-1}({\mathcal D}(\l_i,e_i))$
is open dense in $G_\k\times X_i$. Looking at the image of
$\varphi^{-1}({\mathcal D}(\l_i,e_i))$ under the second projection
$G_\k\times X_i\twoheadrightarrow\,X_i$ we observe that the set
$$X_i^{\rm reg}:=\,{\mathcal D}(\l_i,e_i)\cap (e+\ker\ad f)$$
contains an open dense subset of $X_i$. We are thus reduced to show
that for every $\eta\in \kappa(X_i^{\rm reg})$ the algebra
$\widehat{U}_\eta(\g_\k,e)$ has a $1$-dimensional representation. As
explained in part~(b) of the proof of Theorem~\ref{U-U} this is
equivalent to showing that the reduced enveloping algebra
$U_\eta(\g_\k)$ has a module of dimension $p^{d(e)}$. Note that
$\l_i$ is a proper Levi subalgebra of $\g_\k$ (otherwise $e$ would
be rigid in $\g_\k$).

As every element of ${\mathcal D}(\l_i,e_i)$ is $(\Ad
G_\k)$-conjugate to an element in $e_i+\z(\l_i)_{\rm reg}$, no
generality will be lost by assuming that $\eta=\eta_s+\eta_n$, where
$\eta_n=(e_i,\,\cdot\,)$ and $\eta_s=(z,\,\cdot\,)$ for some
$z\in\z(\l_i)_{\rm reg}$. Since $\eta=\eta_s+\eta_n$ is the Jordan
decomposition of $\eta$ and $\z(\eta_s)=(\g_\k)_z=\l_i$, applying
the Kac--Weisfeiler theorem (as generalised by
Friedlander--Parshall) we derive that $U_\eta(\g_\k)\cong
\Mat_{p^{m_i}}\big(U_{\eta}(\l_i)\big)$, where
$m_i=(\dim\,\g_\k-\dim\,\l_i)/2$; see [\cite{P02}, 2.5], for
instance.

As $U_{\eta}(\l_i)\cong U_{\eta_n}([\l_i,\l_i])\otimes
U_{\eta}(\z(\l_i))$ and $\dim\, (\Ad L_i)\cdot e_i=\,d(e)-m_i$, it
remains to show that the reduced enveloping algebra
$U_{\eta_n}([\l_i,\l_i])$ has a module of dimension
$p^{(d(e)-m_i)/2}$. But this follows from Theorem~\ref{U-U} by our
assumption on $\g$.
\end{proof}
\subsection{}\label{5.7} We are now in a position to state and prove the main result of
this section:

\begin{theorem}\label{E-C}
Suppose $\g$ is well-behaved and let $e$ be any nonrigid nilpotent
element of $\g$. Let ${\mathcal S}_1,\ldots, {\mathcal S}_t$ be the
pairwise distinct sheets of $\g$ containing $e$. Let ${\mathcal
D}(\l_i,e_i)$ be the open decomposition class of ${\mathcal S}_i$
and $X_i={\mathcal S}_i\cap (e+\ker\ad f)$, where $1\le i\le t$.
Then there exists a surjection
$${\rm Comp}(\E(\CC))\twoheadrightarrow\,{\rm
Comp}(X_1)\sqcup\ldots\sqcup{\rm Comp}(X_t)$$ such that for every
component $\mathcal Y$ of $\E(\CC)$ lying over ${\rm Comp}(X_i)$ the
equality $\dim\,{\mathcal Y}=\dim\,\z(\l_i)$ holds.
\end{theorem}
\begin{proof}
We may assume that $e\in\g_\Z$, that $\l_1,\ldots,\l_t$ are standard
parabolic subalgebras of $\g$, and that $e_i\in\l_\Z$ for all $i$.
We then may regard $e$ and $e_i$ as nilpotent elements of $\g_\k$
and $\l_{i,\,\k}$, respectively. Arguing as in part~(a) of the proof
of Theorem~\ref{rigid}, one observes that for $p\gg 0$ each $e_i$ is
rigid in $\l_{i,\,\k}$ and $e$ is not rigid in $\g_\k$. By
Lemma~\ref{base change}, there is a dimension preserving bijection
between ${\rm Comp}(\E(\CC))$ and ${\rm Comp}(\E(\k))$.

Let ${\mathcal S}_\k$ be a sheet of $\g_\k$ containing $e$ and let
${\mathcal D}(\l,e_0)$ be the open decomposition class of ${\mathcal
S}_k$. Since $\l$ is $(\Ad G_\k)$-conjugate to a standard Levi
subalgebra and $e\in{\rm Ind}_\l^{\g_k}\,\O_0$ for some rigid
nilpotent orbit $\O_0\subset [\l,\l]$, our discussion in (\ref{5.5})
shows that there is a dimension preserving bijection between the
sheets of $\g$ containing $e$ and those of $\g_\k$ containing its
image in $\g_\k$. Moreover, every sheet of $\g_\k$ containing
$e\in\g_\k$ has the form  $${\mathcal
S}_{i,\,\k}\,:=\,\overline{{\mathcal D}(\l_{i,\,\k},e_i)}
\cap\g_\k^{(r)},\qquad\quad 1\le i\le t.$$ By our discussion in
(\ref{5.5}), each variety $X_{i,\,\k}:={\mathcal
S}_{i,\,\k}\cap(e+\ker\ad f)$ is equidimensional of dimension
$\dim\,\z(\l_i)$. To finish the proof it suffices now to apply
Theorem~\ref{behave} and Proposition~\ref{E-hat}(ii).
\end{proof}
\begin{rem}
In [\cite{P07}, 3.4] the author made the following conjecture:
\begin{itemize}
\item[1.]\, Every finite $W$-algebra $U(\g,e)$ has an ideal of codimension $1$.

\smallskip

\item[2.]\, The ideals of codimension $1$ in $U(\g,e)$ are finite in number if
and only if the orbit $(\Ad G)\cdot e$ is rigid.

\smallskip

\item[3.]\, For any ideal $I$ of codimension $1$ in $U(\g,e)$ the
annihilator of the $U(\g)$-module
$Q_\chi\otimes_{U(\g,\,e)}(U(\g,e)/I)$ is a completely prime ideal
of $U(\g)$.
\end{itemize}
Theorem~\ref{rigid} reduces part~1 of this conjecture to the case
where $e$ is rigid in $\g$, whereas Theorem~\ref{E-C} and
Proposition~\ref{Y-chi}(iii) show that part~1 implies part~2. Part~3
was recently proved by Losev, who also confirmed part~1 for the Lie
algebras of classical types; see [\cite{Lo}]. As far as I am aware,
part~1 remains open for some rigid nilpotent orbits in Lie algebras
of types ${\rm F_4,\,E_6,\,E_7,\,E_8}$. There are indications that
these open cases will soon be tackled by computational methods.
\end{rem}
\subsection{}\label{5.8} As an application of Theorem~\ref{E-C} we
now wish to describe the commutative quotient $U(\g,e)^{\rm ab}$ for
$\g=\mathfrak{gl}(N)$. We are going to use the explicit presentation
of $U(\g,e)$ obtained by Brundan--Kleshchev in [\cite{BrK}]. Given a
partition $\mu=(q_1\ge\cdots \ge q_m)$ of $N$ with $m$ parts we
denote by $\g(\mu)$ the standard Levi subalgebra
$\mathfrak{gl}({q_1})\oplus\cdots\oplus\mathfrak{gl}({q_m})$ of
$\mathfrak{gl}(N)$. Note that the centre of $\mathfrak{gl}(\mu)$ has
dimension $m$.

Let $\lambda=(p_n\ge p_{n-1}\ge\cdots\ge p_1)$ be a partition of $N$
with $n$ parts. As in [\cite{BrK}], we associate with $\lambda$ a
nilpotent element $e=e_\lambda\in\mathfrak{gl}(N)$ of Jordan type
$(p_1,p_2,\ldots, p_n)$. By [\cite{BrK}, Thm.~10.1], the finite
$W$-algebra $U(\g,e)$ is isomorphic to the shifted truncated Yangian
$Y_{n,\,l}(\sigma)$ of level $l:=p_n$. Here $\sigma$ is an upper
triangular matrix of order $n$ with nonnegative integral entries;
see [\cite{BrK}, \S~7] for more detail. It follows from the main
results of [\cite{BrK}] that $U(\g,e)$ is generated by elements
\begin{eqnarray}
&&\{D_{i}^{(r)}\in U(\g,e)\,|\,\,\,1\le i\le n;\,\,r\ge 1\},\nonumber\\
&&\{E_i^{(r)}\in U(\g,e)\,|\,\,\,1\le i\le n-1;\,\, r>p_{i+1}-p_i\},\\
&&\{F_i^{(r)}\in U(\g,e)\,|\,\,\,1\le i\le n-1;\,\, r\ge
1\},\nonumber
\end{eqnarray}
with $D_1^{(r)}=0$ for $r>p_1$, subject to certain relations; see
[\cite{BrK}, (2.4)\,--\,(2.15)].

Recall from [\cite{P07}, p.~524] that the centre $Z(\g)$ of the
universal enveloping algebra $U(\g)$ identifies canonically with the
centre of $U(\g,e)$ (this holds for for any simple Lie algebra $\g$
and any nilpotent element $e\in\g$).
\begin{theorem}\label{gln} If $\g=\mathfrak{gl}(N)$ and
$e=e_\lambda$, then $U(\g,e)^{\rm ab}$ is isomorphic to a polynomial
algebra in $l=p_n$ variables.
\end{theorem}
\begin{proof}
If $n=1$, then $e$ is regular and $l=N$. Hence $U(\g,e)\cong
Z(\g)\cong \CC[X_1,\ldots, X_l]$. So assume from now that $n\ge 2$
and denote by $d_i^{(r)}$, $e_i^{(r)}$, $f_i^{(r)}$ the images of
$D_i^{(r)}$, $E_i^{(r)}$, $F_i^{(r)}$ in $U(\g,e)^{\rm ab}$.
Applying [\cite{BrK}, (2.6) and (2.7)] with $r=1$ we see that
$e_i^{(s)}=f_i^{(s)}=0$ for all $1\le i\le n-1$ and $s\ge 1$. By
[\cite{BrK}, (2.4)], the elements $D_i^{(r)}$ and $D_j^{(s)}$
commute for all $i,j\le n$ and all $r, s$.

As in [\cite{BrK}], we set $D_i^{(0)}:=1$ and $D_i(u):=\sum_{r\ge
0}\,D_i^{(r)}u^{-r}$, an element of $Y_{n,\,l}(\sigma)[u^{-1}]$, and
define $\widetilde{D}_i^{(r)}$ from the equation $
\widetilde{D}_i(u)\,=\,\textstyle{\sum}_{r\ge
0}\,\widetilde{D}_i^{(r)}u^{-r}\,:=\,-D_i(u)^{-1}.$ Since
$$
D_i(u)^{-1}\,=\,\Big(1+\textstyle{\sum}_{r\ge
1}\,D_i^{(r)}u^{-r}\Big)^{-1} \,=\,1+\textstyle{\sum}_{k\ge
1}\,(-1)^k\Big(\textstyle{\sum}_{r\ge 1} \,D_i^{(r)}u^{-r}\Big)^k,
$$
it is easy to see that $\widetilde{D}_i^{(r)}-D_i^{(r)}$ is a
polynomial in $D_i^{(1)},\ldots, D_i^{(r-1)}$ with initial form of
degree $\ge 2$. In particular, $\widetilde{D}_i^{(0)}=-1$,
$\widetilde{D}_i^{(1)}=D_i^{(1)}$ and
$\widetilde{D}_i^{(2)}=D_i^{(2)}-D_i^{(1)}D_i^{(1)}$. Let
$\tilde{d}_i^{(r)}$ denote the image of $\widetilde{D}_i^{(r)}$ in
$U(\g,e)$. Since $[e_j^{(r)},f_j^{(r)}]=0$, applying [\cite{BrK},
(2.5)] yields
\begin{equation}\label{dd}
\textstyle{\sum}_{t=0}^{r}\,\tilde{d}_j^{(t)}d_{j+1}^{(r-t)}=0\qquad\qquad\quad(1\le
j \le n-1,\ r> p_{i+1}-p_i).
\end{equation}

Set $p_0:=0$ and denote by ${\mathcal A}'$ the subalgebra of
$U(\g,e)^{\rm ab}$ generated by all $d_j^{(k)}$ with $1\le k\le
p_j-p_{j-1}$. We claim that $d_j^{(k)}\in{\mathcal A}'$ for all
$(j,k)$ with $1\le j\le n$ and $k\ge 0$. The claim is certainly true
when $j+k=2$. Suppose $d_j^{(k)}\in{\mathcal A}'$ for all $(j,k)$
with $j+k\le d$ and let $(i,r)$ be such that $D_i^{(r)}\ne 0$ and
$i+r=d+1$. If $r\le p_i-p_{i-1}$, then $d_i^{(r)}\in {\mathcal A}'$
by the definition of ${\mathcal A}'$. If $r>p_i-p_{i-1}$, then $i\ge
2$, for otherwise $D_i^{(r)}=0$. Applying (\ref{dd}) with $j=i-1$ we
obtain
$$d_{i}^{(r)}\in\,\CC\big[\tilde{d}_{i-1}^{(1)},\ldots,\tilde{d}_{i-1}^{(r)},\,d_{i}^{(1)},\ldots,
d_{i}^{(r-1)}\big].$$ Since $d_{i}^{(1)},\ldots,
d_{i}^{(r-1)}\in\,{\mathcal A}'$ by our induction assumption and
$\tilde{d}_{i-1}^{(m)}-d_{i-1}^{(m)}$ is a polynomial in
$d_{i-1}^{(1)}\,\ldots, d_{i-1}^{(m-1)}$, the claim follows by
induction on $d$. Since $d_i^{(0)}=1$, we thus deduce that the
algebra $U(\g,e)^{\rm ab}$ is generated by
$p_1+(p_2-p_1)+\cdots+(p_n-p_{n-1})=p_n=l$ elements.

As a result, there is a surjective algebra map
$\gamma\colon\CC[X_1,\ldots,X_l]\rightarrow\,U(\g,e)^{\rm ab}$. If
$\gamma$ is not injective, then the morphism induced by $\gamma$
identifies $\E(\CC)={\rm Specm}\,U(\g,e)^{\rm ab}$ with a {\it
proper} Zariski closed subset of ${\mathbb A}_{\CC}^l$. Then
$\dim\,\E(\CC)<l$. On the other hand, [\cite{Kraft}, Satz~2.2] says
that $e$ is Richardson in a parabolic subalgebra $\p=\l\oplus\u$ of
$\g=\mathfrak{gl}(N)$ with $\l\cong\g(\lambda')$, where $\lambda'$
is the partition of $N$ conjugate to $\lambda$. In other words,
$(\Ad\,\mathrm {GL}(n))\cdot e={\rm Ind}_{\g(\lambda')}^\g\,\{0\}$.
As $\lambda'$ has $l$ parts, Theorem~\ref{E-C} then yields
$\dim\,\E(\CC)\ge \dim\,\z(\g(\lambda'))=l$. This contradiction
shows that $U(\g,e)^{\rm ab}\,\cong\,\CC[X_1,\ldots,X_l]$.
\end{proof}
\begin{question}
{\it Is it true that for any simple Lie algebra $\g$ and any
nilpotent element $e\in\g$ the algebra $U(\g,e)^{\rm ab}$ has no
nonzero nilpotent elements?}
\end{question}
\subsection{}
Recall from [\cite{P07}, p.~524] that the centre $Z(\g)$ of the
universal enveloping algebra $U(\g)$ can be identified with the
centre of $U(\g,e)$ (this holds for any simple Lie algebra $\g$ and
any nilpotent element $e\in\g$). In [\cite{P02}, Rem.~2], the author
asked whether it is true that the centre of any factor-algebra
$\mathcal A$ of $U(\g,e)$ coincides with the image of $Z(\g)$ in
$\mathcal A$. The aim of this subsection is to show that the answer
to this question is negative already for ${\mathcal A}=U(\g,e)^{\rm
ab}$ and $\g=\mathfrak{gl}(4)$. We keep the notation introduced in
(\ref{5.8}).

The centre of $U(\g,e)$ was determined in [\cite{BrK2}] and
[\cite{BrB}]. Let $Z_1,\ldots, Z_N$ be the central elements  of
$U(\g,e)$ introduced in [\cite{BrB}, Sect.~3] and denote by
$z_1,\ldots, z_N$ their images in $U(\g,e)^{\rm ab}$. Set
$Z_0=z_0=1$ and define $Z(u):=\sum_{i=0}^N\,Z_iu^{N-i}$ and
$z(u):=\sum_{i=0}^N\,z_iu^{N-i}$, elements of $U(\g,e)[u]$ and
$U(\g,e)^{\rm ab}[u]$, respectively. From the explicit presentation
of $Z(u)$ given in [\cite{BrB}, Sect.~3] it follows that $z(u)$
equals the determinant of the diagonal matrix
$${\rm diag}\Big(u^{p_1}d_1(u),(u-1)^{p_2}d_2(u-1),\cdots,
(u-n+1)^{p_n}d_n(u-n+1)\Big).$$

Now suppose $N=4$ and $\lambda=(2,2)$. Then $n=2$ and $p_1=p_2=2$.
Combining [\cite{BrK2}, Thm.~3.5] with the equalities
$f_1^{(r)}=e_1^{(r)}=0,\ r\ge 1$ and $d_1^{(r)}=0,\,r>2$, it is not
hard to observe that $d_2^{(r)}=0$ for all $r> 2$. This implies that
$$z(u)\,=\,\big(u^2+d_1^{(1)}u+d_1^{(2)}\big)\big((u-1)^2+d_2^{(1)}(u-1)+d_2^{(2)}\big).$$
It was mentioned in (\ref{5.8}) that $\tilde{d}_i^{(1)}=d_i^{(1)}$
and $\tilde{d}_i^{(2)}=d_i^{(2)}-d_i^{(1)}d_i^{(1)}$ for $i=1,2$.
The proof of Theorem~{\ref{gln} shows that $U(\g,e)^{\rm
ab}\,=\,\CC[d_1^{(1)},d_1^{(2)}]$, whilst from (\ref{dd}) we get
$d_1^{(1)}+\tilde{d}_2^{(1)}=0$ and
$\tilde{d}_1^{(2)}+\tilde{d}_1^{(1)}d_2^{(1)}+d_2^{(2)}=0.$ This
yields $d_2^{(1)}=\tilde{d}_2^{(1)}=-d_1^{(1)}$ and $d_2^{(2)}
=-d_1^{(2)}.$ Setting $X:=d_1^{(1)}$ and $Y:=d_1^{(2)}$ we obtain
\begin{eqnarray*}
z(u)&=&\big(u^2+Xu+Y\big)\big((u-1)^2-X(u-1)-Y\big)\\
&=&\big(u^2+Xu+Y\big)\big((u^2-(X+2)u+(X-Y+1)\big)\\
&=&u^4-2u^3-\big(X^2+X-1\big)u^2\\
&+&\big(X^2-2XY-2Y+X\big)u +\big(XY-Y^2+Y\big).
\end{eqnarray*}

According to [\cite{BrB}], the image of $Z(\g)$ in $U(\g,e)$ is
generated by $Z_1,Z_2,Z_3, Z_4$. Suppose for a contradiction that
${\mathcal A}=U(\g,e)^{\rm ab}$ coincides with the image of $Z(\g)$
in $\mathcal A$. As $X^2-2XY-2Y+X=(X^2+X)-2Y(X+1)$, we then have the
equality
$$\CC[X,Y]\,=\,{\mathcal A}\,=\,\CC[z_1,z_2,z_3,z_4]\,=\,\CC[X^2+X, Y(X+1), Y(X-Y+1)].$$
It follows that $\CC[X,Y]/(Y)$ is generated by the image of $X(X+1)$
in $\CC[X,Y]/(Y)$. Since $\CC[X,Y]/(Y)\,\cong\,\CC[X]$, this is
impossible, however. This shows that the image of $Z(\g)$ in
$U(\g,e)^{\rm ab}$ is a proper subalgebra of $U(\g,e)^{\rm ab}$.
\section{\bf Generalised Whittaker models for primitive ideals}
\subsection{}\label{2.1}
We denote by $L(\lambda)$ the irreducible $\g$-module of highest
weight $\lambda\in\h^*$. Recall that $L(\lambda)$ is the simple
quotient of the Verma module
$M(\lambda):=U(\g)\otimes_{U(\h\,\oplus\,\n_+)\,}\CC
\tilde{v}_\lambda$, where $\CC \tilde{v}_\lambda$ is a
$1$-dimensional $(\h \oplus\n_+)$-module with $h\cdot
\tilde{v}_\lambda=\lambda(h)\tilde{v}_\mu$ for all $h\in\h$. Given a
primitive ideal $P$ of $U(\g)$ we write $\mathcal{VA}(P)$ for the
associated variety of $P$. The affine variety
$\mathcal{VA}(P)\subset\g^*$ is the zero locus of the $(\Ad
G)$-invariant ideal $\gr\, P$ of $S(\g)=\gr\, U(\g)$. By the
Irreducibility Theorem, $\mathcal{VA}(P)$ coincides with the Zariski
closure of a coadjoint nilpotent orbit in $\g^*$. By Duflo's
Theorem, $P={\rm Ann}_{U(\g)}\,L(\lambda)$ for some
$\lambda\in\h^*$. In general, such a $\lambda$ is not unique, but if
${\rm Ann}_{U(\g)}\,L(\lambda)\,=\,{\rm Ann}_{U(\g)}\,L(\lambda')$
then $\lambda'+\rho=w(\lambda+\rho)$ for some $w\in W$ (here $W=\la
s_\alpha\,|\,\alpha\in\Phi\ra$ is the Weyl group of $\g$ and
$\rho=\frac{1}{2}\sum_{\alpha\in\Phi^+}\alpha$ is the half-sum of
positive roots).

By [\cite{P07}, Thm.~3.1(ii)], if $P={\rm
Ann}_{U(\g)}\big(Q_\chi\otimes_{U(\g,\,e)}V\big)$ for some finite
dimensional irreducible $U(\g,e)$-module $V$, then
$\mathcal{VA}({\mathcal J})=\overline{\mathcal O}_\chi$, where
$\chi=(e,\,\cdot\,)$. A few years ago the author conjectured that
the converse should also be true; that is, for every primitive ideal
$P$ of $U(\g)$ with $\mathcal{VA}(P)=\overline{\mathcal O}_\chi$
there should exist a finite dimensional irreducible $U(\g,e)$-module
$V$ such that $P={\rm
Ann}_{U(\g)}\big(Q_\chi\otimes_{U(\g,\,e)}V\big)$; see [\cite{P07},
Conjecture~3.2]. In [\cite{P07'}], this conjecture was proved under
the assumption that the infinitesimal character of $P={\rm
Ann}_{U(\g)}\,L(\lambda)$ is rational, i.e.
$\la\lambda,\alpha^\vee\ra\in\mathbb Q$ for all $\alpha\in\Pi$. In
proving [\cite{P07}, Conjecture~3.2] under this assumption the
author relied almost entirely on characteristic $p$ methods.

In the meantime, two different proofs of the author's conjecture
have appeared in the literature. The first proof, based on
equivariant Fedosov quantisation, was obtained by Losev; see
[\cite{Lo}, Thm.~1.1.2(viii)]. The second proof, relying on
Harish-Chandra bimodules for quantised Slodowy slices, was recently
found by Ginzburg; see [\cite{Gi}, Thm.~4.5.1].

The main goal of this section is to revisit the proof of
[\cite{P07'}, Thm.~1.1] and make a few amendments sufficient for
confirming [\cite{P07}, Conjecture~3.2] in full generality (this was
announced [\cite{P07'}, p.~745]).
\subsection{}\label{2.2}
Given a Lie algebra $\mathcal L$ over a commutative ring $A$, which
is free as an $A$-module, we denote by $U_n({\mathcal L})$ the $n$th
component of the canonical filtration of the universal enveloping
algebra $U({\mathcal L})$. By the PBW theorem, the corresponding
graded algebra $\gr\,U({\mathcal L})$ is isomorphic to the symmetric
algebra $S({\mathcal L})$ of the free $A$-module $\mathcal L$. Given
a commutative Noetherian ring $R$ we write $\dim\,R$ for the Krull
dimension of $R$.

Let ${\mathcal I}\,=\,{\rm Ann}_{U(\g)}\,L(\mu)$ be a primitive
ideal of $U(\g)$ with $\mathcal{VA(I)}=\overline{\O}_\chi$. From now
on we shall always assume that our admissible ring $A$ contains all
elements $\la\mu,\alpha^\vee\ra$ with $\alpha\in\Pi$. In this case,
$M_A(\mu):=U(\n^{-}_A)\tilde{v}_\mu$ is a $\g_A$-stable $A$-lattice
in the Verma module $M(\mu)$ (here $\n^{-}_A$ stands for the
$A$-span of the $e_\gamma$ with $\gamma\in\n_-$).

Denote by $M^{\rm max}(\mu)$ the unique maximal submodule of
$M(\mu)$, so that $L(\mu)=M(\mu)/M^{\rm max}(\mu)$, and let $v_\mu$
be the image of $\tilde{v}_\mu$ under the canonical homomorphism
$M(\mu)\twoheadrightarrow L(\mu)$. Put $M_A^{\rm max}(\mu):=M^{\rm
max}(\mu)\cap M_A(\mu)$ and define
$$L_A(\mu)\,:=\,M_A(\mu)/M_A^{\rm max}(\mu).$$
Since $M_A(\mu)$ is a Noetherian $U(\g_A)$-module, so are $M^{\rm
max}(\mu)$ and $L_A(\mu)$. For $n\in\Z_+$, put
$L_n(\mu):=U_n(\g)v_\mu=U_n(\n^-)v_\mu$ and
$L_{A,n}(\mu):=U_n(\g_A)v_\mu=U_n(\n_A^-)v_\mu$, and let
$$\gr\,L(\mu)=\,\textstyle{\bigoplus_{n\ge
0}}\,\,L_n(\mu)/L_{n-1}(\mu)\quad\mbox{and}\quad
\gr\,L_A(\mu)=\,\textstyle{\bigoplus_{n\ge
0}}\,\,L_{A,n}(\mu)/L_{A,n-1}(\mu)$$ (here $L_{-1}(\mu)
=L_{A,-1}(\mu)=0$). Note that $\gr\,L(\mu)$ and $\gr\,L_A(\mu)$ are
generated by $v_\mu=\gr_0\,v_\mu$ as modules over $S(\g)=\gr\,U(\g)$
and $S(\g_A)=\gr\, U(\g_A)$, respectively.

We now define
$${\mathcal J}:={\rm Ann}_{S(\g)}\,\gr\,L(\mu)\,=\,{\rm
Ann}_{S(\g)}\,v_\mu\quad \mbox{and}\quad {\mathcal J}_A:={\rm
Ann}_{S(\g_A)}\,\gr\,L(\mu)\,=\,{\rm Ann}_{S(\g_A)}\,v_\mu.$$ These
are graded ideals of $S(\g)$ and $S(\g_A)$, respectively. Put
$${\mathcal R}:=S(\g)/{\mathcal J}\quad \mbox{and}\quad  {\mathcal
R}_A:=S(\g_A)/{\mathcal J}_A.$$ The zero locus of the ideal
${\mathcal J}\subset S(\g)$ in $\g^*$ is called the {\it associated
variety} of $L(\mu)$ and denoted by ${\mathcal V}_\g L(\mu)$. By a
result of Gabber, all irreducible components of the variety
${\mathcal V}_\g L(\mu)$ have dimension $d(e)$; see [\cite{P07'},
2.2] for more detail. In particular, $\dim\,{\mathcal R}=d(e)$.
Since ${\mathcal R}=\bigoplus_{n\ge 0}\,{\mathcal R}(n)$, where
${\mathcal R}(n)\cong L_n(\mu)/L_{n-1}(\mu)$, is a graded Noetherian
algebra with ${\mathcal R}(0)=\CC$, we have that
$d(e)=\dim\,{\mathcal R}=1+\deg\,P_{\mathcal R}(t)$, where
$P_{\mathcal R}(t)$ is the Hilbert polynomial of ${\mathcal R}$; see
[\cite{Eis}, Corollary~13.7].

First we note that ${\mathcal R}_A=\bigoplus_{n\ge 0}{\mathcal
R}_A(n)$ is a finitely generated graded $A$-algebra and all
${\mathcal R}_A(n)\cong L_{A,n}(\mu)/L_{A,n-1}(\mu)$ are finitely
generated $A$-modules. Also, $A\subset\CC$ is a Noetherian domain.
If $0\ne b\in A$, then standard properties of localisation
[\cite{Bo}, Ch.~II, 2.4] yield that ${\mathcal
J}_{A[b^{-1}]}={\mathcal J}\otimes_A A[b^{-1}]$ and
$${\mathcal R}_{A[b^{-1}]}=S(\g_{A[b^{-1}]})/{\mathcal J}_{A[b^{-1}]}\cong
\big(S(\g_A)\otimes_A A[b^{-1}]\big)/\big({\mathcal
J}_A\otimes_AA[b^{-1}]\big)\cong{\mathcal R}_A\otimes_A A[b^{-1}].$$

Denote by $F$ the quotient field of $A$. Since ${\mathcal
R}_F:={\mathcal R}_A\otimes_A F$ is a finitely generated algebra
over a field, the Noether Normalisation Theorem says that there
exist homogeneous, algebraically independent $y_1,\ldots,
y_d\in{\mathcal R}_F$, such that ${\mathcal R}_F$ is a finitely
generated module over its graded polynomial subalgebra
$F[y_1,\ldots, y_d]$; see [\cite{Eis}, Thm.~13.3]. Let $v_1,\ldots,
v_D$ be a generating set of the $F[y_1,\ldots, y_d]$-module
${\mathcal R}_F$ and let $x_1,\ldots, x_{m'}$ be a generating set of
the $A$-algebra ${\mathcal R}$. Then
\begin{eqnarray*}
v_i\cdot v_j&=&\textstyle{\sum}_{k=1}^D \,p_{i,j}^k(y_1,\ldots,
y_d)v_k\ \ \qquad(1\le
i,j\le D)\\
x_i&=&\textstyle{\sum}_{j=1}^{D}\,q_{i,j}(y_1,\ldots, y_d) v_j\ \
\,\qquad(1\le i\le m')
\end{eqnarray*} for some polynomials $p_{i,j}^k,\, q_{i,j}\in
F[X_1,\ldots, X_d].$ The algebra ${\mathcal R}_A$ contains an
$F$-basis of ${\mathcal R}_F$. The coordinate vectors of the
$x_i$'s, $y_i$'s and $v_i$'s relative to this basis and the
coefficients of the polynomials $q_{i,j}$ and $p_{i,j}^k$ involve
only finitely many scalars in $\mathbb Q$. Replacing $A$ by
$A[b^{-1}]$ for a suitable $0\ne b\in A$ if necessary, we may assume
that all $y_i$ and $v_i$ are in ${\mathcal R}_A$ and all $p_{i,j}^k$
and $q_{i,j}$ are in $A[X_1,\ldots,X_d]$. In conjunction with our
earlier remarks this shows that no generality will be lost by
assuming that
\begin{equation}\label{f-gen}{\mathcal R}_A\,=\,A[y_1,\ldots,
y_d]v_1+\cdots+A[y_1,\ldots,y_d]v_D
\end{equation} is a finitely generated module over the polynomial
algebra $A[y_1,\ldots, y_d]$. We may assume without loss that $D!$
is invertible in $A$.

\begin{lemma}\label{free}
There exists an admissible ring $A\subset \CC$ such that each graded
component ${\mathcal R}_A(n)$ of ${\mathcal R}_A$ is a free
$A$-module of finite rank.
\end{lemma}
\begin{proof}
Since ${\mathcal R}_A$ is a finitely generated $A[y_1,\ldots,
y_d]$-module and $A$ is a Noetherian domain, a graded version of the
Generic Freeness Lemma shows that there exists a nonzero $a\in A$
such that each  $\big({\mathcal R}_A(n)\big)[a^{-1}]$  is a free
$A[a^{-1}]$-module of finite rank; see (the proof of) Theorem~14.4
in [\cite{Eis}]. Since it follows from [\cite{Bo}, Ch.~II, 2.4] that
$\big({\mathcal R}_A(n)\big)[a^{-1}]\cong {\mathcal
R}_{A[a^{-1}]}(n)$ for all $n\in\Z_+$, the result follows.
\end{proof}
\subsection{}\label{2.4}
Denote by $L_F(\mu)$ the highest weight module $L_A(\mu)\otimes_A F$
over the split Lie algebra $\g_F$, where $F=\mathrm{Quot}(A)$. Since
$L(\mu)\cong L_F(\mu)\otimes_F\CC$, each subspace ${\mathcal I}\cap
U_n(\g)$ is defined over $F$. It follows that the graded ideal
$$\gr\,{\mathcal I}\,=\,\textstyle{\bigoplus}_{n\ge
0}\,\big({\mathcal I}\cap U_n(\g)\big)/\big({\mathcal I}\cap
U_{n-1}(\g)\big)\subset S(\g)$$ is defined over $F$ as well. Hence,
for every $n\in\mathbb Z_+$ the $F$-subspace
$S^n(\g_F)\cap\gr\,{\mathcal I}$ is an $F$-form of the graded
component $\gr_n\,{\mathcal I}\subset S^n(\g)$. Since $S(\g)$ is
Noetherian, the ideal $\gr\,\mathcal I$ is generated by its
$F$-subspace $\gr\,{\mathcal I}_{F,n'}:=\gr\,{\mathcal
I}\cap\bigoplus_{k\le n'}S^k(\g_F)$ for some $n'=n'(\mu)\in\Z_+$.
From this it follows that $\mathcal I$ is generated over $U(\g)$ by
its $F$-subspace ${\mathcal I}_{F,n'}:=\,U_{n'}(\g_F)\cap\mathcal
I$. Since $\mathcal I$ is a two-sided ideal of $U(\g)$, all
subspaces ${\mathcal I}\cap U_n(\g)$ and $\gr_n\,\mathcal I$ are
invariant under the adjoint action of $G$ on $U(\g)$. It follows
that the $F$-subspaces $\gr\,{\mathcal I}_{F,n'}$ and ${\mathcal
I}_{F,n'}$ are invariant under the adjoint action of the
distribution algebra $U_F:=U_\Z\otimes_\Z F$. Since
$\h_K:=\h\cap\g_F$ is a split Cartan subalgebra of $\g_F$, the
adjoint $\g_F$-modules $\gr\,{\mathcal I}_{F,n'}$ and ${\mathcal
I}_{F,n'}$ decompose into a direct sum of absolutely irreducible
$\g_F$-modules with integral dominant highest weights. Consequently,
these $\g_F$-modules possess $\Z$-forms invariant under the adjoint
action of the Kostant $\Z$-form $U_{\mathbb Z}$; we call them
$\gr\,{\mathcal I}_{\Z, n'}$ and ${\mathcal I}_{\Z,n'}$.

Let $\{\psi_i\,|\,\,i\in I\}$ be a homogeneous basis of the free
$\Z$-module $\gr\,{\mathcal I}_{\Z, n'}$ and let $\{u_i\,|\,\,i\in
I\}$ be any basis of the free $\Z$-module ${\mathcal I}_{\Z,n'}$.
Expressing the $u_i$ and $\psi_i$ via the PBW bases of $U(\g_F)$ and
$S(\g_F)$ associated with our Chevalley basis $\mathcal B$ involves
only finitely many scalars in $F$. Thus, no generality will be lost
by assuming that all $\psi_i$ are in $S(\g_A)$ and all $u_i$ are in
$U(\g_A)$.

Let $K$ be an algebraically closed field whose characteristic is a
good prime for the root system $\Phi$. Let
$\g_K=\g_{\Z}\otimes_{\Z}K$ and let $G_K$ be the simple, simply
connected algebraic $K$-group with hyperalgebra
$U_K:=U_{\Z}\otimes_{\Z}K$. Let ${\mathcal N}(\g)$ and ${\mathcal
N}(\g_K)$ denote the nilpotent cones of $\g$ and $\g_K$,
respectively. As explained in [\cite{P03}] and [\cite{P07'}, 2.5],
there are nilpotent elements $e_1,\ldots, e_t\in\g_{\Z}$ such that
\begin{itemize}
\item[(i)\,] $\{e_1,\ldots, e_t\}$ is a set of representatives for
$\,{\mathcal N}(\g)/G$;

\smallskip

\item[(ii)\,] $\{e_1\otimes 1,\ldots, e_t\otimes 1\}$ is a set of
representatives for $\,{\mathcal N}(\g_K)/G_K$;

\smallskip

\item[(iii)\,] $\,\dim_{\mathbb C}\,(\Ad G)e_i\,=\,\dim_K\,(\Ad
G_K)(e_i\otimes 1)\, $ for all $\,i\le t$.
\end{itemize}
For $1\le i\le t$ set $\chi_i:=(e_i,\,\cdot\,).$ As in
[\cite{P07'}], we assume that $e=e_k$ for some $k\le t$ and
${\mathcal O}(e_i)\subset \overline{{\mathcal O}(e)}$ for $i\le k$.
Since $\mathcal{VA}({\mathcal I})$ is the zero locus of $\gr\,
{\mathcal I}$ and $\gr\,\mathcal I$ is generated by the set
$\{\psi_i\,|\,\,i\in I\}$, we have that $\overline{{\mathcal
O}(\chi)}=\bigcap_{i\in I}\,V(\psi_i)$. It follows that the $\psi_i$
vanish on all $\chi_j$ with $j\le k$. Since all $\psi_i$ are in
$S(\g_A)$, all $e_j$ are in $\g_{\Z}$, and the form
$(\,\cdot\,,\,\cdot\,)$ is $A$-valued, we also have that
$\psi_i(\chi_j)\in A$. Localising further if necessary we may assume
that all nonzero $\psi_i(\chi_j)$ are invertible in $A$.
\subsection{}\label{3.1} Now suppose that $A$ satisfies all the conditions mentioned above.
Take $p\in\pi(A)$ and let $\nu\colon\,A\to {\mathbb F}_p$ be the
algebra homomorphism with kernel ${\mathfrak P}\in{\rm Specm}\,A$.
Write $\k$ for the algebraic closure of ${\mathbb F}_p$ and set
$L_\P(\mu):=L_A(\mu)\otimes_{A}\k$, where it is assumed that $A$
acts on $\k$ via $\nu$. Clearly, $L_\P(\mu)$ is a module over the
Lie algebra $\g_\k=\n^{-}_\k\oplus\h_\k\oplus\n^{+}_\k$, where
$\n^{\pm}_\k:=\n^{\pm}\otimes_A\k$ and $\h_\k:=\h_A\otimes_A\k$.
Furthermore, $\bar{v}_\mu:=v_\mu\otimes 1\in L_\P(\mu)$ is a highest
weight vector for the Borel subalgebra $\h_\k\oplus\n^+_\k$ of
$\g_\k$, and $L_\P(\mu)= U(\n^-_\k)\cdot \bar{v}_\mu$. Denote by
$\bar{\mu}$ the $\h_\k$-weight of $\bar{v}_\mu$. Since
$\mu(h_\alpha)\in A$ for all $\alpha\in \Pi$ and $\nu(a)\in{\mathbb
F}_p$ for all $a\in A$, we have that
$\bar{\mu}(\bar{h}_\alpha)\in{\mathbb F}_p$ for all $\alpha\in \Pi$.

Recall the notation and conventions of Section~2. Similar to
[\cite{P07'}, 3.1], we now set $I_\P(\mu):=\{z\in Z_p\,|\,\,z\cdot
\bar{v}_\mu=0\}$, an ideal of the $p$-centre $Z_p$ of $U(\g_\k)$,
and denote by $V_\P(\mu)$ the zero locus of $I_\P(\mu)$ in
$\g_\k^*$. It is immediate from the preceding remark that
$\bar{e}_\gamma^p\in I_\P(\mu)$ for all $\gamma\in\Phi^+$ and
${\bar{h}^p_\alpha}-\bar{h}_\alpha\in I_\P(\mu)$ for all
$\alpha\in\Pi$. Consequently,
\begin{eqnarray}\label{support} V_\P(\mu)\,\subseteq\,
\{\eta\in\g_\k^*\,|\,\,\,\eta(\h_\k)=\eta(\n^+_\k)=0\}.
\end{eqnarray}
As the $U(\g_\k)$-module $L_\P(\mu)$ is generated by $\bar{v}_\mu$,
we have that $I_\P(\mu)\,=\,{\rm Ann}_{Z_p}\,L_\P(\mu)$. Given
$\eta\in\g_\k^*$ we set $L_\P^\eta(\mu):=L_\P(\mu)/I_\eta\cdot
L_\P(\mu).$ By construction, $L_\P^\eta(\mu)$ is a $\g_\k$-module
with $p$-character $\eta$. It follows from (\ref{support}) that
every $\xi\in V_\P(\mu)$ has the form $\xi=(x,\cdot\,)$ for some
$x\in\n^+_\k$.
\begin{lemma}\label{le1}
If $\eta\in V_\P(\mu)$, then $L_\P^\eta(\mu)$ is a nonzero
$U_\eta(\g_\k)$-module.
\end{lemma}
\begin{proof}
Replace $L_p(\mu)$ by $L_\P(\mu)$ and $I_p(\mu)$ by $I_\P(\mu)$, and
argue as in the proof of [\cite{P07'}, Lemma~3.1].
\end{proof}

Set $L_{\P,n}(\mu):=U_n(\g_\k)\bar{v}_\mu$ and
$\gr\,L_\P(\mu):=\,\textstyle{\bigoplus_{n\ge
0}}\,\,L_{\P,n}(\mu)/L_{\P, n-1}(\mu),$ where $n\in\Z_+$. Note that
$\gr\,L_\P(\mu)$ is a cyclic $S(\g_\k)$-module generated by
$\bar{v}_\mu=\gr_0\, \bar{v}_\mu$. Also,
$$L_{\P,n}(\mu)=U_n(\g_k)\bar{v}_\mu=\big(U_n(\g_A)v_\mu\big)\otimes_A\k
=L_{A,n}(\mu)\otimes_A\k.$$ We put ${\mathcal J}_\P:={\rm
Ann}_{S(\g_\k)}\,\gr\,L_\P(\mu)={\rm Ann}_{S(\g_\k)}\,\bar{v}_\mu$
and ${\mathcal R}_\P:=S(\g_\k)/{\mathcal J}_\P$, and denote by
${\mathcal V}_\g L_\P(\mu)$ the zero locus of ${\mathcal J}_\P$ in
${\rm Specm}\,S(\g_\k)=\g_\k^*$. Since $\bar{v}_\mu$ is a highest
weight vector for $\h_\k\oplus\n_\k^+$, all linear functions from
${\mathcal V}_\g L_\P(\mu)$ vanish on $\h_\k\oplus\n_\k^+$.

By Lemma~\ref{free}, all graded components ${\mathcal R}_{A,n}\cong
L_{A,n}(\mu)/L_{A,n-1}(\mu)$ of ${\mathcal R}_A$ are free
$A$-modules of finite rank. From this it is immediate that so are
the $A$-modules $L_{A,n}$, and ${\mathcal R}_\P\cong {\mathcal
R}_A\otimes_A\k$ as graded $\k$-algebras. Comparing the Hilbert
polynomials of ${\mathcal R}={\mathcal R}_A\otimes_A\mathbb C$ and
${\mathcal R}_\P\cong {\mathcal R}_A\otimes_A\k$ we see that $\dim
{\mathcal R}_\P\,=\,\dim{\mathcal R}\,=\,d(e);$ see [\cite{Eis},
Corollary~13.7]. As a consequence,
\begin{eqnarray}\label{13}\label{Krull}\dim_\k {\mathcal V}_\g L_\P(\mu)\,=\,
\dim {\mathcal R}_\P=\,d(e).
\end{eqnarray}

Recall from (\ref{f-gen}) the generators $v_1,\ldots, v_D$ of the
$A[y_1,\ldots,y_d]$-module ${\mathcal R}_A$. Since ${\mathcal
R}={\mathcal R}_A\otimes_A\mathbb C$, the above discussion also
shows that $d=\dim\,{\mathcal R}=d(e)$. We stress that $D=D(\mu)$
depends on $\mu$, but not on $\P$.
\begin{lemma}\label{le2}
For every $\eta\in\g_\k^*$ we have that $\,\dim_\k L_\P^\eta(\mu)\le
Dp^{d(e)}$.
\end{lemma}
\begin{proof}
Repeat verbatim the proof of Lemma~3.2 in [\cite{P07'}].
\end{proof}
\subsection{}\label{3.4} Since $D!$
is invertible in $A$, we have that $p>D$ for all $p\in\pi(A)$. As
before, we identify $\g_\k$ with $\g_\k^*$ by using the
$G_\k$-equivariant map $x\mapsto (x,\cdot\,)$. Then
$V_\P(\mu)\subseteq \n^+_\k$; see (\ref{support}). The $p$-centre
$Z_p(\n_\k^-)=Z_p\cap U(\n_\k^-)$ of $U(\n_\k^-)$ is isomorphic to
the polynomial algebra in $\bar{e}_{\gamma}^p$, where
$\gamma\in\Phi^-$, hence can be identified with the subalgebra
$S(\n_\k^-)^p$ of all $p$-th powers in $S(\n_\k^-)$. Therefore, we
may regard $I_\P(\mu)\cap Z_p(\n_\k^-)$ as an ideal of the graded
polynomial algebra
$S(\n_\k^-)^p=\k[\bar{e}_{\gamma}^p\,|\,\,\gamma\in\Phi^-]$. It
follows from our discussion in (\ref{3.1}) and the above
identifications that
\begin{eqnarray}\label{VIP}
V_\P(\mu)\,=\,V(I_\P(\mu)\cap Z_p(\n_\k^-))\cap\n_\k^+.
\end{eqnarray}
Let $\gr\big(I_\P(\mu)\cap Z_p(\n_\k^-)\big)$ be the homogeneous
ideal of $S(\n_\k^-)^p$ spanned by the highest components of all
elements in $I_\P(\mu)\cap Z_p(\n_\k^-)$. From (\ref{VIP}) it
follows that the zero locus of $\gr\big(I_\P(\mu)\cap
Z_p(\n_\k^-)\big)$ in $\n_\k^+$ coincides with ${\mathbb
K}(V_\P(\mu))$, the associated cone to $V_\P(\mu)$ (associated cones
are defined in [\cite{BK1}, \S3], for instance). Since
$I_\P(\mu)\cap Z_p(\n_\k^-)$ is contained in ${\rm
Ann}_{Z_p}\,\bar{v}_\mu$, all elements of $\gr\big(I_\P(\mu)\cap
Z_p(\n_\k^-)\big)$ annihilate $\gr_0\,
\bar{v}_\mu\in\gr\,L_\P(\mu)$. Then $\gr\big(I_\P(\mu)\cap
Z_p(\n_\k^-)\big)\subset {\mathcal J}_\P\cap S(\n_\k^-)$, which
yields
\begin{eqnarray}\label{VA-VP}
{\mathcal V}_\g L_\P(\mu)\,=V({\mathcal J}_\P\cap
S(\n_\k^-))\cap\n_\k^+\,\subseteq\, {\mathbb K}(V_\P(\mu)).
\end{eqnarray}
\begin{theorem}\label{thm1}
Under the above assumptions on $A$, the variety $V_\P(\mu)$ contains
an irreducible component of maximal dimension which coincides with
the Zariski closure of an irreducible component of
$\n_\k^+\cap(\Ad\,G_\k)e$.
\end{theorem}
\begin{proof}
This is a slight generalisation of [\cite{P07'}, Thm.~3.1]. In view
of (\ref{VA-VP}) and (\ref{Krull}) one just needs to replace
$V_p(\mu)$ by $V_\P(\mu)$,  ${\mathcal V}_\g L_p(\mu)$ by ${\mathcal
V}_\g L_\P(\mu)$ and ${\mathcal J}_p$ by ${\mathcal J}_\P$, and
repeat the argument used in [\cite{P07'}].
\end{proof}

Recall from (\ref{2.4}) the generating set $\{u_i\,|\,\,i\in I\}$ of
the primitive ideal $\mathcal I$. By construction, $u_i\in U(\g_A)$
for all $i$ and the $A$-span of the $u_i$'s is invariant under the
adjoint action of $\g_A$. Let $\bar{u}_i$ be the image of $u_i$ in
$U(\g_\k)=U(\g_A)\otimes_A\k$. Clearly, the $\k$-span of the
$\bar{u}_i$'s is invariant under the adjoint action of $\g_\k$. Let
$\varphi_\chi\colon\, U(\g_A)\twoheadrightarrow
Q_{\chi,\,A}=U(\g_A)/N_{\chi,\,A}$ be the canonical homomorphism,
and denote by $\bar{\varphi}_\chi$ the induced epimorphism from
$U_{\chi}(\g_\k)$ onto $Q_{\chi}^\chi$; see (\ref{4.2}) and
Lemma~\ref{Q-eta}(i). By Lemmas~\ref{Q-eta} and \ref{lem4}, there
exists a finite subset $C$ of $\Z_+^{d(e)}$ such that
\begin{eqnarray}\label{u-h}\varphi_\chi(u_i)\,=\, \sum_{{\bf c}\in\, C}\,X^{\bf c}
h_{i,\,{\bf c}}(1_\chi)\qquad\qquad \ \, \big(h_{i,\,\bf c}\in
U(\g_A,e),\ \,i\in I\big).
\end{eqnarray}
On the other hand, it follows from Lemma~\ref{L1} that the
$\k$-algebra $U_\chi(\g_\k,e)$ is a homomorphic image of the
$\k$-algebra $U(\g_\k,e)$. Let $\bar{h}_{i,\,\bf c}$ denote the
image of $h_{i,\,\bf c}\otimes 1$ in $U_\chi(\g_\k,e)$.
From(\ref{u-h}) we get
\begin{eqnarray}\label{u-h-p}\bar{\varphi}_\chi(\bar{u}_i)\,=\, \sum_{{\bf c}\in\, C}\,
\bar{X}^{\bf c} \bar{h}_{i,\,{\bf
c}}(\bar{1}_{\bar{\chi}})\qquad\qquad \ (\forall\,i\in I).
\end{eqnarray}
Put $c:=\max_{{\bf c}\in\, C}|{\bf c}|$. From now on we shall assume
that $c!$ is invertible in $A$. This will ensure that the components
of all tuples in $C$ are smaller that any prime in $\pi(A)$.
\begin{prop}\label{prop2}
Under the above assumptions on $A$, for every $\P\in{\rm Specm}\, A$
with $A/\P\cong{\mathbb F}_p$ there is a positive integer
$k=k(\P)\le D=D(\mu)$ such that the algebra $U_\chi(\g_\k,e)$ has an
irreducible $k$-dimensional representation $\rho$ with the property
that $\rho(\bar{h}_{i,\,{\bf c}})=0$ for all ${\bf c}\in C$ and all
$i\in I$.
\end{prop}
\begin{proof}
Let $\P\in{\rm Specm}\,A$ be such that $A/\P\cong{\mathbb F}_p$. By
Lemma~\ref{le1} and Theorem~\ref{thm1}, there exists $g\in G_\k$
such that $L_\P^{g\,\cdot\,\chi}(\mu)\ne 0$, where $g\cdot\chi=({\rm
Ad}^*\,g)\chi$. By [\cite{P95}, Thm.~3.10] and Lemma~\ref{le2},
every composition factor $V$ of the $\g_\k$-module
$L_p^{g\,\cdot\,\chi}(\mu)$ has dimension $kp^{d(e)}$ for some
$k=k(V)\le D$. Since $u_i\in{\rm Ann}_{U(\g_A)}\,L_A(\mu)$ for all
$i\in I$, the elements $\bar{u}_i\in U(\g_\k)$ annihilate
$L_p(\mu)=L_A(\mu)\otimes_A\k$. Consequently, all $\bar{u}_i$
annihilate
$L_\P^{g\,\cdot\,\chi}(\mu)=L_\P(\mu)/I_{g\cdot\,\chi}L_\P(\mu)$,
and hence $V$.

Since $(\Ad\,g)(I_\chi)=I_{g\cdot\,\chi}$, the map
$\Ad\,g\colon\,U(\g_\k)\rightarrow\,U(\g_\k)$ gives rise to an
algebra isomorphism  $U_\chi(\g_\k)\stackrel{\sim}{\longrightarrow}
U_{g\cdot\,\chi}(\g_\k)$. Let $V'=\{v'\,|\,\,\,v\in V\}$, a vector
space copy of $V$. Give $V'$ a $\g_\k$-module structure by setting
$x\cdot v'\,:=\,((\Ad g)^{-1}x \cdot v)'$ for all $x\in\g_\k$ and
$v'\in V'.$ Since all elements $((\Ad g)x)^p-((\Ad
g)x)^{[p]}-\chi(x)^p1$ annihilate $V$, the $\g_\k$-module $V'$ has
$p$-character $\chi$. Furthermore, all elements $(\Ad g)\bar{u}_i$
annihilate $V'$. The $\Z$-span of $\{u_i\,|\,\,i\in I\}$ is
invariant under the adjoint action of $U_\Z$ on $U(\g_\Z)$; see
(\ref{2.4}). Since $U_\Z\otimes_\Z\k$ is the hyperalgebra of $G_\k$,
the $\k$-span of the $\bar{u}_i$'s is invariant under the adjoint
action of $G_\k$ on $U(\g_\k)$. In conjunction with our preceding
remark this implies that $\bar{u}_i\in{\rm Ann}_{U(\g_\k)}\,V'$ for
all $i\in I$. Let
$$V'_0\,=\,\{v'\in V'\,|\,\,\,x\cdot v'=\,\chi(x)v'\ \, \mbox{for all
}\,x\in\m_\k\}.$$ Since $U_\chi(\g_\k,e)\,\cong
\big(U_{\chi}(\g_\k)/U_{\chi}N_{\chi,\,\k}\big)^{\ad\m_\k}$ by
Lemma~\ref{Q-eta}(ii), the algebra $U_\chi(\g_\k,e)$ acts on $V'_0$.
Since $\m_\k$ is a $\chi$-admissible subalgebra of dimension $d(e)$
in $\g_\k$, it follows from [\cite{P02}, Thm.~2.4] that $V'_0$ is an
irreducible $k$-dimensional $U_\chi(\g_\k,e)$-module. We let $\rho$
stand for the corresponding representation of $U_\chi(\g_\k,e)$.

Denote by $V''$ the $U_\chi(\g_\k)$-module
$Q_{\chi}^\chi\otimes_{\,U_\chi(\g_\k,\,e)}V_0'$ and let
$v_1',\ldots, v_k'$ be a basis of $V_0'$. It follows from
Lemma~\ref{lem4} that the vectors $\bar{X}^{\bf a}\otimes v_j'$ with
$0\le a_i\le p-1$ and $1\le j\le k$ form a basis of $V''$ over $\k$.
Since $V'$ is an irreducible $\g_\k$-module, there is a
$\g_\k$-module epimorphism $\tau\colon\,V''\twoheadrightarrow V'$
sending $v'\otimes 1$ to $v'$ for all $v'\in V_0'$. Since $\dim_\k
V'=kp^{d(e)}$, the map $\tau$ is an isomorphism. Let $\tilde{\rho}$
stand for the representation of $U_{\chi}(\g_\k)$ in $\End_\k V''$.
As $N_{\chi,\,\k}$  annihilates $V'_0\otimes 1\subseteq V''$, it
follows from (\ref{u-h-p}) that
$$
0\,=\,\tilde{\rho}(\bar{u}_i)(v'\otimes
1)\,=\,\tilde{\rho}(\bar{\varphi}_\chi(\bar{u}_i))(v'\otimes 1)
\,=\,\,\textstyle{\sum}_{{\bf c}\in\, C}\,\,\bar{X}^{\bf
c}\otimes\rho(\bar{h}_{i,\,{\bf c}})(v')
$$
for all $v'\in V'_0$. As the nonzero vectors of the form
$\bar{X}^{\bf c}\otimes\rho(\bar{h}_{i,\,{\bf c}})(v')$ with $v'$
fixed are linearly independent by our assumption on $A$, we see that
$\rho(\bar{h}_{i,\,{\bf c}})=0$ for all ${\bf c}\in C$ and all $i\in
I$. This completes the proof.
\end{proof}
\subsection{}\label{6.6}
By our discussion in (\ref{4.2'}), there are polynomials
$H_{i,\,{\bf c}}\in A[X_1,\ldots, X_r]$ such that $h_{i,\,{\bf
c}}=H_{i,\,{\bf c}}(\Theta_1,\ldots, \Theta_r)$ for all ${\bf c}\in
C$ and $i\in I$. Let ${\mathcal I}_{\mathcal W}$ be the two-sided
ideal of $U(\g,e)$ generated by the $h_{i,\,{\bf c}}$'s. In view of
(\ref{relations}) and [\cite{P07'}, Lemma~4.1], the algebra
$U(\g,e)/{\mathcal I}_{\mathcal W}$ is isomorphic to the quotient of
the free associative algebra ${\mathbb C}\la X_1,\ldots, X_r\ra$ by
its two-sided ideal generated by all elements
$[X_i,X_j]-F_{ij}(X_1,\ldots, X_r)$ with $1\le i<j\le r$ and all
elements $H_{{\bf c},\,l}(X_1,\ldots, X_r)$ with ${\bf c}\in C$ and
$l\in I$. Given a natural number $d$ we denote by ${\mathcal M}_d$
the set of all $r$-tuples $(M_1,\ldots, M_r)\in \Mat_d({\mathbb
C})^r$ satisfying the relations
\begin{eqnarray*}
[M_i,M_j]-F_{ij}(M_1,\ldots, M_r)&=&0\qquad\quad (1\le i<j\le r)\\
H_{{\bf c},\,l}(M_1,\ldots, M_r)&=&0 \qquad\quad ({\bf c}\in C,\,\,
l\in I)
\end{eqnarray*}
(the monomials in $M_1,\ldots,M_r$ involved in $F_{ij}(M_1,\ldots,
M_r)$ and $H_{{\bf c},\,l}(M_1,\ldots, M_r)$ are evaluated by using
the matrix product in $\Mat_d(\CC)$). The preceding remark shows
that ${\mathcal M}_d$ is nothing but the variety of all matrix
representations of degree $d$ of the algebra $U(\g,e)/{\mathcal
I}_{\mathcal W}$.
\begin{lemma}\label{max-ideals}
The set $\pi(A)$ of all primes $p$ such that $A/\P\cong {\mathbb
F}_p$ for some $\P\in{\rm Specm}\,A$ is infinite for any finitely
generated $\Z$-subalgebra $A$ of $\CC$.
\end{lemma}
\begin{proof}
By Hilbert's Nullstellensatz, there is an algebra homomorphism
$A\to\overline{\mathbb Q}$. Thus, in proving the lemma we may assume
that $A\subset\overline{\mathbb Q}$. Then $A$ is a finitely
generated $\Z$-subalgebra of an algebraic number field $K={\mathbb
Q}[X]/(f)$, where $f\in\Z[X]$ is a polynomial of positive degree
irreducible over $\mathbb Q$. Then $A\subseteq \Z[b^{-1}][X]/(f)$
for some $b\in\Z^\times$. Since $b$ has only finitely many prime
divisors, we may assume without loss of generality that
$A=\Z[X]/(f)$ and $\deg f>1$.

Given $x\in\mathbb R$ denote by $\pi(x)$ the number of primes $\le
x$. If $p$ is a prime, let $N_p(f)$ be the number of zeros of $f$ in
${\mathbb F}_p=\Z/p\Z$. As explained in [\cite{Serre}], for
instance, it follows from Burnside's Lemma and Chebotarev's Density
Theorem that
\begin{equation}\label{lim}
\lim_{x\to\infty}\,\frac{\textstyle{\sum}_{p\le
\,x}\,N_p(f)}{\pi(x)}\,=\,1.
\end{equation}
Because $A=\Z[X]/(f)$, the set $\pi(A)$ consists of all primes $p$
with $N_p(f)\ne 0$. In view of (\ref{lim}) this implies that
$|\,\pi(A)|=\infty$.
\end{proof}

\subsection{}\label{new} Let ${\bf J}_d$ be the ideal of ${\mathcal P}:={\mathbb
C}\,[x_{ab}^{(k)}\,|\,\,1\le a,b\le d,\, 1\le k\le r]$ generated by
the matrix coefficients of all $[M_i,M_j]-F_{ij}(M_1,\ldots, M_r)$
and $H_{{\bf c},\,l}(M_1,\ldots, M_r)$, where $M_k$ is the generic
matrix $\big(x_{ab}^{(k)}\big)_{1\le a,b\le d}$. Note that
${\mathcal M}_d$ is nothing but the zero locus of ${\mathbf J}_d$ in
${\rm Specm}\,{\mathcal P}=\,{\mathbb A}^{rd^2}({\mathbb C})$. In
particular, ${\mathcal M}_d$ as a Zariski closed subset of ${\mathbb
A}^{rd^2}({\mathbb C})$. As all $F_{ij}$ and $H_{{\bf c},\,l}$ are
in $A[X_1,\ldots, X_r]$, the ideal ${\bf J}_d$ is generated by a
finite set of polynomials in ${\mathcal
P}_A=A[x_{ab}^{(k)}\,|\,\,1\le a,b\le d,\, 1\le k\le r]$, say
$\{f_1,\ldots, f_N\}$. Given $g\in{\mathcal P}_A$ and an algebra
homomorphism $\nu\colon\,A\to {\mathbb F}_p$, we write $^{\nu\!}{g}$
for the image of $g$ in ${\mathcal P}_A\otimes_A(A/\ker\,\nu)\subset
{\mathcal P}_A\otimes_A\overline{\mathbb F}_p$ and denote by
${\mathcal M}_d(\overline{\mathbb F}_p)$ the zero locus of
$^{\nu\!}f_1,\ldots, {^{\nu\!}f_N}$ in ${\mathbb
A}^{rd^2}({\overline{\mathbb F}_p})$.
\begin{prop}\label{prop3}
The algebra $U(\g,e)/{\mathcal I}_{\mathcal W}$ has an irreducible
representation of dimension at most $D=D(\mu)$.
\end{prop}
\begin{proof}
We need to show that ${\mathcal M}_d(\CC)\ne\varnothing$ for some
$d\le D$. Suppose this is not the case. Then
$g_1f_1+\cdots+g_Nf_N=1$ for some $g_1,\ldots,g_N\in \mathcal P$.
Let $B$ be the $A$-subalgebra of $\CC$ generated by the coefficients
of $g'_1,\ldots,g'_N$. By Lemma~\ref{max-ideals}, the set $\pi(B)$
is infinite. Take $p\in\pi(B)$ and let
$\nu\colon\,B\twoheadrightarrow{\mathbb F}_p$ be an algebra map such
that $B/\ker\,\nu\,\cong\, {\mathbb F}_p$. Denote by $\tilde{\nu}$
the composite
$${\mathcal P}_B\,\twoheadrightarrow\, {\mathcal
P}_B/({\ker\,\nu}) {\mathcal P}_B\,\hookrightarrow\,
\overline{\mathbb F}_p[x_{ab}^{(k)}\,|\,\,1\le a,b\le d,\, 1\le k\le
r]\,\cong\,{\mathcal P}_A\otimes_A\overline{\mathbb F}_p.$$ Since
$\tilde{\nu}(F)={^{\nu\!}}F$ for all $F\in{\mathcal P}_B$, we have
that ${{^\nu\!}}g_1 \,{{^\nu\!}}f_1+\cdots+ {{^\nu\!}}g_N\,
{^{\nu\!}}f_N=1.$ But then ${\mathcal M}_d(\overline{\mathbb
F}_p)=\varnothing$ for all $d\le D$. Since this contradicts
Proposition~\ref{prop2}, we conclude that ${\mathcal M}_d(\CC)\ne
\varnothing$ for some $d\le D$.
\end{proof}

We are ready to prove the main results of this section.
\begin{theorem}\label{main}
For any primitive ideal $\mathcal I$ of $U(\g)$ with
$\mathcal{VA}({\mathcal I})\,=\,\overline{\mathcal O}_\chi$ there is
a finite dimensional irreducible $U(\g,e)$-module $V$ such that
${\mathcal I}\,=\,{\rm Ann}_{U(\g)}(Q_\chi\otimes_{U(\g,\,e)} V)$.
\end{theorem}
\begin{proof}
By Proposition~\ref{prop3}, there is an irreducible finite
dimensional representation $\rho\colon\,U(\g,e)\to\End V$ such that
${\mathcal I}_{\mathcal W}\subseteq \ker\,\rho$. Associated with
$\rho$ is a representation of $U(\g)$ in ${\rm End}\big(
Q_\chi\otimes_{U\g,\,e)}V\big)$; call it $\tilde{\rho}$. It follows
from Skryabin's theorem [\cite{Sk}] and [\cite{P07}, Thm.~3.1(ii)]
that $\ker\tilde{\rho}$ is a primitive ideal of $U(\g)$ with
$\mathcal{VA}(\ker\tilde{\rho})=\overline{\mathcal{O}}_\chi$. From
(\ref{u-h}) it follows that $$\tilde{\rho}(u_i)(1_\chi\otimes
v)\,=\,\tilde{\rho}(\varphi_\chi(u_i))(1_\chi\otimes v) \,=\,
\textstyle{\sum}_{{\bf c}\in\, C}\,X^{\bf c}\otimes \rho(h_{i,\,{\bf
c}})(v)$$ for all $v\in V$ and $i\in I$. Since ${\mathcal
I}_{\mathcal W}\subseteq \ker\,\rho$, all $\tilde{\rho}(u_i)$
annihilate $1_\chi\otimes V\subset\widetilde{V}$. Since
$1_\chi\otimes V$ generates the $\g$-module
$Q_\chi\otimes_{U(\g,\,e)} V$ and the span of the $u_i$'s is stable
under the adjoint action of $\g$, we have that
$u_i\in\ker\tilde{\rho}$ for all $i\in I$. Since the $u_i$'s
generate the ideal $\mathcal I$, it must be that ${\mathcal
I}\subseteq \ker\tilde{\rho}$. Since the primitive ideals $\mathcal
I$ and $\ker\tilde{\rho}$ have the same associated variety, applying
[\cite{BK}, Korollar~3.6] gives $\mathcal{I}=\ker\tilde{\rho}$.
\end{proof}
\subsection{}\label{last'}
A more invariant definition of the algebra $U(\g,e)$ was given by
Gan--Ginzburg in [\cite{GG}]. Let $\n_\chi=\bigoplus_{i\le
-1}\,\g(i)$ and $\n_\chi':=\bigoplus_{i\le-2}\g(i)$, and denote by
$\widehat{Q}_\chi$ the Kazhdan-filtered $\g$-module
$U(\g)/U(\g)N_\chi'$, where $N_\chi'$ is the left ideal of $U(\g)$
generated by all $x-\chi(x)$ with $x\in \n_\chi'$. Note that
$\widehat{Q}_\chi$ is a $U(\n_\chi)$-bimodule and
$\widehat{Q}_\chi^{\,\ad\,\n_\chi}$ carries a natural algebra
structure. By [\cite{GG}], the algebra
$\widehat{Q}_\chi^{\,\ad\,\n_\chi}$ is canonically isomorphic to
$U(\g,e)$. Denote by $\widehat{\varphi}_\chi$ and $\varphi_\m$ the
canonical projections $U(\g)\twoheadrightarrow \widehat{Q}_\chi$ and
$\widehat{Q}_\chi\twoheadrightarrow Q_\chi$, respectively. The
adjoint action of $G$ on $U(\g)$ gives rise to a rational action of
the reductive part $C(e)=G_e\cap G_f$ of the centraliser $G_e$ on
$\widehat{Q}_\chi$. Clearly, the $\g$-module map
$\widehat{\varphi}_\chi$ is $C(e)$-equivariant and
$\varphi_\m\circ\widehat{\varphi}_\chi=\varphi_\chi$.

Recall from (\ref{4.1}) the Witt basis
$\{z_1',\ldots,z_s',z_1,\ldots,z_s\}$ of $\g(-1)$ and write $Z'^{\bf
b}$ for the monomial $z_1'^{b_1} \cdots z_s'^{b_s}\in U(\g)$, where
${\bf b}=(b_1,\ldots, b_s)\in\Z_+^s$.  Let $\hat{1}_\chi$ be the
image of $1$ in $\widehat{Q}_\chi$. Arguing as in [\cite{Sk}] it is
easy to observe that the monomials $X^{\bf a}Z'^{\bf
b}(\hat{1}_\chi)$ with ${\bf a}\in \Z_+^{d(e)}$ and ${\bf b}\in
\Z_+^s$ form a free basis of the right $U(\g,e)$-module
$\widehat{Q}_\chi$. Note that for any $h_{{\bf a},{\bf b}}\in
U(\g,e)$ we have that $\widehat{\varphi}_\chi\big(X^{\bf a}Z'^{\bf
b}h_{{\bf a},{\bf b}}(\hat{1}_\chi)\big)=X^{\bf a}h_{{\bf a},{\bf
b}}(1_\chi)$ if ${\bf b}={\bf 0}$ and $0$ otherwise.
\begin{lemma}\label{annil}
Let $M$ be any $U(\g,e)$-module and $u\in{\rm
Ann}_{U(\g)}\big(Q_\chi\otimes_{U(\g,\,e)}M\big)$. Then
$\widehat{\varphi}_\chi(u)\,=\,\sum_{{\bf a},\,{\bf b}}\,X^{\bf
a}Z'^{\bf b}h_{{\bf a},{\bf b}}(\hat{1}_\chi)$ for some $h_{{\bf
a},{\bf b}}\in{\rm Ann}_{U(\g,\,e)}\,M$.
\end{lemma}
\begin{proof}
Set $\Omega(u)=\{({\bf a},{\bf
b})\in\Z_+^{d(e)}\times\Z_+^s\,|\,\,\,h_{{\bf a},{\bf b}}\not\in{\rm
Ann}_{U(\g,\,e)}\,M\}$ and denote by $\Omega_{\max}(u)$ the set of
all $({\bf a},{\bf b})\in\Omega(u)$ for which the Kazhdan degree of
$X^{\bf a}\in U(\g)$ is maximal possible. Suppose
$\Omega_{\max}\neq\varnothing$ and denote bt $\Delta(u)$ the set of
all ${\bf b}\in {\rm pr}_2(\Omega_{\max}(u))$ for which $h_{{\bf
a},{\bf b}}\not\in{\rm Ann}_{U(\g,\,e)}\,M$ (here ${\rm pr}_2$ is
the second projection $\Z_+^{d(e)}\times\Z_+^s\twoheadrightarrow
\Z_+^s$). Order the elements in  $\Z_+^s$ lexicographically and
denote by $\bf m$ the largest element in $\Delta(u)$. Let ${\bf
a}_1,\ldots,{\bf a}_l$ be all elements in $\Z_+^{d(e)}$ for which
$({\bf a}_i,{\bf m})\in\Omega_{\max}(u)$.

Set $u':=\prod_{i=1}^s(\ad z_i)^{m_i}(u)$, an element of ${\rm
Ann}_{U(\g)}\big(Q_\chi\otimes_{U(\g,\,e)}M\big)$. Since
$\widehat{Q}_\chi$ is a Kazhdan-filtered $\g$-module, we have that
$\Omega_{\max}(u')\subseteq \Omega_{\rm max}(u)$, while it is
immediate from the definition of $\{{\bf a}_1,\ldots,{\bf a}_l\}$
that $({\bf a}_i,{\bf 0})\in \Omega_{\max}(u')$ for all $i\le l$.
Furthermore,
$$\widehat{\varphi}_\chi(u')\,=\,\textstyle{\sum}_{i=1}^l\,\lambda_i X^{{\bf a}_i}h_{{\bf a}_i,{\bf m}}(\hat{1}_\chi)+
\textstyle{\sum}_{({\bf a},{\bf b})\not\in\,\Omega_{\max}(u')}
X^{\bf a}Z'^{\bf b}h'_{{\bf a},{\bf b}}(\hat{1}_\chi),\qquad\
\lambda_i\ne 0,$$ by our choice of $\bf m$. Hence
$(\varphi_\m\circ\hat{\varphi}_\chi)(u')=\sum_{i=1}^l\,\lambda_i
X^{{\bf a}_i}h_{{\bf a}_i,{\bf m}}(1_\chi)+\sum_{{\bf a}\ne{\bf
a}_0}X^{\bf a}h'_{\bf a}(1_\chi)$ for some $h'_{\bf a}\in U(\g,e)$.
As $h_{{\bf a}_i,{\bf m}}\not\in{\rm Ann}_{U(\g,\,e)}\,M$ and
$\lambda_i\ne 0$ for all $i\le l$, we obtain $u'\not\in{\rm
Ann}_{U(\g)}\big(Q_\chi\otimes_{U(\g,\,e)}M\big)$, a contradiction.
This completes the proof.
\end{proof}
\begin{corollary}\label{ideal}
Let $M$ be as in Lemma~\ref{annil} and denote by ${\mathfrak I}_M$
the $U(\g)$-submodule of $\widehat{Q}_\chi$ generated by ${\rm
Ann}_{U(\g,\,e)}\,M\subseteq \widehat{Q}_\chi^{\,\ad \n_\chi}$. Then
$${\rm Ann}_{U(\g)}\big(Q_\chi\otimes_{U(\g,\,e)}M\big)\,=\,
\bigcap_{g\in G}\,(\Ad
g)\big(\widehat{\varphi}_\chi^{\,-1}({\mathfrak I}_M)\big).$$
\end{corollary}
\begin{proof}
Let $I={\rm Ann}_{U(\g)}\big(Q_\chi\otimes_{U(\g,\,e)}\,M\big)$ and
$I'=\bigcap_{g\in G}\,(\Ad
g)\big(\widehat{\varphi}_\chi^{\,-1}({\mathfrak I}_M)\big).$ It
follows from Lemma~\ref{annil} that
$I\subseteq\widehat{\varphi}_\chi^{\,-1}({\mathfrak I}_M)$. Since
$I$ is a two-sided ideal of $U(\g)$, it is invariant under the
adjoint action of $G$. Hence $I\subseteq I'$. On the other hand,
$I'$ is a left ideal of $U(\g)$ contained in
$\widehat{\varphi}_\chi^{\,-1}({\mathfrak I}_M)$ and invariant under
the adjoint action of $G$. Therefore, $I'$ is $(\ad \g)$-stable and
annihilates the subspace $1_\chi\otimes M$ of
$Q_\chi\otimes_{U(\g,\,e)} M$ (one should keep in mind that
$\n_\chi'\subseteq \m$). Since the latter generates the $\g$-module
$Q_\chi\otimes_{U(\g,\,e)} M$, we deduce that $I=I'$.
\end{proof}

Since $C(e)$ stabilises both $\n_\chi$ and $\n_\chi'$, it acts on
$U(\g,e)=\widehat{Q}_\chi^{\,\ad \n_\chi}$ as algebra automorphisms;
see [\cite{P07}, 2.1] for more detail. Thus, we can twist the module
structure $U(\g,e)\times M\rightarrow M$ of any $U(\g,e)$-module $M$
by an element $g\in C(e)$ to obtain a new $U(\g,e)$-module, $M^g$,
with underlying vector space $M$ and the $U(\g,e)$-action given by
$u\cdot m=g(u)\cdot m$ for all $u\in U(\g,e)$ and $m\in M$. Since
the map $\widehat{\varphi}_\chi$ is $C(e)$-equivariant and
$g({\mathfrak I}_M)={\mathfrak I}_{M^g}$, it follows from
Lemma~\ref{ideal} that
\begin{eqnarray}\label{C(e)}
\qquad\quad{\rm
Ann}_{U(\g)}\big(Q_\chi\otimes_{U(\g,\,e)}M\big)\,=\,{\rm
Ann}_{U(\g)}\big(Q_\chi\otimes_{U(\g,\,e)}M^g\big)\qquad\
\,\big(\forall \,g\in C(e)\big).
\end{eqnarray}
\subsection{}\label{last}
As explained in [\cite{P02}, 6.2] and [\cite{P07}, p.~524], the
centre $Z(\g)$ of $U(\g)$ maps isomorphically onto the centre of
$U(\g,e)$. Thus, we may identify $Z(\g)$ with the centre of
$U(\g,e)$. Given an algebra map $\lambda\colon\,Z(\g)\to\CC$ denote
by ${\rm Irr}_\lambda \,U(\g,e)$ the set of all isoclasses of finite
dimensional irreducible $U(\g,e)$-modules with central character
$\lambda$. As we recalled in (\ref{last'}), the reductive part
$C(e)=G_e\cap G_f$ of the centraliser $G_e$ acts on $U(\g,e)$ as
algebra automorphisms. Since $\Ad G$ acts trivially on $Z(\g)$, the
group $C(e)$ acts on each set ${\rm Irr}_\lambda\,U(\g,e)$.

By [\cite{P07}, Sect.~2], the Lie algebra $\g_e(0)$ of $C(e)$ embeds
into $U(\g,e)$ in such a way that the adjoint action of
$\g_e(0)\subset U(\g,e)$ on $U(\g,e)$ coincides with the
differential of the above-mentioned action of $C(e)$ on $U(\g,e)$.
This implies that twisting the module structure $U(\g,e)\times M\to
M$ of a finite dimensional $U(\g,e)$-module $M$ by an element of the
connected component of $C(e)$ does not affect the isomorphism type
of $M$. We thus obtain, for any $d\in\N$, a natural action of the
component group $\Gamma(e)=G_e/G_e^\circ\,\cong\, C(e)/C(e)^\circ$
on the set of all isoclasses of $d$-dimensional $U(\g,e)$-modules.
By the same token, $\Gamma(e)$ acts on each set ${\rm
Irr}_\lambda\,U(\g,e)$.

Let $\mathfrak X$ be the primitive spectrum of $U(\g)$ and denote by
${\mathfrak X}^\lambda$ the set of all $I\in\mathfrak X$ with $I\cap
Z(\g)=\ker\lambda$. Given a coadjoint nilpotent orbit
$\O\subset\g^*$ we write ${\mathfrak X}_\O$ for the set of all
$I\in\mathfrak X$ with $\mathcal{VA}(I)=\overline{\O}$, and set
${\mathfrak X}_\O^\lambda\,:=\,{\mathfrak X}^\lambda\cap{\mathfrak
X}_\O$. It follows from Theorem~\ref{main} that for any algebra
homomorphism $\lambda\colon\,Z(\g)\to\CC$ the map
$$\psi_\lambda\colon\,{\rm Irr}_\lambda\,U(\g,e)\longrightarrow\,{\mathfrak
X}_{\O(\chi)}^\lambda,\qquad\quad [V]\longmapsto\,{\rm
Ann}_{U(\g)}\big(Q_\chi\otimes_{U(\g,\,e)}V\big)$$ is surjective
(here $[V]$ stands for the isomorphism class of a $U(\g,e)$-module
$V$ and $\O(\chi)=({\rm Ad}^*\,G)\chi$). For any finite dimensional
irreducible $U(\g,e)$-module $M$, the $\g$-modules
$Q_\chi\otimes_{U(\g,\,e)} M^g$, where $g\in C(e)$, have the same
annihilator in $U(\g)$; see (\ref{C(e)}). From this it is immediate
that that all fibres of $\psi_\lambda$ are $\Gamma(e)$-stable.

In his talk at the MSRI workshop on Lie Theory in March 2008, the
author conjectured that $\Gamma(e)$ acts transitively on the fibres
of $\psi_\lambda$; that is, the fibres of $\psi_\lambda$ are
precisely the $\Gamma(e)$-orbits in ${\rm Irr}_\lambda\,U(\g,e)$.
This conjecture was known to hold in some special cases; see
[\cite{P07'}] and [\cite{BrK2}]. Very recently, the author's
conjecture was proved in full generality by Losev; see [\cite{Lo1},
Thm.~1.2.2]. We would like to finish this paper by putting on record
the following interesting consequence of Losev's result. For $\g$
semisimple, it solves an old problem posed by Borho and Dixmier in
the early 70s; see [\cite{Dix}, Problem~2].
\begin{theorem}\label{last}
For any complex semisimple Lie algebra $\g$ the primitive spectrum
of $U(\g)$ is a countable disjoint union of quasi-affine algebraic
varieties.
\end{theorem}
\begin{proof}
Let $\g_1,\ldots,\g_k$ be the simple ideals of the Lie algebra $\g$.
Let $I$ be a primitive ideal of $\g$ and set $I_j:=I\cap U(\g_j)$,
where $1\le j\le k$. Since $I$ is the annihilator in $U(\g)$ of a
simple highest weight module, by Duflo's Theorem, it is
straightforward to see that each $I_j$ is a primitive ideal of
$U(\g_j)$ and $I=\sum_{j=1}^k\,U(\g)I_j$. From this it is immediate
that the primitive spectrum of $U(\g)$ is the direct product of the
primitive spectra of the $U(\g_j)$'s. Thus, in proving the theorem
we may assume that $\g$ is simple. Since there are finitely many
coadjoint nilpotent orbits in $\g^*$, it suffices to show that
${\mathfrak X}_{\O(\chi)}$ is a countable union of quasi-affine
algebraic varieties.

The group ${\rm GL}(d)$ acts on $\Mat_d(\CC)^r$ by simultaneous
conjugations and preserves its Zariski closed subset ${\mathcal
M}_d$ defined in (\ref{6.6}). Since ${\rm GL}(d)$ is a reductive
group, the invariant algebra $\CC[{\mathcal M}_d]^{{\rm GL}(d)}$ is
finitely generated and the $\CC$-points of the affine variety
$R_d:={\rm Specm}\big(\CC[{\mathcal M}_d]^{{\rm GL}(d)}\big)$
parametrise the closed ${\rm GL}(d)$-orbits in ${\mathcal M}_d$.
Moreover, the morphism $\pi_d\colon\,{\mathcal M}_d\to R_d$ induced
by inclusion $\CC[{\mathcal M}_d]^{{\rm GL}(d)}\hookrightarrow
\CC[{\mathcal M}_d]$ is surjective and takes the ${\rm
GL}(d)$-stable closed subsets of ${\mathcal M}_d$ to closed subsets
of $R_d$; see [\cite{Kr}, Ch.~II, 3.2], for example. It follows from
Procesi's results on invariants of $r$-tuples of $d\times d$
matrices that the closed ${\rm GL}(d)$-orbits in ${\mathcal M}_d$
are in 1-1 correspondence with the equivalence classes of {\it
semisimple} $d$-dimensional matrix representations of $U(\g,e)$.
Thus, the $\CC$-points of $R_d$ can be identified with the
isoclasses of semisimple $d$-dimensional $U(\g,e)$-modules.

It is well-known (and easily seen) that the set of all reducible
$d$-dimensional matrix representations of $U(\g,e)$ is Zariski
closed in ${\mathcal M}_d$. Since this set is also ${\rm
GL}(d)$-stable, the above-mentioned properties of $\pi_d$ show that
the subset ${\rm Irr}_d\subseteq R_d$ consisting of the isoclasses
of {\it irreducible} $d$-dimensional $U(\g,e)$-modules is Zariski
open in $R_d$. As we mentioned earlier, the component group
$\Gamma(e)$ acts on $R_d$ and preserves its open subset ${\rm
Irr}_d$. As $\Gamma(e)$ is a finite group, the quotient space
$R_d/\Gamma(e)$ is an affine variety (the coordinate algebra of
$R_d/\Gamma(e)$ is nothing but the invariant algebra
$\CC[R_d]^{\Gamma(e)}$). Furthermore, the quotient morphism
$\pi_{\Gamma(e)}\colon\,R_d\twoheadrightarrow R_d/\Gamma(e)$ is open
in the Zariski topology. Since ${\rm Irr}_d$ is open in $R_d$, the
set $\pi_{\Gamma(e)}({\rm Irr}_d)={\rm Irr}_d/\Gamma(e)$ is open in
$R_d/\Gamma(e)$.

Thus, each orbit space ${\rm Irr}_d/\Gamma(e)$ is a quasi-affine
variety. On the other hand, it follows from Theorem~\ref{main} and
[\cite{Lo1}, Thm.~1.2.2] that there is a bijection
$$\textstyle{\bigsqcup}_{d\ge 0}\,\big({\rm Irr}_d/\Gamma(e)\big)\,\stackrel{\sim}
{\longrightarrow}\,{\mathfrak X}_{\O(\chi)}.$$ Since this holds for
any coadjoint nilpotent orbit in $\g^*$, the primitive spectrum of
$U(\g)$ is a countable union of quasi-affine algebraic varieties.
\end{proof}

\end{document}